\documentclass[11pt]{article}
\usepackage[utf8]{inputenc}
\usepackage[T1]{fontenc}
\usepackage[ruled,vlined]{algorithm2e}
\usepackage{algorithmic}

\DeclareMathSizes{12}{30}{16}{12}

\usepackage{amsfonts}
\usepackage{amsmath}
\usepackage{framed}

\allowdisplaybreaks

\usepackage[english]{babel}
\usepackage{amsthm}
\usepackage{framed}

\usepackage{amsmath}
\usepackage{amsfonts}
\usepackage{amssymb}
\usepackage{mathtools}
\usepackage{commath}
\usepackage{graphicx} 
\usepackage[sc,osf]{mathpazo}

\newcommand{\R}{\mathbb{R}}
\newcommand{\N}{\mathbb{N}}
\newcommand{\B}{\mathcal{B}}
\newcommand{\I}{\mathcal{I}}
\newcommand{\E}{\mathbb{E}}
\newcommand{\V}{\mathbb{V}}
\newcommand{\D}{\mathcal{D}}

\newcommand{\uA}{u_B}
\newcommand{\uI}{u_{aux}}
\newcommand{\cI}{c_{aux}}
\newcommand{\cA}{c_B}
\newcommand{\hB}{h_{b}}
\newcommand{\hT}{h_{tail}}
\newcommand{\Tb}{T_{bulk}}
\newcommand{\psib}{\psi_{bulk}^*}
\newcommand{\psit}{\psi_{tail}^*}
\newcommand{\rhob}{{\rho_{bulk}^*}}
\newcommand{\rhot}{{\rho_{tail}^*}}
\newcommand{\TL}{\widetilde L_n}
\newcommand{\ThreshB}{{C_{\psi_{b}}}}
\newcommand{\RadB}{{C'_b}}
\newcommand{\ThreshT}{C''}
\newcommand{\CK}{C_K}
\newcommand{\CKprime}{{C_K^{(2)}}}
\newcommand{\cstar}{c_{\star}}

\newcommand{\comega}{c^{(\omega)}}
\newcommand{\Comega}{C^{(\omega)}}
\newcommand{\ch}{c^{(h)}}
\newcommand{\Ch}{C^{(h)}}
\newcommand{\cp}{c^{(p)}}
\newcommand{\Cp}{C^{(p)}}
\newcommand{\ctilde}{\widetilde c}

\newcommand{\Cmom}{\bar C}

\newcommand{\CPsiOne}{C_{\psi_1}}
\newcommand{\barhT}{\bar{h}_{tail}}

\newcommand{\cphi}{c^{(\phi)} }
\newcommand{\Cphi}{C^{(\phi)} }

\newcommand{\Ctilde}{\widetilde C}
\newcommand{\CbulkLB}{C_{bulk}^{LB}}
\newcommand{\CtailLB}{C_{tail}^{LB}}
\newcommand{\cu}{c_u}
\newcommand{\ctail}{c_{tail}}

\newcommand{\A}{\mathcal{A}}

\newcommand{\CvarNull}{C_{\V, H_0}}
\newcommand{\Ccov}{C^{(\text{c})}}
\newcommand{\CBT}{C_{BT}}
\newcommand{\T}{\mathcal{T}}

\newcommand{\AJtwo}{A_{J_2}}
\newcommand{\BJtwo}{B_{J_2}}
\newcommand{\CJtwo}{C_{J_2}}
\newcommand{\DJtwo}{D_{J_2}}
\newcommand{\gammadown}{\gamma^{(\downarrow)}}
\newcommand{\gammaup}{\gamma^{(\uparrow)}}
\newcommand{\Gammadown}{\Gamma^{(\downarrow)}}
\newcommand{\Gammaup}{\Gamma^{(\uparrow)}}

\newcommand{\cdown}{c^{(\downarrow)}}
\newcommand{\Cdown}{C^{(\downarrow)}}
\newcommand{\philj}{\phi_l^{(j)}}
\newcommand{\Elj}{E_l^{(j)}}
\newcommand{\hlj}{h_l^{(j)}}
\newcommand{\zlj}{z_l^{(j)}}
\newcommand{\cupward}{c^{(\uparrow)}}
\newcommand{\Xtilde}{\widetilde{\mathcal{X}}}
\newcommand{\An}{\A_1}
\newcommand{\Aonce}{\A}
\newcommand{\Aoncej}{\Aonce^{(j)}}
\newcommand{\Cint}{\|f\|_1}

\newcommand{\Poi}{\text{Poi}}
\newcommand{\ptens}{p_0^{\otimes \widetilde n}}
\newcommand{\qtens}{q_b^{(\widetilde n)}}
\newcommand{\pup}{p^{(\uparrow)}}
\newcommand{\pdown}{p^{(\downarrow)}}
\newcommand{\Azeroone}{A_0^{(1)}}
\newcommand{\Azerotwo}{A_0^{(2)}}
\newcommand{\Cq}{C_q}
\newcommand{\Cpiup}{C_{\pi}^{(\uparrow)}}
\newcommand{\Cintt}{\|f\|_{\alpha t}^{\alpha t}}
\newcommand{\Cgap}{C_{\text{gap}}}
\newcommand{\Asep}{\mathcal{A}_{\text{sep}}}
\newcommand{\cR}{c_r}
\newcommand{\nzero}{n_0}
\newcommand{\rhor}{{\rho_{r}^*}}

\newcommand{\cbeta}{c_\beta}
\newcommand{\calpha}{c_\alpha}
\newcommand{\Calpha}{C_\alpha}

\newcommand{\cch}{c_h}
\newcommand{\AL}[1]{A_{#1}}
\newcommand{\BL}[1]{B_{#1}}
\newcommand{\CL}[1]{C_{#1}}
\newcommand{\cL}[1]{c_{#1}}
\newcommand{\ssmall}[1]{can be made arbitrarily small by choosing #1 small enough}
\newcommand{\llarge}[1]{can be made arbitrarily large by choosing #1 large enough}
\newcommand{\slarge}[1]{can be made arbitrarily small by choosing #1 large enough}
\newcommand{\lsmall}[1]{can be made arbitrarily large by choosing #1 small enough}
\newcommand{\nbulk}{n_{bulk}}
\newcommand{\CvarAlt}{C_{\mathbb{V}, H_1}}

\newcommand{\psiout}{\psi_{\text{out}}}
\newcommand{\cout}{c_{\text{out}}}
\newcommand{\Cout}{C_{\text{out}}}
\newcommand{\Cd}{C_d}

\newcommand{\hm}{h_m}
\newcommand{\cm}{c_m}
\newcommand{\um}{u_m}
\newcommand{\Apsitwo}{A_{\psi_2}}
\newcommand{\Bpsitwo}{B_{\psi_2}}
\newcommand{\CDelta}{C_\Delta}
\newcommand{\TB}{{\widetilde{\B}}}
\newcommand{\tcA}{{\widetilde{\cA}}}
\newcommand{\tuA}{{\widetilde{\uA}}}
\newcommand{\Ctn}{C_{t_n}}
\newcommand{\Diag}{{\text{Diag}}}
\newcommand{\bI}{b_I}
\newcommand{\qI}{q_{aux}}
\newcommand{\cIM}{{\cI'}}
\newcommand{\cAM}{{\cA'}}
\newcommand{\rhobM}{{\rho_{bulk}^{*\;\mathcal{M}}}}
\newcommand{\rhotM}{{\rho_{tail}^{*\;\mathcal{M}}}}
\newcommand{\rhorM}{{\rho_{remain}^{*\;\mathcal{M}}}}
\newcommand{\barn}{\overline{n}}
\newcommand{\homega}{h_\Omega}
\newcommand{\csmall}{c_{small}}
\newcommand{\Ctaun}{C_{\tau_n}}
\newcommand{\Tdeg}{T}
\newcommand{\psidegen}{\psi_d^*}
\newcommand{\taun}{\tau_n}
\newcommand{\Clarge}{C_{large}}
\newcommand{\boundedcase}[1]{}
\newcommand{\unboundedcase}[1]{#1}
\boundedcase{\newcommand{\Ibar}{{\mathbf{I}}}}
\unboundedcase{\newcommand{\Ibar}{\N^*}}%
\unboundedcase{\newcommand{\Ibarzero}{\mathbf{I}_0}}%

\boundedcase{\newcommand{\Ibarzero}{{\Ibar_0}}}

\newtheorem{theorem}{Theorem}
\newtheorem{lemma}{Lemma}
\newtheorem{proposition}{Proposition}
\newtheorem{corollary}{Corollary}
\newtheorem{definition}{Definition}

\usepackage{lmodern}
\usepackage[T1]{fontenc}

\title{Local Goodness-of-Fit Testing for H\"older-Continuous Densities: Minimax Rates}

\usepackage{authblk}
\author[1]{Julien Chhor }
\author[2]{Alexandra Carpentier}
\affil[1]{CREST/ENSAE}
\affil[2]{OvGU, Magdeburg}

\date{}

\newcommand{\al}[1]{{\color{red} alex: #1}}

\begin{document}

\maketitle
\begin{center}
    Contact: jchhor@hsph.harvard.edu, carpentier@uni-potsdam.de
\end{center}

\hfill

\begin{abstract}
        \noindent We consider the goodness-of fit testing problem for H\"older smooth densities over $\R^d$: 
    given $n$ iid observations with unknown density $p$ and given a known density $p_0$, we investigate how large $\rho$ should be to distinguish, with high probability, the case $p=p_0$ from the composite alternative of all H\"older-smooth densities $p$ such that $\|p-p_0\|_t \geq \rho$ where $t \in [1,2]$. 
    The densities are assumed to be defined over \unboundedcase{$\R^d$ }\boundedcase{any arbitrary cubic domain $\Omega$ of $\R^d$ (including the unbounded domains $\R_+^d$ or $\R^d$ itself)} and to have H\"older smoothness parameter $\alpha>0$.
    In the present work, we solve the case $\alpha \leq 1$ and handle the case $\alpha>1$ using an additional technical restriction on the densities.
    We identify matching upper  and lower bounds on the local minimax rates of testing, given explicitly in terms of $p_0$. 
    We propose novel test statistics which we believe could be of independent interest. 
    We also establish the first definition of an explicit cutoff $\uA$ allowing us to split \boundedcase{$\Omega$}\unboundedcase{$\R^d$} into a bulk part (defined as the subset of \unboundedcase{$\R^d$}\boundedcase{$\Omega$} where $p_0$ takes only values greater than or equal to $\uA$) and a tail part (defined as the complementary of the bulk), each part involving fundamentally different contributions to the local minimax rates of testing. 
\end{abstract}

\hfill


\section{Introduction}

This paper studies the local Goodness-of-Fit testing problem for $\alpha$-Hölder densities over \unboundedcase{$\Omega = \R^d$.}\boundedcase{a cubic domain $\Omega \subseteq \R^d$.
By convention, $\R^d$ and $\R_+^d$ are considered as cubic domains of $\R^d$.}
For all $\alpha, L>0$, $H(\alpha,L)$ denotes the class of $\alpha$-Hölder densities over $\Omega$. We place ourselves on a subclass $\mathcal{P}(\alpha, L)$ of $H(\alpha, L)$.
The classes $\mathcal{P}(\alpha,L)$ and $H(\alpha,L)$ are defined in Section \ref{Problem_Statement}. 
We endow $\mathcal{P}(\alpha, L)$ with some distance denoted by $\text{dist}(\cdot,\cdot)$, which in our setting, can be any $L_t$ distance for $t \in [1,2]$: $\text{dist}(p,q) = \|p-q\|_t$. 
Given the iid observations $X_1, \dots, X_n$ with same unknown density $p \in \mathcal{P}(\alpha, L)$, and given a known density $p_0 \in \mathcal{P}(\alpha,L)$, we consider the non-parametric testing problem:
\begin{equation}\label{abstract_testing_pb}
    H_0: p=p_0 ~~~~ \text{ vs } ~~~~ H_1(\rho): p \in \mathcal{P}(\alpha,L) \text{ and } \text{dist}(p,p_0) \geq \rho.
\end{equation}
This problem is called the goodness-of-fit problem for continuous densities, which has been thoroughly studied in many works~\cite{barron1989uniformly, ingster1986minimax, ingster1984asymptotically, ingster2012nonparametric, ermakov1995minimax, fromont2006adaptive,  fan1998goodness}.\\

Following~\cite{ingster1984asymptotically, ingster1986minimax, ingster2012nonparametric}, we will focus on establishing, up to a multiplicative constant, the smallest possible separation distance $\rho^* = \rho^*(p_0, n, \text{dist})$ in a minimax sense such that a uniformly consistent test exists for Problem \eqref{abstract_testing_pb} - this condition will be specified in more details in Section \ref{Problem_Statement}.\\ 

Problem \eqref{abstract_testing_pb} has most often been studied for the uniform density $p_0$ over a bounded domain, e.g.~$[0,1]^d$~\cite{ingster1986minimax}, \cite{gine2016mathematical}. 
It has been extended to the case of densities $p_0$ constrained to be bracketed between two constants, still on a bounded domain~\cite{ingster1984asymptotically}, \cite{ermakov1995minimax}. See~\cite[Chapter 6.2]{gine2016mathematical} for a more recent overview. In the case where $p_0$ is the uniform density on $[0,1]^d$ and for $\alpha$-H\"older densities with $L=1$, and when the distance is defined as $d(p,q) = \|p-q\|_t$ where $\|.\|_t$ is the $L_t$ norm with $t\in [1,\infty]$, the minimax-optimal separation radius for Problem \eqref{abstract_testing_pb} is 
\begin{equation} \label{eq:ings}
    n^{-2\alpha/(4\alpha+d)}.
\end{equation}
See e.g.~\cite[Theorem 4.2]{ingster2012nonparametric} for the case where $d=1$ and in the related sequence space model over Besov balls. 
However, these results hinge on the assumption that $p_0$ is lower bounded by a positive constant. Hence, they cannot be extended to null densities on unbounded domains.\\ 

In fact, there is a fundamental gap between testing on bounded or unbounded domains. This was recently illustrated in the paper~\cite{balakrishnan2019hypothesis} which considers the case of Hölder continuous densities for $\alpha \in (0, 1]$ with separation in total variation distance ($L_1$ distance). The authors prove that there can be substantial heterogeneity when it comes to the minimax-optimal radius $\rho$, depending on $p_0$: testing some null hypotheses can be much easier than testing others. 
More precisely, they prove that uniformly over the class of $L$-Lipschitz densities, the minimax separation distance is bracketed as follows:
\begin{equation*}
    \left(\frac{L^{d/2} \left(\int_{p_0 \geq a(p_0)}p_0^{\frac{2}{3+d}}\right)^{\frac{3+d}{2}}}{n}\right)^\frac{2}{4+d} \lesssim \rho^*(p_0,n,\|.\|_1) \lesssim  \left(\frac{L^{d/2} \left(\int_{p_0 \geq b(p_0)}p_0^{\frac{2}{3+d}}\right)^{\frac{3+d}{2}}}{n}\right)^\frac{2}{4+d},
\end{equation*}
where $a(p_0)> b(p_0)>0$ are quantities - that are small and matching in order of magnitude for many cases, albeit not all - that depend only on $n,p_0$ and that are defined implicitly. 
See Section \ref{comparison_TV} for a thorough description of their results. 
The authors formally prove the interesting fact that the minimax separation distance depends on $p_0$ and they provide a test adapted to the shape of the density. 
For instance, if the density $p_0$ defined over $\mathbb R$ has essentially all its mass on e.g.~$[0,1]$, then the minimax optimal $\rho$ is $L^{1/5}n^{-2/5}$ - unsurprisingly comparable with in~\cite{ingster2012nonparametric}. However, if $p_0$ is heavy tailed, e.g.~corresponds to the Pareto distribution with parameter $\beta$, then the minimax optimal $\rho$ is $L^{1/5}n^{-2\beta/(2+3\beta)}$ - differing considerably from the rate of~\cite{ingster2012nonparametric}. This example highlights a specificity of testing heavy-tailed distributions, and by extension, distributions with unbounded support. 
To encompass all cases, it is therefore important to derive \textit{local} results where both the separation distance and the associated tests depend on $p_0$ in a refined way. 
The results in \cite{balakrishnan2019hypothesis} follow on ideas from a stream of literature concerning property testing.
For goodness-of-fit testing in the discrete (multinomial) setting, see~\cite{hoeffding1965asymptotically, fienberg1979use, batu2000testing, acharya2015optimal, chan2014optimal} for global results and ~\cite{diakonikolas2016new, valiant2017automatic, diakonikolas2017near, chhor2020sharp} for local results - see also~\cite{balakrishnan2018hypothesis} for an excellent survey. In the related setting of goodness-of-fit testing under local differential privacy, see \cite{berrett2020locally, dubois2021goodness}. 
Further papers considering estimation and hypothesis testing under privacy are~\cite{pmlr-v80-sheffet18a,pmlr-v89-acharya19b,9264231,9333586,9591593}.
Closest to our setting is~\cite{chhor2020sharp}, which studies the problem of goodness-of-fit for multinomials in $L_t$ norm for $t \in [1,2]$ - see Section~\ref{Discussion} for a thorough description of their results, and comparison.\\ 

In this paper, we focus on the problem of goodness-of-fit testing for H\"older smooth densities $p_0$, \boundedcase{possibly} defined on unbounded domains, extending over classical goodness-of-fit testing results following~\cite{ingster2012nonparametric}. We find how the minimax separation distance $\rho$ depends on $p_0$ and therefore provide local results. 
We consider a variety of separation distances, going beyond the $\|.\|_1$ distance from~\cite{balakrishnan2019hypothesis}: namely we consider all the $L_t$ distances for $t \in [1,2]$, as in \cite{chhor2020sharp} for the multinomial case. 
We cover all the scale of Hölder classes $H(\alpha, L)$ for all $\alpha > 0$, under technical assumptions for $\alpha >1$, extending from the case $\alpha \leq 1$ studied in \cite{balakrishnan2019hypothesis}. We identify the matching upper and lower bounds on $\rho(p_0, n, \text{dist})$ and provide the corresponding optimal tests in all the cases described above. In our results, the radius $\rho(n,p_0, \text{dist}_{L_t})$ is given explicitly as a function of $p_0$.\\

We now give a brief overview of the related literature on density testing, and explain more in details our contributions.
\begin{enumerate}
    \item \textbf{Testing for $\alpha$-H\"older distributions:} In the setting where $p_0$ is the uniform distribution over $[0,1]^d$ the global minimax rates are well understood, see~Equation~\eqref{eq:ings} and \cite{ingster2012nonparametric} for an adaptation of these results to the setting where $p_0$ is uniform over $[0,1]^d$. 
    In the local setting, however, little is known. 
    In the breakthrough results of \cite{balakrishnan2019hypothesis}, the case $\alpha \leq 1$ is almost completely solved. 
    However, the general case $\alpha>0$ is not considered and from the construction of the tests in \cite{balakrishnan2019hypothesis}, it is clear that the case $\alpha>1$ is far from being a trivial extension. Indeed, the test statistics in \cite{balakrishnan2019hypothesis} are built using heterogeneous histograms, which are not smooth enough when $\alpha >1$. 
    In the present paper, we solve the case of $\alpha$-H\"older densities for any $\alpha$ - but we need to introduce a technical condition for $\alpha >1$. 
    This assumption is akin to assuming that all the derivatives of the densities $p$ take value $0$ in their inflection points whose value is close to $0$. 
    We introduce novel test statistics, based on kernel estimators with heterogeneous bandwidth, simpler than the test statistic defined on an heterogeneous partition of the space from  \cite{balakrishnan2019hypothesis}. 
    We believe that this test statistic can also be of independent interest. See Section~\ref{Discussion} for a comparison with~\cite{balakrishnan2019hypothesis}.
    \item \textbf{Extension to $L_t$ distance:} The choice of distance influences the geometry of the alternative and consequently the nature of optimal tests as well as the expression for the minimax separation radius. 
    In Section \ref{Discussion}, we highlight that changing the norm can actually change the null densities $p_0$ which are the easiest or most difficult ones to test.
    The distances considered in the density testing literature are often either the $L_1$ distance - as is the case of local results in \cite{balakrishnan2019hypothesis} where the separation is only considered in $L_1$ distance, the $L_\infty$ distance \cite[Chapter 6.2]{gine2016mathematical}, or other $f$-divergences, such as the Kullback-Leibler, $\chi^2$ divergences or the Hellinger distance~\cite{daskalakis2018distribution}.
    In the discrete setting, the $L_1$ norm is often considered \cite{valiant2017automatic}, as well as the $L_2$ norm \cite{berrett2020locally}. The paper \cite{ghoshdastidar2020two} considers the two sample testing problem in inhomogeneous random graphs, in order to study the effect of various distances (total variation, Frobenius distance, operator norm, Kullback-Leibler divergence).
    The paper \cite{dan2020goodness} considers the goodness-of-fit testing problem in inhomogeneous random graphs for the Frobenius and operator norm distances.
    However, as will appear in the present paper, the phenomena occurring for testing in $L_t$ distances are similar for all $t \in [1,2]$. This property has already been identified in \cite{chhor2020sharp} in the discrete setting (multivariate Poisson families, inhomogeneous random graphs, multinomials). 
    The present paper extends the results from \cite{chhor2020sharp} to the continuous case, highlighting a deep connection between the two settings. However, as discussed in Section \ref{Discussion}, the results from \cite{chhor2020sharp} cannot be directly transferred to the density setting. 
    In our paper, we extend the results of \cite{balakrishnan2019hypothesis} to the case of more general norms. This impacts the choice of test statistics and new regimes appear, see Section~\ref{Discussion} for a comparison with~\cite{balakrishnan2019hypothesis} .
    \item \textbf{Matching upper  and lower bounds:} The local rates established by \cite{balakrishnan2019hypothesis} provide the first upper  and lower bounds for density testing in the local case. 
    Although matching in most usual cases, the authors discuss quite pathological cases for which their upper and lower bounds do not match. 
    Indeed, the method proposed in \cite{balakrishnan2019hypothesis} builds on the well-known multinomial identity testing analysis from \cite{valiant2017automatic}, which identifies upper  and lower bounds on the minimax separation radius for testing in total variation distance.
    However, even in the discrete setting, some specific cases can be found for which these upper  and lower bounds do not match, explaining the untightness of \cite{balakrishnan2019hypothesis} in some cases. In the present paper we bridge the gap, by proposing a new way to define a cut-off between bulk (set of large values of $p_0$) and tail (set of small values of $p_0$). This approach leads us to provably matching upper  and lower bounds on the minimax separation radius. 
    As opposed to \cite{balakrishnan2019hypothesis}, our result can moreover be expressed as an explicit function of $p_0$.
\end{enumerate}

The paper is organized as follows. In Section \ref{Problem_Statement}, we define the testing problem. 
In Section \ref{results}, we state our main theorem identifying the sharp minimax rate for the testing problem. 
We then analyse separately two different regimes, namely the bulk regime (in Section \ref{Bulk_regime_section}) and the tail regime (in Section \ref{Tail_regime_section}). 
We finally discuss our results in Section \ref{Discussion}.





\hfill

\section{Problem Statement}\label{Problem_Statement}

\subsection{Definition of the class of densities {$\mathcal{P}(\alpha, L)$}}
To ensure the existence of consistent tests for Problem \eqref{abstract_testing_pb}, structural assumptions need to be made on the class of densities we consider. 
Indeed, as shown in \cite{lecam1973convergence} and \cite{barron1989uniformly}, no consistent test can distinguish between an arbitrary $p_0$ and alternatives separated in $l_t$ norm if no further assumption is imposed on the set of alternatives. 
Throughout the paper, we place ourselves on a restricted subclass of the Hölder class of functions. 
Our class corresponds to the densities on \boundedcase{ the cubic domain $\Omega \subset \R^d$}\unboundedcase{ $\Omega = \R^d$ $(d \in \mathbb{N}^*)$}, with Hölder regularity and satisfying Assumption \eqref{simplifyingAssumption} defined below. 
\boundedcase{By convention, $\R^d$ is considered as a cubic domain of $\R^d$.}\\

Let $\alpha,L>0$ and denote by $\|\cdot\|$ the Euclidean norm over $\R^d$. We recall the definition of the Hölder class over $\Omega$. Set\footnote{Here $\lceil x \rceil$ is the smallest integer greater than or equal to a given real number $x$.} $m = \lceil \alpha \rceil - 1$ and consider a function $p:\Omega \longrightarrow \mathbb R$ that is $m$ times differentiable. Write $z\mapsto P_p(x, z)$ for the Taylor polynomial of degree $m$ of $p$ at $x$. The Hölder class is defined as:
\unboundedcase{
\begin{align*}
    H(\alpha, L) = \big\{p ~|~ \Omega \longrightarrow \mathbb R: ~ & p \text{ is }m \text{ times differentiable and } \\
    &\forall  x,y \in \Omega: |p(x)- P_p(x, y-x)| \le L \| x-y\|^\alpha\big\}.
\end{align*}}

\boundedcase{\begin{align*}
    H(\alpha, L) = \left\{\right.p : \Omega \longrightarrow \mathbb R: ~ & p \text{ is }m \text{ times differentiable and } |p(x)- P_p(x, y-x)| \le L \| x-y\|^\alpha\left.\right\}.
\end{align*}}

Our class of densities is obtained by intersecting $H(\alpha,L)$ with the set of densities $p: \Omega \longrightarrow \R_+$ satisfying:
\begin{equation}\label{simplifyingAssumption}
    \forall (x,y) \in \Omega : ~~ |p(x) - p(y)| \leq \; \cstar  p(x) \, + \,  L \|x-y\|^\alpha \tag{$\star$}, 
\end{equation}
for some fixed constant $\cstar \in (0,\frac{1}{2})$. Note that Assumption \eqref{simplifyingAssumption} is automatically satisfied for $\alpha \leq 1$. We discuss Assumption \eqref{simplifyingAssumption} in Section \ref{Discussion}. The class of densities therefore considered throughout the paper is defined as:
\boundedcase{
\begin{equation}\label{class_densities}
    \mathcal{P}(\alpha,L, \cstar) = \Big\{p \in H(\alpha,L)  \; \Big| \; \int_{\Omega} p = 1,\; p\geq 0 \text{ and $p$ satisfies \eqref{simplifyingAssumption} } \Big\}.
\end{equation} }
\unboundedcase{\begin{equation}\label{class_densities}
    \mathcal{P}_\Omega(\alpha,L, \cstar) = \Big\{p \in H(\alpha,L)  \; \Big| \; \int_{\Omega} p = 1,\; p\geq 0 \text{ and $p$ satisfies \eqref{simplifyingAssumption} } \Big\}.
\end{equation}}
When no ambiguity arises, we will drop the lower index $\Omega$ since it is assumed to be equal to $\R^d$. 

\subsection{Minimax testing framework}
Throughout the paper, we fix $t\in [1,2]$. For $f \in L_t(\Omega)$, we denote by $\|f\|_t$ the $L_t$ norm of $f$ with respect to the Lebesgue measure:
$$ \|f\|_t = \Big(\int_{\Omega} |f|^t \; dx\Big)^{1/t}.$$ 

\boundedcase{\al{
If $\Omega = \mathbb R^d$
$$ \|f\|_t = \Big(\int_{\Omega} |f|^t \; dx\Big)^{1/t},$$
and otherwise if $\Omega = [0,A]^d$
$$ \|f\|_t = \Big(\int_{\tilde \Omega} |f|^t \; dx\Big)^{1/t},$$
where $\tilde \Omega = [u, A-u]^d$. The latter definition is used to avoid boundary problems, see the discussion.
}}

Assume \textit{wlog} that the number of observations $n$ is even: $n=2k \; (k \in \N^*)$. 
We fix two constants $\alpha,L >0$. 
Assume moreover that we observe $n$ \textit{iid} random variables $X_1, \cdots, X_n$ with the same \textit{unknown} density $p \in \mathcal{P}(\alpha,L,\cstar)$. 
Let $p_0 $ be one particular \textit{known} density in $\mathcal{P}(\alpha,L,\cstar)$ and fix $\delta >0$. 
For some $\rho > 0$, the \textit{goodness-of-fit} testing problem is defined as:
\begin{align}\label{testing_pb}
\begin{split}
    H_0~~~ &: \; p = p_0 \;\;~~~~~~\; \text{ versus }\\
    H_1(\rho,t) &:\; p \in \mathcal{P}(\alpha,L', \cstar') \text{ s.t. } \|p-p_0\|_t \geq \rho,
\end{split}
\end{align}
where $L' = (1+\delta) L$ and $\cstar' = (1+\delta)  \cstar$. The parameter $\delta>0$ can be chosen arbitrarily small. 
Note that in Theorem 4.1 from~\cite{balakrishnan2019hypothesis}, a similar restriction, equivalent to enlarging the alternative, is also needed in the lower bounds. 
Namely, the authors test equality to $p_0$ against the set of $(\alpha,L)$-Hölder densities separated from $p_0$ in $L_1$ distance. 
For the lower bound, they need to assume that for some constant $c_{int} \in (0,1)$, the density $p_0$ has Hölder constant $c_{int}L$ rather than $L$. 
In comparison, we assume that $p_0$ has Hölder constant equal to $L$ and that, on the alternative, the densities have Hölder constant $L(1+\delta)$.
Clearly, these two formulations are equivalent. 
Similarly to~\cite{balakrishnan2019hypothesis}, this assumption is never needed in our upper bounds. 
However, enlarging the alternative is important for obtaining local minimax lower bounds since without this assumption, problems can arise when $p_0$ is on the boundary of the class $H(\alpha,L)$. 
For instance, consider $p_0 = \sqrt{L}(1-|x| /\sqrt L)_+$ for $\alpha = 1$. 
Here, we can say that $p_0$ is on the boundary of the class $H(1,L)$, as its Lipschitz constant is exactly equal to $L$. 
Therefore, perturbations of $p_0$ of the form $p_0\pm \phi$ for some function $\phi:\R^d\to \R$ can be out of the class (they might not be $L$-Lipschitz, even for arbitrarily small such $\phi$). 
Since the least favorable functions in the lower bounds are small perturbations of $p_0$, putting $\delta=0$ is problematic. 
\vspace{2mm}

Our goal is to establish how large $\rho$ should be for \eqref{testing_pb} to be feasible in a sense we now formally specify.\\ 

A \textit{test function} $\psi: \Omega^n \longrightarrow \{0,1\}$ is defined as a measurable function of the observations $(X_1, \dots, X_n)$ taking only values $0$ or $1$. The quality of any given test $\psi$ is measured by its \textit{risk}, defined as the sum of its type-I and type-II errors:
\begin{equation}\label{risk}
    R(\psi, \rho) \overset{def}{=} \mathbb{P}_{p_0}(\psi=1) + \sup_{p \in \mathcal{P}_1(\rho,t)} \mathbb{P}_p(\psi=0)
\end{equation}
where $\mathcal{P}_1(\rho,t) = \{p \in \mathcal{P}(\alpha,L', \cstar') : \|p-p_0\|_t \geq \rho\}$ is the set of all $p$ satisfying $H_1(\rho,t)$. We are looking for a test with smallest possible risk, if it exists. We therefore introduce the \textit{minimax risk} as:
\begin{align}
    R^*(\rho) = R^*(n,p_0,t, \rho) &= \inf_{\psi} R(\psi) = \inf_{\psi} \Big\{\mathbb{P}_{p_0}(\psi = 1) + \sup_{p \in \mathcal{P}_1(\rho,t)} \mathbb{P}_p(\psi = 0)\Big\}, \label{def_minimax_risk}
\end{align}
which corresponds to the risk of the best possible test. Here, $\inf\limits_{\psi}$ denotes the infimum over all tests. Note that if $R^*(\rho) = 1$, then random guessing is optimal. 
Hence, to have a non-trivial testing problem, it is necessary to guarantee $R^*(\rho) \leq \eta$ for some fixed constant $\eta \in (0,1)$. Noting that this bound on $R^*$ can only hold for $\rho$ large enough, we introduce the \textit{minimax separation radius}, also called \textit{minimax (testing) rate} or \textit{critical radius}, defined as the smallest $\rho>0$ ensuring $R^*(\rho) \leq \eta$.

\begin{definition}[Minimax separation radius]
We define the minimax separation radius, or minimax (testing) rate, as:
\begin{equation}\label{separation_radius}
    \rho^* := \inf \big\{\rho>0 : R^*(n,p_0,t,\rho) \leq \eta \big\}.
\end{equation}
\end{definition}

In the following, we fix $\eta \in (0,1)$. The aim of the paper is two-fold.
\begin{enumerate}
    \item Find the minimax rate $\rho^* = \rho^*(n,p_0,\alpha,L,t)$ defined in \eqref{separation_radius} and associated to problem \eqref{testing_pb}, up to multiplicative constants which are allowed to depend on $t, \eta, \alpha$ and $d$. 
    \item Find a test $\psi^*$ and a constant $C>0$ such that $R(\psi^*, C \rho^*) \leq \eta$. This ensures that if the hypotheses are separated by $C\rho^*$, then Problem \eqref{testing_pb} is guaranteed to have a decision procedure with risk at most $\eta$, namely $\psi^*$. 
\end{enumerate}

\subsection{Notation}

In what follows, we will define $x \lor y = \max(x,y)$, $x \land y = \min(x,y)$ and $x_+ = x\lor 0$. 
We will use $\|\cdot\|$ to denote the Euclidean norm over $\R^d$.
The \textbf{support} of a function $p: \Omega \longrightarrow \R$ is defined as $\{x \in \Omega : p(x) \neq 0\}$.
We write $\lceil x \rceil$ the smallest integer greater than or equal to a given real number $x$. For any set $A \subset \Omega$ and any function \boundedcase{$f \in L_t(\Omega)$}\unboundedcase{$f \in L_t$}, we also define $\|f_A\|_t = \big(\int_{A} |f|^t \big)^{1/t}$. \textbf{Throughout the paper, we will call "constant" any strictly positive constant depending only on $\eta, \alpha, t$ and $d$.} 
For any two nonnegative functions $f,g$, we will write $f \lesssim g$ if there exists a constant $C > 0$ such that $f \leq Cg$, where $C = C(\eta, d, t, \alpha)$. 
We will also write $f \gtrsim g$ if $g \lesssim f$ and $f \asymp g$ if $f \lesssim g$ and $f \gtrsim g$. 
For any two real numbers $a,b$ we will write $[a \pm b] = [a-b,a+b]$. 
For any set $A \subset \Omega$, we will denote by $A^c$ the complement of $A$ in $\Omega$: $A^c = \Omega \setminus A$. 
Denote by $\mathfrak{B}(\Omega)$ the Borel $\sigma$-algebra over $\Omega$. For any two probability distributions $P,Q$ over $(\Omega, \mathfrak{B}(\Omega))$, we will denote by $d_{TV}(P,Q) = \sup\limits_{A \in \mathfrak{B}(\Omega)} \left|P(A)-Q(A)\right|$ the total variation distance between $P$ and $Q$. If $P$ and $Q$ are absolutely continuous with respect to some measure $\mu$ over $(\Omega, \mathfrak{B}(\Omega))$ with densities $p$ and $q$ respectively, we will also write 
$$ d_{TV}(p,q) = \frac{1}{2}\int_\Omega |p-q|\, d\mu = d_{TV}(P,Q).$$

We also denote by $\text{Unif}(A)$ the uniform distribution over any bounded Borel set $A \subset \R^d$.


\section{Results}\label{results}

We fix $p_0 \in \mathcal{P}(\alpha, L)$ defined in \eqref{class_densities}, along with some constant $\eta \in (0,1)$. We first give an overview of our results. The domain $\Omega$ will be split into two parts, namely the \textit{bulk} part, where $p_0$ takes only large values, and the \textit{tail} part, where $p_0$ takes only small values. The explicit definitions of $\B(u)$ and $\mathcal T(u)$ are given below.\\

We will analyze separately the bulk and the tail regimes. In each case, we will restrict $p_0$ to each particular set, and separately establish the minimax separation radii $\rhob$ and $\rhot$. Likewise, we will identify the optimal tests $\psib$ and $\psit$ independently on each set. The overall minimax separation radius $\eqref{separation_radius}$ will be given - up to multiplicative constants - by the sum of the two terms: $\rho^* \asymp \rhob + \rhot$, and the overall optimal test by the combination of the two tests: $\psi^* = \psib \lor \psit$.

\subsection{Partitioning the domain $\Omega$}
\subsubsection{Splitting the domain into bulk and tail}
It has been well known since \cite{valiant2017automatic} that, in the multinomial setting, the local goodness-of-fit problem involves splitting the null distribution into bulk and tail. In our analysis, we also divide $\Omega$ into a bulk $\B = \{x \in \Omega ~|~ p_0(x) \geq \uA\}$ and a tail $\T = \B^c = \{x \in \Omega ~|~ p_0(x) < \uA\}$ for some value $\uA$ specified later. On the other hand, like in \cite{balakrishnan2019hypothesis}, a further key idea is to divide $\Omega$ into smaller cubes with possibly varying edge lengths. Each cube will be considered as a single coordinate of a multinomial distribution, allowing us to (approximately) represent our continuous density as a discrete multinomial distribution.\\

The fundamental idea of our tail definition is to ensure the following condition. Assume the tail has been split into cubes with suitable edge length $\hT$ (specified below). If $H_0$ holds, then with high probability, none of the tail cubes will contain $2$ observations or more. The cut-off $\uA$ is designed to ensure this condition. Before giving its expression, we first introduce:
\begin{equation}\label{def_B}
    \forall u \geq 0: ~~ \B(u) := \left\{x\in \Omega: p_0(x) \geq u \right\} ~~ \text{ and } ~~ \T(u) := \left\{x\in \Omega: p_0(x) < u \right\} = \B(u)^c.
\end{equation}

\hfill

For any Borel set $A \subset \Omega$ and any measurable nonnegative function $f : \Omega \longrightarrow \R_+$, we write $f[A] = f(A) = \int_A f$. 
For any $\gamma>0$, the notation $f^\gamma[A] $ will always denote the quantity $\int_A f^\gamma$ and the notation $f[A]^\gamma$ will always denote the quantity $\left(\int_A f\right)^\gamma$.
We now introduce an auxiliary value $\uI$, used to define the cut-off $\uA$:
\begin{equation}\label{def_uI}
    \uI := \sup \Bigg\{u\geq 0: \frac{(p_0^2)\left[\T(u)\right]}{\big(p_0\left[\T(u)\right]\big)^{d/(\alpha+d)}} \leq \cI \TL^{1/(\alpha+d)}\Bigg\}, ~~ \text{ where $\TL = \frac{L^d}{n^{2\alpha}}$} 
\end{equation}
and $\cI$ is a small enough constant. We will also refer to the following notation throughout the paper:
\begin{equation}\label{def_I}
    \I := \int_{\B(\uI)} p_0^{r} =: (p_0^r)\big[\B(\uI)\big], ~~ \text{ where } r = \frac{2\alpha t}{(4-t)\alpha +d}.
\end{equation}
We now introduce the value $\uA$ defining our cut-off as 
\begin{equation}\label{def_uA}
        \uA = \uI \lor \left[\cA \frac{ L^\frac{d}{4\alpha + d}}{ \left(n^2 \I \right)^\frac{\alpha}{4\alpha + d}}\right]^\frac{(4-t)\alpha +d}{(2-t)\alpha +d} \hspace{-1cm},
\end{equation}
where $\cA$ is a small enough constant. The constants $\cI$ and $\cA$ can be chosen arbitrarily small, as long as they only depend on $\eta, \alpha, t$ and $d$. \boundedcase{Moreover, it is possible to impose the relation $\cI^{\alpha+d} = \cA^{4\alpha+d}$. }
The value $\uA$ is the value defining our cut-off between bulk and tail. In the sequel we will write 
\begin{equation}
    \B = \B(\uA) ~~ \text{ and } ~~ \T = \T(\uA).
\end{equation}

\boundedcase{
We will denote by $\homega$ the edge length of $\Omega$ if $\Omega$ is bounded, and $\homega = +\infty$ otherwise. We start with the following Proposition handling degenerate cases which we wish to exclude in the sequel. 

\begin{proposition}\label{Lnhomega_leq_1_proposition}
Assume $\Omega$ is bounded and write \textit{wlog} $\Omega = [\pm \homega/2]^d$. Define a sufficiently small constant $\csmall>0$ and a sufficiently large constant $\Clarge>0$.
\begin{enumerate}
    \item If $\max\limits_\B \hB > \Clarge\homega$, then $ \rho^*(p_0,n,\alpha,L) \asymp \left({\sqrt{n}\homega^{d-d/t}}\right)^{-1}.$
\end{enumerate}
Assume now $nL\homega^{\alpha+d} \leq \csmall$. 
\begin{enumerate}
    \item[2.] If $\alpha \leq 1$, then the minimax separation radius is trivial: $\rho^* = \text{diam}\left(\mathcal{P}(\alpha,L,\cstar)\right) \asymp \left(n\homega^{d-d/t}\right)^{-1}$.
    \item[3.] If $\alpha>1$ then $\rho^*(p_0,n,\alpha,L)  \asymp \left({\sqrt{n}\homega^{d-d/t}}\right)^{-1}$.
\end{enumerate}
\end{proposition}

Proposition \ref{Lnhomega_leq_1_proposition} is proved in Appendix \ref{hb>homega}, \ref{degen_alpha_leq_1} and \ref{degen_alpha>1}. \textbf{From now on, we will always assume $nL\homega^{\alpha+d} > \csmall$ and $\max\limits_\B \hB \leq \Clarge\homega$.} }We now state our main theorem:
\begin{theorem}\label{main_th} 
\boundedcase{Suppose $nL\homega^{\alpha+d} > \csmall$ and $\max\limits_\B \hB \leq \Clarge\homega$. }Set $\TL = L^d/n^{2\alpha}$ and $r = \frac{2\alpha t}{(4-t)\alpha+d}$. There exists a constant $\barn = \barn(d,\eta,t,\alpha)$ independent of $p_0$ such that, for all $n \geq \barn$, the minimax separation radius associated to problem \eqref{testing_pb} is given by
\vspace{2mm}
\large{
\begin{equation}\label{expression_rho_star}
    \rho^* \asymp \rhob + \rhot + \rhor,
\end{equation}} 
where 

\vspace{-6mm}

\begin{align}
    \rhob = \bigg[\TL~ \big\|\, p_{0,\,\B(\uI)}\, \big\|_{\,r}^{2\alpha}\bigg]^\frac{1}{4\alpha+d},&\hspace{1cm} \rhot = \Big[\, \TL^{t-1}\; p_0[\, \T \, ] ^{(2-t)\alpha+d}\, \Big]^{\frac{1}{t(\alpha+d)}} \label{def_rho_bulk_and_tail}\\
    \text{ and } ~ &\rhor = \left[\, \frac{L^{d(t-1)}}{n^{\alpha t + d}}\,\right]^\frac{1}{t(\alpha + d)}.\label{eq:def_rhor}
\end{align}
\end{theorem}
In the above Theorem, $p_{0,\B(\uI)} = p_0\, \mathbb{1}\{\B(\uI)\}$. Note that $\rhob$ depends on $n$ as $n^{-2\alpha/(4\alpha+d)}$. Note moreover that $\rhot + \rhor \asymp \left[\TL^{\frac{t-1}{\alpha+d}} \left(\frac{1}{n} + \int_{\T}p_0 \right)^\frac{(2-t)\alpha+d}{\alpha+d}\right]^{1/t}\hspace{-3mm}.$ The quantity $\rhor$ is a \underline{r}emainder term which is analogous to $\frac{1}{n}$ in discrete testing (see e.g. \cite{chhor2020sharp}).
The optimal test achieving this rate is given by $\psi^* = \psib \lor \psi_1 \lor \psi_2$ where $\psib$ is defined in \eqref{def_test_bulk}, $\psi_1$ in \eqref{test_psi1} and $\psi_2$ in \eqref{test_psi2}. We now successively study the bulk and the tail regimes individually.\\

\section{Bulk regime}\label{Bulk_regime_section}
In this section, we place ourselves on the bulk and analyze separately the upper bound and the lower bound on the minimax separation radius. 

%

\subsection{Bulk upper bound}


In this subsection, we construct a test statistic $\psib$ over the bulk. 
For each $x \in \B(\frac{\uA}{2})$, introduce the following bandwidth value depending on $p_0(x)$:
    \begin{equation}\label{def_hbulk}
        \hB(x) = \frac{p_0(x)^\frac{2}{(4-t)\alpha + d}}{\Big(n^2 L^4 \, \I \Big)^\frac{1}{4\alpha + d}}.
    \end{equation}

This bandwidth allows us to give some interpretation for the cut-off $\uA$ defined above. 
On the bulk, we have $p_0(x) \geq \uA$ by definition. 
Moreover, for the upper and lower bounds, it turns out to be important for the following relation to hold on the bulk: $\forall x \in \B(\uA): p_0(x) \geq \cA \, L \, \hB(x)^\alpha$, and $p_0(x) \geq \uI$. 
This suggest introducing our cut-off as the smallest value $u \geq \uI$ such that $\forall x \in \R^d, \left[p_0(x)\geq u \implies p_0(x) \geq \cA \, L \, \hB(x)^\alpha\right]$. 
This exactly corresponds to the definition of $\uA$ from~\eqref{def_uA}.

Throughout the paper, we will say that \textit{the bulk dominates (over the tail)} whenever
\begin{equation}\label{def_bulk_dominates}
    \CBT\rhob \geq \rhot,
\end{equation}
for some sufficiently large constant $\CBT$. In the converse case, (when $\CBT\rhob < \rhot$), we will say that the \textit{tail dominates}. We recall that "\textit{constant}" denotes any positive real number allowed to depend only on $\eta, t, d$ and $\alpha$. To define our bulk test, we distinguish between two cases. We set
\begin{equation}\label{def_tuA}
    \tuA = \begin{cases}\frac{\uA}{2} & \text{ if the bulk dominates }\\ \uA & \text{ if the tail dominates,}\end{cases} ~~~~ \text{ and } ~~~~ \TB = \B(\tuA).
\end{equation}

Note that over $\TB$, it always holds $p_0 \geq \tcA L \hB^\alpha$ where $\tcA = 2^{\frac{2\alpha}{(4-t)\alpha+d} -1}\cA$.  The bulk upper bound will be analyzed over $\TB$ rather than $\B$.\\


Define a kernel $K$ over $\R^d$ and introduce the usual notation $K_h(x) = \frac{1}{h^d}K(\frac{x}{h})$ for any $h>0$ and $x \in \R^d$. We choose $K$ such that
\begin{itemize}
    \item $K$ is of order $\alpha$, i.e.~for any $f \in \mathcal H(\alpha, L)$ and $h>0$: $\|f - K_h(x- ~ \cdot ~)*f \|_\infty \leq \CK L h^\alpha.$ 
    \item $K$ is bounded in absolute value by a constant depending on $\alpha$ and $d$.
    \item $K$ is $0$ over $\{x\in\R^d : \|x\|_2 > \frac{1}{2}\}$.
\end{itemize}

In the above definition, we set, for $m= \lceil \alpha \rceil - 1$:
\begin{equation}\label{def_CK}
    \CK = \frac{1}{m!} \int_{\R^d} \|u\|_2^m \; |K(u)| \; du.
\end{equation}

We first split the data $(X_1, \dots, X_{2k})$ in two equal-sized parts $(X_1, \dots, X_k)$ and $(X_{k+1}, \dots, X_n)$. We set $h(x)=\cch\hB(x)$ where $\cch = (\tcA/4)^\frac{1}{\alpha}$ and build for each batch an estimator of the true underlying distribution $p$ over $\TB$:
\begin{equation}\label{estimates_of_p}
    \hat p(x) = \frac{1}{k} \sum_{i=1}^k K_{h(x)}(x-X_i),~~~~~~\hat p'(x) = \frac{1}{k} \sum_{i=k+1}^{2k} K_{h(x)}(x-X_i).
\end{equation}
For all $x \in \TB$, $\hat p(x)$ and $\hat{p}'(x)$ are independent random variables. Note moreover the variable bandwidth $h(x)$ depends on $x$. We propose the following test statistic:

\vspace{5mm}

\begin{equation}\label{bulk_test_statistic}
    \Tb = \int_{\TB}   \omega(x) \left[\hat p(x) - p_0(x) \right]\left[\hat p'(x)   - p_0(x) \right]dx,
\end{equation}

\begin{align}
    \text{where  } ~~ \omega (x) &= p_0(x)^\frac{2\alpha t - 4 \alpha}{(4-t) \alpha + d}.\label{def_omega}
\end{align}

\vspace{3mm}

We can now define the optimal test on the bulk:
\begin{equation}\label{def_test_bulk}
    \psib = \mathbb{1} \big\{ \Tb > \ThreshB t_n\big\} ~~ \text{ where } t_n = \Ctn \left(L^{2d}\I^{2\alpha+d}n^{-{4\alpha}}\right)^{1/(4\alpha + d)},
\end{equation}
where $\ThreshB$ and $\Ctn>1$ are sufficiently large constants.\\

We use sample splitting in~\eqref{bulk_test_statistic} to simplify the analysis of the variance of $\Tb$. The re-weighting $\omega(x)$ is a re-normalizing factor whose role is to balance the expectation and 
variance of $\Tb$. 
Note that $\omega(x)$ increases as $p_0(x)$ decreases. Therefore, a large dispersion observed at some $x \in \TB$ for which $p_0(x)$ is small will be amplified by $\omega(x)$ and contribute to a larger increase in $\Tb$. 
Up to a constant, the threshold $t_n$ corresponds to the standard deviation of $\Tb$ under $H_0$.\\


The following proposition yields guarantees on $\rhob$ in the bulk regime. Recall that $L'=(1+\delta)L$ and $\cstar' = (1+\delta)\cstar$ where $\delta \in (0,1)$ is a constant.

\begin{proposition}\label{upper_bound_bulk}
For any $\rho>0$, define $\mathcal{P}_{Bulk}(\rho) = \big\{p \in \mathcal{P}(\alpha, L', \cstar') ~\big| ~ \displaystyle\int_\TB |p-p_0|^t \geq \rho^t\big\}$. There exists a constant $\RadB = \RadB(\eta, \alpha, d, t)>0$, such that: 
$$ \mathbb{P}_{p_0}(\psib=1) + \sup_{p \in \mathcal{P}_{Bulk}(\RadB\rhob)} \mathbb{P}_p(\psib=0) \leq \frac{\eta}{2}.$$
\end{proposition}

Note that Proposition \ref{upper_bound_bulk} does not directly yield an upper bound on the minimax separation radius $\rho^*$. 
Indeed, it only ensures that, with high probability, the test 
$\psib$ can detect any probability distribution $p \in \mathcal{P}(\alpha, L', \cstar')$ that is separated from $p_0$ by a sufficiently large $L^t$ discrepancy over $\TB$, but does not say anything about detecting perturbations over the complement of $\TB$. 
In order to obtain the full upper bound, one has to combine Proposition~\ref{upper_bound_bulk} with the tail upper bound stated in Proposition~\ref{upper_bound_tail}, Section~\ref{Tail_regime_section}. 

\subsection{Bulk lower bound}\label{bulk_lower_bound_subsection}


\textbf{Throughout Subsection \ref{bulk_lower_bound_subsection}, we assume that the bulk dominates.} In the following sections, we will justify why, here, we can make this assumption without loss of generality.\\

We recall that we aim at proving $\rho^* \gtrsim \rhob + \rhot + \rhor$. 
In this Subsection, we explain how to obtain the lower bound $\rho^* \gtrsim \rhob$ when the bulk dominates.
To do so, we use Le Cam's two-point method with a prior distribution referred to as ``\textit{bulk prior}'', defined as a mixture of densities in $\mathcal{P}(\alpha,L', \cstar')$. 
We further ensure that this prior is "indistinguishable" from $p_0$ with risk at most $\eta$, meaning that the total variation between $n$ data from $p_0$ and $n$ data from the mixture is constrained to be $\leq 1-\eta$. \\

We now explain the high-level construction of the bulk prior.
The idea is to split the bulk domain $\B$ into small cubes of various edge lengths using Algorithm~\ref{Algo_partitionnement}. 
More precisely, Algorithm~\ref{Algo_partitionnement} yields a \textit{covering} of the domain $\B$ by disjoint cells $(B_1,\dots,B_N)$ satisfying $\B \subset \cup_{j=1}^N B_j$, for some $N\in \N$.  
Now, our prior is defined by perturbing $p_0$ over each cell $B_j$ independently at random. 
More precisely, for each $j \in \{1,\dots,N\}$, we define a deterministic function $\phi_j$ supported on $B_j$ and satisfying 
\begin{align*}
    \int_{B_j} \phi_j = 0, \quad p_0\pm \phi_j \geq 0, \quad \text{ and } \quad \phi_j \in H(\alpha, \delta L).
\end{align*} 
The first two conditions ensure that, for any $j\in \{1,\dots, N\}$, $p_0 \pm \phi_j$ is a density. 
Our prior is simply defined as follows: Let $\epsilon_1,\dots, \epsilon_N \overset{iid}{\sim} \operatorname{Rad}(1/2)$ and let $\epsilon = \left(\epsilon_1,\dots, \epsilon_N\right)$. 
Conditional on $\epsilon$, we define the function
\begin{align*}
    p_\epsilon^{(n)} := p_0 + \sum_{j=1}^N \epsilon_j \phi_j.
\end{align*}

This prior satisfies the following useful properties. For any fixed $\epsilon \in \{\pm1\}^N$, The function $p_\epsilon^{(n)}$
\begin{itemize}
    \item is a density by construction,
    \item belongs to $ \mathcal{P}(\alpha,L',\cstar')$ (see Proposition~\ref{prior_bulk_admissible}),
    \item  is separated from $p_0$ by $\|p_\epsilon^{(n)}-p_0\|_t = \|\sum_j \phi_j\|_t \gtrsim \rhob$ if the bulk dominates (see equation~\eqref{eq:separation_LB_bulk_phij}).
\end{itemize}

Moreover, unconditionally on $\epsilon$, we also have that $d_{TV}\big(p_\epsilon^{(n)}, p_0^{\otimes n}\big) \leq 1-\eta$ (see equation~\eqref{eq:TV_prior_tail}). 
Therefore, the lower bound $\rho^* \gtrsim \rhob$ follows.
\\

The bulk prior is defined in full details in Appendix \ref{Lower_bound_bulk_appendix}, and the following Proposition, also proved in Appendix~\ref{Lower_bound_bulk_appendix}, yields the bulk lower bound. 

\begin{proposition}\label{Lower_bound_bulk_proposition}
In the case where the bulk dominates, i.e. when \eqref{def_bulk_dominates} holds, there exists a constant $\CbulkLB$ such that $\rho^* \geq \CbulkLB\, \rhob$.
\end{proposition}

\section{Tail regime}\label{Tail_regime_section}

In this section, we place ourselves on the tail 
and analyze separately the upper bound and the lower bound. 
We recall that the tail is defined such that, with high probability under $H_0$, when split into cubes with suitable edge length $\hT$, no cube contains more than one observation. Recalling that $\T = \T(\uA)$, define:
\begin{equation}\label{def_htail}
    \hT := \Big(n^2 L\; p_0[\T]\Big)^{-\frac{1}{\alpha + d}}.
\end{equation}

For both the upper and the lower bound, we define a binning of the tail domain. 
This is done using the following algorithm. As inputs, the algorithm takes a value $u>0$ and a length $\tilde h>0$. It (implicitly) defines a grid of cubes $C_{j_1,\dots, j_d} := [j_1\, \tilde h, (j_1+1) \tilde h]\times \dots \times [j_d \tilde h, (j_d+1)\tilde h]$ for all $(j_1,\dots,j_d) \in \mathbb{Z}^d$\boundedcase{ with $\tilde h \asymp h$}. It returns the indices of all such cubes whose intersection with $\T(u)$ is empty (hence indices of cubes to be removed from the tail covering). \boundedcase{ We recall that $\Omega$ is a cubic domain of $\R^d$.If $\Omega$ is finite, we denote by $h_\Omega$ the edge length of $\Omega$.}\\

\begin{algorithm}[H]\label{splitting_tail}
\SetAlgoLined
\begin{enumerate}
    \item \textbf{Input:} $u, \boundedcase{h}\unboundedcase{\tilde h}$. 
    \boundedcase{\item \textbf{If} $\Omega \neq \R^d$: \textbf{then} $\tilde h := h_\Omega \lceil h_\Omega/h\rceil^{-1}$ \textbf{else} $\tilde h = h$.}
    \item Set $\lambda \in \N$ such that $\B(u) \subset [-\lambda \tilde h, \lambda \tilde h]^d$. Set $P = \emptyset$. 
    \item For $(j_1, \dots, j_d) \in ([-\lambda, \lambda]\cap \mathbb{Z})^d$:\\
    \hspace{.5cm} \textbf{if} $\T(u) \cap \big[j_1\, \tilde h, (j_1+1)\tilde h\big]\times \dots \times \big[j_d \tilde h, (j_d+1)\tilde h\big] = \emptyset$: \textbf{then} $P \longleftarrow P\cup \{(j_1,\dots, j_d)\}$.
    \item \textbf{Return} $P$.
\end{enumerate}
\caption{Tail splitting}
\end{algorithm}

\hfill\break

The tail splitting is defined as follows. 
We denote by $P$ the output of Algorithm \ref{splitting_tail} \unboundedcase{and set $I = \mathbb{Z}^d\setminus P$.}\boundedcase{ and consider the \boundedcase{(possibly infinite)}\unboundedcase{(infinite)} family of cubes $\left(C_{j}\right)_{j \in \mathbb{Z}^d \setminus P}$. 
We let $I$ be the set of all indices $j \in \mathbb{Z}^d \setminus P$ such that $C_j \cap \Omega \neq \emptyset$ and we consider $(C_j)_{j\in I}$. }
\boundedcase{This family is infinite if $\Omega = \R^d$ and finite otherwise. }
Since the bulk is a bounded subset of $\Omega$, $P$ is finite so that $I$ is infinite.
Moreover, since the sum $\sum\limits_{j \in I} \int_{\widetilde C_j} p_0$ is finite ($\leq 1$), 
it is possible to sort the cubes $\left(C_{j}\right)_{j \in I}$ as $(\widetilde C_l)_{l \in \Ibar}$, while ensuring that $\left(\int_{\widetilde C_l}p_0\right)_{l \in \Ibar}$ is sorted in decreasing order. 
\boundedcase{We take $\Ibar = \N^*$ if $\Omega = \R^d$ and $\Ibar = \{1,\dots,M\}$ for some $M \in \N$ otherwise. We call $(\widetilde C_j)_{j \in \Ibar}$ the \textit{covering defined by Algorithm \ref{splitting_tail} with inputs $u$ and $h$}. }
Note that $\T(u) \subset \bigcup\limits_{j\in \Ibar} \widetilde C_j$, but that the reverse inclusion does not necessarily hold. 
Moreover, for $j \neq l$, $\widetilde C_j \cap \widetilde C_l$ has Lebesgue measure $0$. Therefore, almost surely, any observation $X_i$ belongs to at most one of the cubes $(\widetilde C_j)_j$. 

\subsection{Tail upper bound}

To define our tail test, we distinguish between two cases. In the sequel, $\CBT$ denotes a large constant. 
\begin{itemize}
    \item \underline{If the tail dominates}, i.e. if $\CBT\rhob \leq \rhot$, then we set $(\widetilde C_j)_{j \in \Ibar}$ to be the covering defined by Algorithm \ref{splitting_tail} with inputs $\tuA=\uA$ and $\tilde h = \hT$. \boundedcase{In the analysis of the tail upper bound (Appendix \ref{UB_tail_regime}), we show that in this case, necessarily $h < \homega$.} 
    \item \underline{If the bulk dominates}, i.e. if $\CBT\rhob \geq \rhot$, then we set $(\widetilde C_j)_{j \in \Ibar}$ to be the covering defined by Algorithm \ref{splitting_tail} with inputs $\tuA=\frac{\uA}{2}$ and $\tilde h = \hm$ where
    \begin{equation}\label{def_hm}
        \hm =  \frac{\cm}{(n^2 L^4 \I)^\frac{1}{4\alpha+d}}\left[\cA \frac{ L^\frac{d}{4\alpha + d}}{ \left(n^2 \I \right)^\frac{\alpha}{4\alpha + d}}\right]^\frac{2}{(2-t)\alpha +d},
    \end{equation}
    and $\cm = \left((\frac{1}{2} - \frac{\cstar}{2}) \frac{\cA}{\sqrt{d}^\alpha}\right)^\frac{1}{\alpha}$. To understand why $\hm$ is a natural bandwidth to introduce, set $\um = \left[\cA \frac{ L^d}{ \left(n^2 \I \right)^\alpha}\right]^{\frac{1}{4\alpha + d}\frac{(4-t)\alpha +d}{(2-t)\alpha +d}}$ and note that $\uA = \uI \lor \um$. Observe moreover that $\um$ is the unique value ensuring that, if for $x\in\Omega, \; p_0(x) = \um$, then $\cm \hB(x) = \hm$.
\end{itemize}



\hspace{3cm}

The tail test $\psit$ is defined as a combination of two tests:
\vspace{-2mm}
\begin{itemize}
    \item The first test $\psi_1$ counts the \textit{total number of observations} on the tail, i.e.~on the union of the sets $\widetilde C_j$, and rejects $H_0$ when this total mass is substantially different from its expectation under $H_0$.
    \vspace{-1mm}
    \item The second test $\psi_2$ rejects $H_0$ whenever there exists one cell $\widetilde C_j$  containing two observations or more.
\end{itemize} 
For each cube $\widetilde C_j$, the integer $N_j$ is defined as the total number of observations on $\widetilde C_j$:
\begin{equation}\label{Nj}
    N_j = \sum_{i=1}^n \mathbb{1}\{X_i \in \widetilde C_j\}.
\end{equation}
We call $(N_j)_j$ the \textit{histogram} of $(X_i)_i$ on the tail. Note that the family $(\widetilde C_j)_j$ is infinite but that all of the observations are only contained in a finite number of sets $\widetilde C_j$. 
Thus, the number of values $N_j$ that are nonzero is finite. Recalling that $\T = \T(\uA)$, our tail test $\psit$ is defined as $\psit = \psi_1 \lor \psi_2$ where:
\begin{align}
    \psi_1 &= \mathbb{1}\Big\{\Big|\,\frac{1}{n}\sum_{j \in \Ibar} N_j - p_0[\T]\, \Big| > \CPsiOne \sqrt{\frac{p_0[\T]}{n}}\Big\}\label{test_psi1},\\
    &\nonumber\\
    \psi_2 & = \begin{cases} 1 \text{ if } N_j \geq 2 \text{ for some  } j, \\
    0 \text{ otherwise}.\end{cases} \label{test_psi2}
\end{align} Here, $\CPsiOne$ is a sufficiently large constant. The following proposition yields guarantees on $\rhot$ in the tail regime.

\begin{proposition}\label{upper_bound_tail}
For all $\rho>0$, define $\mathcal{P}_{Tail}(\rho) = \big\{p \in \mathcal{P}(\alpha, L', \cstar') ~\big| ~ \int_{\T\left(\tuA\right)} |p-p_0|^t \geq \rho^t\big\}$. There exists a constant $\ThreshT = \ThreshT(\eta, \alpha, d, t) >0$, such that
$$ \mathbb{P}_{p_0}(\psit=1) ~~ + \sup_{p \,\in\, \mathcal{P}_{Tail}\left(\ThreshT\rho^*\right)} \mathbb{P}_p(\psit=0) ~\leq ~ \frac{\eta}{2}.$$
\end{proposition}

We can now explain why Propositions~\ref{upper_bound_bulk} and~\ref{upper_bound_tail} yield the upper bound $\rho^* \lesssim \rhob + \rhot + \rhor$. 
We note that for any $p \in \mathcal{P}(\alpha, L', \cstar')$ such that $\|p-p_0\|_t\geq C \rho^*$, we either have $\int_{\TB} |p-p_0|^t \geq \frac{1}{2}C^t {\rhob}^t$ or $\int_{\T\left(\tuA\right)} |p-p_0|^t \geq \frac{1}{2}C^t {\rho}^t$.
Therefore, writing $\psi^* = \psib \lor \psit$, and taking the constant $C$ sufficiently large, we get by Propositions~\ref{upper_bound_bulk} and~\ref{upper_bound_tail} that
\begin{align*}
    &\mathbb{P}_{p_0}(\psi^*=1) 
    + \sup_{
        \substack{p \in \mathcal{P}(\alpha, L', \cstar')\\\\ \|p-p_0\|_t\geq C \rho^*}} 
        \mathbb{P}_p(\psi^*=0) \\
    \leq & ~ \mathbb{P}_{p_0}(\psi^*=1) 
    + \sup_{p \in \mathcal{P}_{Bulk}(\RadB\rhob)} \mathbb{P}_p(\psi^*=0) 
    + \sup_{p \,\in\, \mathcal{P}_{Tail}\left(\ThreshT\rho^*\right)} \mathbb{P}_p(\psi^*=0)\\
    \leq & ~\mathbb{P}_{p_0}(\psib=1)
    + \mathbb{P}_{p_0}(\psit=1) 
    + \sup_{p \in \mathcal{P}_{Bulk}(\RadB\rhob)} \mathbb{P}_p(\psib=0) 
    + \sup_{p \,\in\, \mathcal{P}_{Tail}\left(\ThreshT\rho^*\right)} \mathbb{P}_p(\psit=0)\\
    \leq & ~\frac{\eta}{2} + \frac{\eta}{2} = \eta.
\end{align*}

This ensures that $\rhob + \rhot + \rhor$ is an upper bound on the minimax separation radius, and that $\psib\lor\psit$ is a test reaching this bound. 
Note that in Proposition \ref{upper_bound_tail}, the separation radius on the tail is $\rho^* \asymp \rhob + \rhor$ if the bulk dominates, or $\rho^* \asymp \rhot + \rhor$ if the tail dominates.


\subsection{Tail lower bound}

To begin with, we state Proposition \ref{remainder_term} which handles the case where $\int_\T p_0 < \frac{\ctail}{n}$ for a large constant $\ctail$.

\begin{proposition}\label{remainder_term}
There exists a constant $n_0$ such that whenever $n \geq \nzero $ and $\int_\T p_0 < \frac{1}{n}\ctail$, it holds that \large{$\rho^* \gtrsim \rhor := L^{\frac{d(t-1)}{t(\alpha+d)}}n^{-\frac{\alpha t + d}{t(\alpha+d)}}$.}
\end{proposition}

To analyze the tail lower bound, Proposition \ref{remainder_term} allows us to make the following two assumptions \textit{wlog}:
\begin{itemize}
    \item[(a)] $\CBT \rhob \leq \rhot$ i.e. the tail dominates,
    \item[(b)] $\int_\T p_0 \geq \ctail/n$.
\end{itemize}

Indeed, for (a), if we have the converse inequality $\CBT \rhob > \rhot$, then Propositions \ref{upper_bound_bulk} and \ref{upper_bound_tail} already establish that $\rhob + \rhor$ is an upper bound over $\rho^*$. If the bulk dominates, then Propositions \ref{Lower_bound_bulk_proposition} and \ref{remainder_term} yield that $\rhob + \rhor$ is also a lower bound over $\rho^*$. Therefore, (a) can from now be assumed \textit{wlog}. \\

As for (b), Propositions \ref{upper_bound_bulk} and  \ref{upper_bound_tail} already establish that $\rhot + \rhor$ is an upper bound over $\rho^*$ when (a) holds. If $\int_\T p_0 < \frac{\ctail}{n}$, this upper bound further simplifies as $\rhot + \rhor \asymp \rhor$ and Proposition \ref{remainder_term} yields the matching lower bound $\rho^* \gtrsim\rhor$.\\

The following proposition yields a lower bound in the tail regime.

\begin{proposition}\label{Lower_bound_tail_proposition}
If the tail dominates, i.e. $\CBT \rhob \geq \rhot$, and if $\int_{\T(\uA)} p_0 \geq \frac{\ctail}{n}$, there exists a constant $\CtailLB$ such that $\rho^* \geq \CtailLB \rhot$.
\end{proposition}

Proposition \ref{Lower_bound_tail_proposition} is a corollary of Proposition \ref{lower_bound_tail} proved in Appendix \ref{LB_tail_appendix}. 
The proof again relies on a Le Cam two-point argument, with a prior defined as a mixture of probability distributions over the alternative space. 
\vspace{1mm}

The definition of the prior is very involved and is deferred to Appendix~\ref{LB_tail_appendix}. 
At a high level, it is defined by first covering the tail with disjoint cells with the same edge length $\hT$. 
Then, we perturb $p_0$ on each cell independently at random. 
On most cells, the perturbation added to $p_0$ is small and negative. 
Conversely, on a small number of cells, the perturbation added to $p_0$ is very large and positive. 
This prior can be understood as a ``sparse'' perturbation of $p_0$, and it contrasts with the paper~\cite{balakrishnan2019hypothesis}, which does not propose a tail lower bound.

\hfill

\section{Discussion}\label{Discussion}

\subsection{Discussion of the results}
\subsubsection{Rates}
For $n$ larger than a constant $\nzero$, we prove matching upper and lower bounds leading to the following expression for the critical radius
$$ \rho^*(p_0, \alpha,L, n) \asymp \TL^\frac{1}{4\alpha+d} \left(\int_{\B(\uI)}p_0^r\right)^\frac{(4-t)\alpha + d}{t(4\alpha+d)} + \TL^\frac{t-1}{\alpha+d} \left(\frac{1}{n}+p_0\big[\T(\uA)\big] \right)^\frac{(2-t)\alpha+d}{t(\alpha+d)},$$
where $r = \frac{2\alpha t}{4\alpha+d}$ and $\TL = L^d/n^{2\alpha}$. The bulk term involves the quantity $n^{-2\alpha/(4\alpha+d)}$ which is the classical non-parametric rate for testing the null hypothesis of the uniform distribution on $[0,1]^d$ against the alternative composed of $(\alpha,L)$-Hölder densities separated from $p_0$ in $\|\cdot\|_t$ norm (see e.g. \cite{ingster2012nonparametric}). 
This rate is faster than the non-parametric rate of estimation $n^{-\frac{\alpha}{2\alpha+d}}$.
We observe that, for fixed $p_0$, the quantities $\uA$ and $\uI$ both decrease to zero as $n$ increases. 
Consequently, if a fixed $p_0$ is supported on a fixed bounded domain, then the bulk eventually dominates for $n$ larger than a critical value (depending on $p_0$). 
The asymptotic rate therefore simplifies as $\rho^*(p_0,n,\alpha,L) \asymp \TL^\frac{1}{4\alpha + d}\left(\int_{\Omega}p_0^r\right)^\frac{(4-t) + d}{t(4\alpha+d)}$ which decays with $n$ at the non-parametric rate of testing $n^{-\frac{2\alpha}{4\alpha+d}}$. 
However, when $\Omega$ is not bounded, we give in Subsection \ref{Examples_subsection} examples of fixed null densities $p_0$ for which the tail always dominates, leading to critical radii $\rho^*(p_0,n,\alpha,L)$ decaying with $n$ at slower rates than $n^{-\frac{2\alpha}{4\alpha+d}}$.\\

In the tail test statistic, we combine the tests $\psi_1$ and $\psi_2$ from \eqref{test_psi1} and \eqref{test_psi2}. Test $\psi_1$ compares the first order moment of $p$ with that of $p_0$. 
The test $\psi_2$ checks that no cell $\widetilde C_j$ contains at least two observations. 
This condition is actually equivalent to checking that the second moment of $p$ is no larger than that of $p_0$. Indeed, on the tail, the second moment of $p_0$ is so small that, \textit{whp}, any cell $\widetilde C_j$ contains at most one observation under $H_0$. Conversely, if the second moment of $p$ is substantially larger than that of $p_0$, then \textit{whp} one of the cells will contain at least two observations.\\

\subsubsection{Discussion on the regularity conditions - Assumption \eqref{simplifyingAssumption}}

Our results constitute an attempt to address the case of arbitrary $\alpha$-Hölder densities over $\R^d$. Our analysis relies on Assumption \eqref{simplifyingAssumption}, which is automatically satisfied for $\alpha \leq 1$. 
For $\alpha>1$, Assumption \eqref{simplifyingAssumption} essentially implies two limitations.
\begin{itemize}
    \item 
    {\bf Limitation for two points which are close:} First, any $p$ satisfying \eqref{simplifyingAssumption} should be "approximately constant" over the balls $B(x,h(x))$ - namely the Euclidean balls centered at $x$ with radius $h(x)\asymp \left(p(x)/L\right)^{1/\alpha}$.
    Formally, for any $y\in B(x,h(x))$ and for all $c \in [\frac{1}{2},1)$, it imposes that $\frac{p(y)}{p(x)} \in \big[1 \pm c\big]$ whenever $y\in B(x,h(x))$ where $h(x)=\left(\frac{c - \cstar }{L}p(x)\right)^{1/\alpha}$. Noting that the bulk precisely consists of all $x$ such that $C\hB(x) \leq  h(x)$ for some $C>0$, this condition allows us to exclude fast variations of $p$ and $p_0$ over the bulk. 
 
    \item {\bf Limitation for two points which are far:} Second, when $y \notin B(x,h(x))$, Assumption \eqref{simplifyingAssumption} bounds the maximum deviations of $p$ as $|p(x) - p(y)| \lesssim L\|x-y\|^\alpha$. This condition naturally arises for $x$ corresponding to small values of $p(x)$. In particular, it allows us to exclude fast variations of $p$ and $p_0$ over the tail.
\end{itemize}
This assumption is therefore implied by - and in fact, up to multiplicative constants, equivalent to - assuming that for any $m,M$ such that $0\leq 2(1+\cstar)m<M$ and such that the level sets $\{p\leq m\}$ and $\{p \geq M\}$ are not empty, the smallest distance between any two points in these level sets should be at least $(M/L)^{1/\alpha}$. 
This is implied by assuming that whenever $p$ has a local minimum at a point $x$ where $p(x)$ is close to $0$, then all its derivatives up to order $\lfloor \alpha \rfloor$ are zero at this point. 
We believe that Assumption \eqref{simplifyingAssumption} is not very restrictive. For example, it is always satisfied for unimodal densities, or any densities that are monotone outside of a fixed compact, such that the ratio of the upper and lower bound of $p$ on this compact is bounded by a constant. 
\vspace{2mm}


\subsubsection{Influence of the norm}

We cover the scale of all $L_t$ distances for $t\in [1,2]$. Among these distances, only the $L_1$ distance is an $f$-divergence. 
We also identify a duality between the norms: When testing in $L_t$ norm, the bulk radius is expressed in terms of the $L_r$ norm, where $r$ and $t$ are linked through the relation $r = \frac{2\alpha t}{(4-t)\alpha+d}$. 
Depending on the value of $r$, the hardest and easiest null $p_0$ to test are different.
\boundedcase{For $r < 1$, $\|p_0\|_r$ is maximal when $p_0$ uniformly small and is minimal for spiked $p_0$'s (see the example of the spiky null in Subsection \ref{Examples_subsection}). 
Over a fixed bounded domain $[0,\lambda]^d$ (for large $\lambda$ depending on $L$), the uniform distribution is therefore the most difficult density to test, while the spikiest densities are the easiest ones.
This hierarchy is reversed when $r>1$ (see Subsection \ref{Examples_subsection}).\\

}\unboundedcase{If $r < 1$, then $\|p_0\|_r$ can be made arbitrarily large if the density $p_0$ is sufficiently small everywhere on $\R^d$. Conversely, $\|p_0\|_r$ is minimal for spiked $p_0$ and so is $\rhob(p_0)$ (see the example of the spiky null in Subsection \ref{Examples_subsection}). This hierarchy is reversed when $r \geq 1$.
}\\

One could argue that the $L_t$ distances lack operational meaning (for example, it is possible for a density $ p$ to be arbitrarily close to the true density $p_0$ in $L_2$ distance, but trivially distinguishable from $p_0$). 
In fact, it is true that our expression of the minimax separation radius $\rho^*$ is not invariant by rescaling when $t>1$ (see Proposition~\ref{Rescaling_proposition}). 
In particular, it is not a ``good'' notion of difficulty for the considered testing problem. 
The right notion of difficulty is rather given by the \textit{ratio}  
\begin{align}
    \operatorname{Diff}(p_0,L,\alpha,n,t) = \frac{\rho^*(p_0,L,\alpha,n,t)}{\rhor(\alpha,L,n=1,t)},\label{eq:def_notion_difficulty}
\end{align}
where $\rhor$ is defined in~\eqref{eq:def_rhor}. 
By Corollary~\ref{cor:notion_difficulty}, $\operatorname{Diff}(p_0,L,\alpha,n,t)$ is invariant by rescaling. 
The idea is to compare the numerical value $\rho^*(p_0,L,\alpha,n,t)$ with the remainder term $\rhor(\alpha,L,n=1,t) \asymp L^{\frac{d(t-1)}{t(\alpha + d)}}$. 
When $n=1$, this term represents the minimax separation radius of the \textit{easiest} density $p_0^*$ in the class, namely the best approximation of the Dirac distribution in the class (see Example 4). 
Taking $n=1$ in the denominator ensures that the dependence of $\operatorname{Diff}(p_0,L,\alpha,n,t)$ with respect to $n$ is the same as that of $\rho^*(p_0,L,\alpha,n,t)$.
In a nutshell, the term $L^{\frac{d(t-1)}{t(\alpha + d)}}$ furnishes a benchmark to characterize the \textit{relative} difficulty of testing $p_0$, which allows us to exclude artificial rescaling problems. 
Note that for the $L_1$ distance, which is an $f$-divergence, no rescaling is needed since $L^{\frac{d(t-1)}{t(\alpha + d)}} = 1$.

\subsection{Examples}\label{Examples_subsection}
To illustrate our results, we give examples of null densities $p_0$ and of the associated radii.\\ 

\textbf{Example 1} (Uniform null distribution over $\Lambda = [0,\lambda]^d$). 

\vspace{3mm}

We consider $p_0 = \lambda^{-d}$ over $[0, \lambda]^d$ and set $|\Lambda| = \lambda^d$. \unboundedcase{Note that this example cannot be handled by our present results since $p_0$ is not defined over $\R^d$. However, this case has already been analyzed in \cite{ingster2012nonparametric}, which we give here for a comparison with our results. The asymptotic minimax rate (as $n \to \infty$) writes:
\begin{equation}\label{rate_uniforme}
    \rho^*(\alpha, L, n, p_0) \asymp \TL^\frac{1}{4\alpha+d} |\Lambda|^\frac{(4-3t)\alpha+d}{(4\alpha+d)t}, ~~ \text{ where $\TL = L^d/n^{2\alpha}$.}
\end{equation}
} 

\boundedcase{We get:
\begin{equation}\label{rate_uniforme}
    \rho^*(\alpha, L, n, p_0) \asymp \TL^\frac{1}{4\alpha+d}\left( |\Lambda| ~ \land ~ \TL^{-\frac{1}{\alpha+d}}\right)^\frac{(4-3t)\alpha+d}{(4\alpha+d)t}, ~~ \text{ where $\TL = L^d/n^{2\alpha}$.}
\end{equation}}

\boundedcase{Testing equality to $p_0$ therefore involves two regimes. Assume e.g. that $L=1$. As long as $n < \left(\cI |\Lambda|\right)^\frac{\alpha+d}{2\alpha}$, the rate decays as $n^{-2\alpha\frac{t-1}{\alpha+d}}$. For larger $n$, the rate becomes $n^{-\frac{2\alpha}{4\alpha+d}}$, which is the classical (asymptotic) rate associated to the uniform distribution \cite{ingster2012nonparametric}.\\}

This is the most commonly studied setting in the literature \cite{ingster2012nonparametric}, \cite{ermakov1995minimax}, \cite{ingster1986minimax}, \cite{ingster1984asymptotically}, \cite{gine2016mathematical}. Indeed, for fixed  constants $C>c>0$, and any smooth density $p_0$ satisfying $c \leq p_0 \leq C$, its critical radius is given by \eqref{rate_uniforme}. However, when we relieve this last assumption and let $L$ be arbitrary, $\rho^*(\alpha,L,n,p_0)$ can substantially deviate from \eqref{rate_uniforme}.\\

\textbf{Example 2} (Gaussian null)
\vspace{3mm}

Suppose $p_0$ is the density of $\mathcal{N}(0,\sigma^2 I_d)$ over $\R^d$, where $\sigma>0$. Fix $\alpha,L$ and $\sigma$, and consider the asymptotics as $n \to +\infty$. The \textit{asymptotic} minimax rate associated to $p_0$ is 
\begin{equation}\label{rate_gaussian}
    \rho^*(\alpha,L,n,p_0) \asymp \TL^\frac{1}{4\alpha+d}\, (\sigma^d)^{\frac{(4-3t)\alpha+d}{t(4\alpha+d)}}.
\end{equation}
This asymptotic rate exclusively corresponds to the bulk rate $\rhob$. It decays with $n$ at the classical non-parametric rate of testing $n^{-\frac{2\alpha}{4\alpha+d}}$. Note the similarity between \eqref{rate_uniforme} and \eqref{rate_gaussian} when $\sigma^d$ plays the role of $|\Lambda|$. Regardless of the fixed constant $\sigma$, testing equality to $\mathcal{N}(0,\sigma^2 I_d)$ or to $\text{Unif}\left([-\sigma,\sigma]^d\right)$  are asymptotically equally difficult.\\


\boundedcase{
\textbf{Example 3} (Arbitrary $p_0$ over $\Omega = [0,1]^d$ with $L=1$.)
\vspace{3mm}

For any $p_0 \in \mathcal{P}(\alpha,1,\cstar)$, the rate simplifies as:
\begin{equation}\label{rate_arbitrary}
    \rho^*(\alpha, 1, n, p_0) \asymp n^{-\frac{2\alpha}{4\alpha+d}}.
\end{equation}
Noticeably, this rate is independent of $p_0$ and coincides with \eqref{rate_uniforme}. Indeed, fixing $\Omega=[0,1]^d$ and $L=1$ constrains $p_0$ to have only limited variations so that all of its mass cannot be concentrated on a small part of the domain. Therefore, $p_0$ cannot deviate enough from the uniform distribution to affect the rate. This case illustrates that, for fixed \textit{bounded} domain and for $L=1$, all of the testing problems \eqref{testing_pb} are equally difficult regardless of $p_0$. Conversely, when $L$ or $\Omega$ are allowed to depend on $n$ or when $\Omega$ is unbounded, the local rate $\rho^*(\alpha,L,n,p_0)$ can significantly deviate from \eqref{rate_uniforme}. This is illustrated in the following examples. \\}

\unboundedcase{
\textbf{Example 3} (Arbitrary $p_0$ with support over $\Omega' = [-1,1]^d$ and $L=1$.)
\vspace{3mm}

We recall that the support of $p:\R^d \longrightarrow \R$ is $\{x \in \R^d: p(x) \neq 0\}$. In this example, we therefore consider an arbitrary null density $p_0$ defined over $\R^d$ which is zero outside $\Omega' = [-1,1]^d$. For any such $p_0 \in \mathcal{P}(\alpha,1,\cstar)$, the rate simplifies as:
\begin{equation}\label{rate_arbitrary}
    \rho^*(\alpha, 1, n, p_0) \asymp n^{-\frac{2\alpha}{4\alpha+d}}.
\end{equation}
Noticeably, this rate is independent of $p_0$ and coincides with \eqref{rate_uniforme}. Indeed, fixing the support $\Omega'=[-1,1]^d$ and $L=1$ constrains $p_0$ to have only limited variations. In this case, all its mass cannot be concentrated on a small part of the domain. This example illustrates that, for fixed \textit{bounded} support and for $L=1$, all of the testing problems \eqref{testing_pb} are equally difficult regardless of $p_0$. Conversely, when $L$ or $\Omega'$ are allowed to depend on $n$ or when the support $\Omega'$ is unbounded, the local rate $\rho^*(\alpha,L,n,p_0)$ can significantly deviate from \eqref{rate_uniforme}. This is illustrated in the following examples. \\}

\textbf{Example 4} (Spiky null).
\vspace{3mm}

Let $\Omega = \R^d$. Define $f\geq 0$ such that $f \in H(\alpha,1) \cap C^\infty$ over $\R^d$ and $f$ is nonzero only over $\left\{x \in \R^d : \|x\|<1/2 \right\}$. We here moreover assume that $f$ satisfies
\begin{align*}
    \forall x,y \in \R^d: |f(x)-f(y)| \leq \cstar f(x) + \|x-y\|^\alpha.
\end{align*}
The spiky null density is defined as follows:
\begin{equation}\label{def_spiky_null}
    p_0(x) = La^\alpha f\left(\frac{x}{a}\right)
\end{equation}
where $a = \left(\Cint L\right)^{-\frac{1}{\alpha+d}}$. Informally, this corresponds to an approximation of the Dirac distribution $\delta_0$ by a density in $\mathcal{P}(\alpha,L,\cstar)$. 
In this case we have:
\begin{equation}\label{rate_spiky}
    \rho^*(p_0,\alpha,L,n) \asymp L^\frac{d(t-1)}{t(\alpha+d)} n^{-\frac{2\alpha}{4\alpha+d}}.
\end{equation}
Note that the density \eqref{def_spiky_null} is supported on $[-a,a]^d$ and the corresponding minimax rate is the same as for the uniform density over $[-a,a]^d$.
Note that there is only one regime: Indeed, by the choice of $a$, the value $|\Omega| \asymp L^{-d(\alpha+d)}$ is always smaller than $\TL^{-\frac{1}{\alpha+d}}$ up to constants.
Assume now that $L \to \infty$ as $n \to \infty$, so that $p_0$ is supported over $\widetilde \Omega = [-1,1]^d$ for $n$ large enough. Suppose moreover that $\TL \to 0$ so that, over $\widetilde\Omega$, the rate \eqref{rate_uniforme} simplifies as $\TL^\frac{1}{4\alpha+d}$. 
We then note that \eqref{rate_spiky} is faster than $\TL^\frac{1}{4\alpha+d}$ if, and only if, $r \leq 1$. It is possible to show that over a bounded domain, for all $L_t$ norms such that $r\leq 1$, the uniform null distribution has maximum separation radius whereas the spiky null has the smallest one. Conversely for $t$ such that $r >1$, the uniform distribution has minimum separation radius whereas the spiky null has the largest one.\\

Note that by rescaling (see Proposition \ref{Rescaling_proposition}), letting $L \to +\infty$ for fixed $n$ and fixed support is equivalent to letting the support size go to infinity while $L,n$ are fixed. This example illustrates that over growing domains, some compactly supported densities can substantially deviate from the uniform distribution over their support. \\

\textbf{Example 5} (Pareto null)
\vspace{3mm}

We place ourselves over \boundedcase{$\Omega = \R_+$}\unboundedcase{$\R$} (hence $d=1$) and consider for $x_0 > 0$ and $\beta \in (0,1)$ the null density $p_0(x) = \frac{\beta x_0^\beta}{x^{\beta+1}}$ over $[x_1, +\infty)$. Here $x_1>x_0$ is chosen so that $p_0$ can be extended over $(-\infty, x_1]$ to get a density in $\mathcal{P}(\alpha,L,\cstar)$ which is $0$ over $(-\infty, x_{-1}]$ for $x_{-1}< x_0$. 
For simplicity we only give the rate for $\alpha \leq 1$ and $t=1$ i.e. for the total variation distance, although the general rate can be established for all $\alpha>0$ and $t\in [1,2]$. The minimax separation radius simplifies as
\begin{equation}\label{rate_pareto}
    \rho^*(\alpha,L,n,p_0) \asymp \TL^\frac{\beta}{3\beta+1} = \left(\frac{L^d}{n^{2\alpha}}\right)^ \frac{\beta}{3\beta+\alpha+1} .
\end{equation}

Interestingly, \textit{for all} $L,n$, this rate exclusively corresponds to the dominating term $\rhot$. This example illustrates that for some heavy-tailed densities defined on unbounded domains, the separation distance substantially deteriorates compared to the spiky null with $L=1$ - and justifies the importance of establishing a tight rate in the tail regime. Noticeably, the whole scale of rates from $1$ to $n^{-\frac{2\alpha}{4+\alpha}}$ can be obtained for $\alpha \leq 1$ and are all slower than the classical non-parametric rate of testing $n^{-\frac{2\alpha}{4\alpha+1}}$. For $t>1$, a slower rate (depending on $\alpha,\beta,t$) can similarly be observed as compared to the case of a  spiky null density. \\


\boundedcase{\textbf{Remark:} Assume $L=1$ and fix \unboundedcase{$p_0: \R^d \to \R$ any null density supported on $[0,1]^d$}\boundedcase{$p_0: [0,1]^d\to \R$}. From the previous examples, its critical radius cannot substantially deviate from that of the uniform distribution over $[0,1]^d$. 
However, this no longer holds even for $L=1$ if we consider the restriction of $p_0$ to some subset $\Omega'$ of $[0,1]^d$.
More precisely, suppose that we observe iid $X_1, \dots, X_n$ with unknown density $p$ over $\Omega$ and that for some $\Omega' \subset [0,1]^d$ we would like to test 
$$ H_0 : ~p\,\mathbb{1}_{\Omega'} = p_0\,\mathbb{1}_{\Omega'}  ~~ \text{ vs } ~~ H_1: \big\|(p-p_0)\mathbb{1}_{\Omega'}\big\|_t \geq \rho.$$ 
If $\int p_0\mathbb{1}_{\Omega'}$ is much smaller than $1$, then even with $L=1$, the critical radius can substantially deviate from that of the uniform distribution.
}

\subsection{Comparison with prior work}
\subsubsection{Special case of the $\|\cdot\|_1$ norm (total variation)}\label{comparison_TV}

The case of the $L_1$ norm has been studied in \cite{balakrishnan2019hypothesis}. Here we state their main result for density testing. 
Suppose that we observe $X_1, \dots, X_n$ with density $p$ over $\Omega$ and fix a particular density $p_0$ over $\Omega$. Assume that $p$ and $p_0$ are in $H(\alpha,L)$ for $\alpha \leq 1$ and consider the identity testing problem
\begin{equation}\label{testing_pb_WB}
    H_0: p=p_0 ~~ \text{ vs } ~~ H_1: \|p-p_0\|_1 \geq \rho  \text{ and $p\in H(\alpha,L)$}.
\end{equation}
Problem \eqref{testing_pb_WB} is a special case of our setting where $\alpha \leq 1$ and $t=1$. For all $\sigma>0$, let $\mathcal{B}_\sigma = \{B: \mathbb{P}_{p_0}(B) \geq 1-\sigma\}$ and define the functional
\begin{equation}\label{functional_WB}
    T_\sigma(p_0) = \inf_{B \in \mathcal{B}_\sigma}\left(\int_{B}p_0^\gamma\right)^{1/\gamma},
\end{equation}
where $\gamma = \frac{2\alpha}{3\alpha+d}$. For two explicit constants $c,C>0$, define the upper and lower critical radii as the solutions of the fixed-point equations:
\begin{equation}\label{critical_radii_WB}
    v_n(p_0) = \left(\frac{L^{d/2\alpha}T_{Cv_n(p_0)}(p_0)}{n}\right)^\frac{2\alpha}{4\alpha+d} ~~ \text{ and } ~~ w_n(p_0) = \left(\frac{L^{d/2\alpha}T_{cw_n(p_0)}(p_0)}{n}\right)^\frac{2\alpha}{4\alpha+d}.
\end{equation}

\begin{theorem}\label{rate_lipschitz_WB}[Balakrishnan, Wasserman (2017)]
    For $\alpha \in(0,1)$, there exist two constants $c,C>0$ such that the critical radius for problem \eqref{testing_pb_WB} satisfies $\rho^*(p_0, \alpha, L, n) \leq C w_n(p_0)$. Moreover, if $p_0\in H(\alpha, c'L)$ where $c' \in (0,1)$, then it holds that $\rho^*(p_0, 1, L, n) \geq c v_n(p_0)$.
\end{theorem}

We now state our Theorem \ref{main_th} in the special case $\alpha = 1$ and $t=1$. For all $L>0$, for all cubic domain $\Omega$, for all $p_0$ $L$-Lipschitz over $\Omega$, we have:
\begin{equation}\label{rate_lipschitz_TV}
    \rho^*(p_0,L,n) ~ \asymp ~  \left(\frac{L^{d/2}}{n}\left(\int_{\B(\uI)}  p_0^r\right)^{1/r}\right)^\frac{2}{4+d} + ~ p_0\big[\T(\uA)\big] ~+~ \frac{1}{n}, ~~~~ \text{ where $r=\frac{2}{3+d} = \gamma$.} 
\end{equation}

We first note that our bulk term $\left(\frac{L^\frac{d}{2}}{n}\left(\int_{\B(\uI)}  p_0^r\right)^{1/r}\right)^\frac{2}{4+d}$ is the analog of $v_n(p_0)$ and $w_n(p_0)$ in Theorem \ref{rate_lipschitz_WB}. However, it is defined explicitly in terms of $p_0$ and does not involve solving a fixed-point equation. In Theorem \ref{rate_lipschitz_WB} the critical radius is bracketed between $v_n(p_0)$ and $w_n(p_0)$. Although these two quantities are of the same order in most usual cases, the authors in \cite{balakrishnan2019hypothesis} discuss pathological cases for which $w_n(p_0) \ll v_n(p_0)$. This non-tightness can be attributed to possibly large discrepancies between $T_{\sigma_1}(p_0)$ and $T_{\sigma_2}(p_0)$ (with $\sigma_1 \neq \sigma_2$) for some carefully chosen $p_0$. In the present paper, we bridge this gap by identifying matching upper and lower bounds in the considered class and more general classes corresponding to any $\alpha>0$.\\

Moreover, the authors of~\cite{balakrishnan2019hypothesis} need to consider the Poissonized minimax risk. 
Indeed, their upper bound involves reducing the continuous testing problem to a discrete one, by binning the domain into cubic cells to recover a multinomial distribution. 
Classically, the Poissonization trick can be used to transform a multinomial histogram into a vector with independent entries, simplifying the proof of the upper bound in~\cite{balakrishnan2019hypothesis}. 
In comparison, we use a Kernel estimator for the bulk upper bound, which does not rely on any similar notion of spacial independence. 
Our tail upper bound in fact also relies on a binning of the domain into cubic cells, but we reject $H_0$ if the \textit{total number} of observations in the tail is too large, or if at least one cell contains two observations or more. 
For the first part, controlling the sum over all cells does not require any independence across the cells. 
For the second part, we control the type I and type II errors through a union bound, which does not require any independence of the cells either.
Nevertheless, in the proof of the tail lower bound, we do use a Poissonization argument to compare the risk induced by our prior with the corresponding Poissonized risk, which is easier to control. \\

In $L_1$ separation, we identify a tail contribution given by $\rhot \asymp \int_{\T(\uA)} p_0$. 
\boundedcase{If $\Omega$ is unbounded}
\unboundedcase{Note that, since $\Omega=\R^d$}, this contribution implies that one can always pick $p_0$ depending on $L$ and $n$ so that $\rhot \asymp 1$. Indeed, for fixed $n,L$ consider a suitably smooth density $p_0$ such that $\max_\Omega p_0 \leq \cI \TL^\frac{1}{\alpha+d}$. 
Then by the definition of $\uI$ from equation \eqref{def_uI}, we have $\int_{\T(\uA)} p_0 \geq \int_{\T(\uI)} p_0 = 1$. 
This illustrates that even in the most favorable regime $L\to 0$ and $n \to \infty$, there exist smooth null densities $p_0$ over $\Omega$ associated to the trivial maximal separation radius $\rho^*(p_0,n,\alpha,L) \asymp 1$. 
These correspond to the worst case densities over the class.
On unbounded domains, it is therefore crucial to identify local results since the global problem (i.e. the worst case over the class) has a trivial rate.
The fact that \textit{estimation} of $\alpha$-Hölder densities in total variation over unbounded domains has a trivial rate was already highlighted in \cite{ibragimov1984more}, \cite{juditsky2004minimax}, \cite{goldenshluger2012adaptive}. Recalling that estimation is more difficult than testing, we therefore recover this result on density estimation.


\subsubsection{Comparison with the discrete setting \cite{chhor2020sharp}}

The paper \cite{chhor2020sharp} considers a discrete analog of the present problem. 
Suppose we observe iid $X_1, \dots, X_n$ distributed as $\mathcal{M}(p)$ where $\mathcal{M}(p)$ denotes the multinomial distribution over $\{1,\dots, d\}$. 
When $X \sim \mathcal{M}(p)$, we have $\forall j \in \{1,\dots, d\}: \mathbb{P}(X = j)=p(j)$. Suppose we are given a known discrete distribution $p_0$ over $\{1,\dots,d\}$ and assume \textit{wlog} that the entries of $p_0$ are sorted in decreasing order. 
We consider the following testing problem:
\begin{equation}\label{discrete_testing_problem}
    H_0: p=p_0 ~~ \text{ vs } ~~ H_1: \|p-p_0\|_t \geq \rho,
\end{equation}
where $\|p-p_0\|_t = \Big(\sum\limits_{j=2}^d \big|p(j)-p_0(j)\big|^t\Big)^{1/t}$ and $t \in [1,2]$. Introduce the index $I$ as
\begin{equation}\label{def_I_discret}
    I = \min \Big\{j \in \{1,\dots,d\} ~\big|~ \sum_{j>I} n^2 p^2(j) \leq \cIM \Big\},
\end{equation}
for some  constant $\cIM = \cIM(\eta,t)$. Moreover, define the index $A$ as:
\begin{equation}\label{def_A_discret}
    A = \min \Big\{j\leq I ~\big|~ p^{b/2}(j) \geq \frac{\cAM}{\sqrt{n}\big(\sum\limits_{j \leq I} p_0^{r'}(j)\big)^{1/4}}\Big\}
\end{equation}
where $b = \frac{4-2t}{4-t}$ and $r'=\frac{2t}{4-t}$. 
Then one has the following result:
\begin{theorem}\label{theorem_discret}[Chhor, Carpentier (2020)]
    It holds that:
    $$ \rho^*(p_0, n, d, t) ~\asymp~ \rhobM ~+~ \rhotM ~+~ \rhorM,$$
    where $~~~~~~\rhobM = \sqrt{\frac{1}{n}\left\|(p_0)_{\leq I}^{-\max}\right\|_{r'}}$, $~~~~~~\rhotM = {n^{\frac{2}{t}-2}\|(p_0)_{>A}\|_1^{1-2/t}}$, $~~~~~~~\rhorM =  1/n$, \\$\left\|(p_0)_{\leq I}^{-\max}\right\|_{r'} = \left(\sum\limits_{j=2}^I p_0^{r'}(j)\right)^{1/r'}$ and $~~~~~\|(p_0)_{>A}\|_1 = \sum\limits_{j>A} p_0(j)$.
\end{theorem}

Upper and lower bounds in $L_1$ separation for the multinomial identity testing were previously known (see \cite{valiant2017automatic,blais_et_al:LIPIcs:2017:7536,blais2019distribution} for the local problem, and~\cite{paninski2008coincidence} for the global one).\\ 

This discrete setting and our present continuous setting involve many similar phenomena. The discrete tail, defined as $\{p_0(j) \;|\; j > A\}$, is designed so that \textit{whp} under $H_0$, no coordinate $j>A$ is observed twice among the $n$ data $X_1, \dots, X_n$. 
Our approach in the present paper aims at transferring this tail definition to the continuous setting.\\



Theorem \ref{theorem_discret} identifies a three-fold contribution to $\rho^*$ (bulk, tail and remainder term) which is similar to ours. However, there are some substantial challenges in our continuous setting compared to the discrete testing problem.
\begin{itemize}
    \item First, the discretization that we adopt (for both the lower bounds and for the tail statistic), as well as the bandwidth of the kernel (for the bulk statistic) must depend on $p_0$. Finding this optimal discretization/bandwith is challenging in itself, and raises some fundamental information theoretic questions, as well as some difficult technical issues. For instance, the bulk test statistic is based on the integral of a function of an inhomogeneous kernel approximation of $p$, which is very different from what is done in the discrete setting.
    \item Second, even when this discretization/kernelisation has been done, the test statistics are not direct analogs of the discrete test statistics from~\cite{chhor2020sharp}. In both cases - discrete and continuous - the bulk test statistics is a reweighted $\chi^2$ test statistic with inhomogeneous weights, depending on each coordinate of $p_0$. However in the continuous case, the reweighting factor $\omega(x)$ cannot be directly deduced from the discrete setting. Indeed, there is a distortion in the integral coming from the non-homogeneity of the Kernel bandwidth, whose effect has to be taken into account on top of the non-homegeneity coming from $\omega(x)$.
\end{itemize}

\hfill

\textbf{Acknowledgements} We would like to thank warmly Alexandre Tsybakov for insightful comments and a careful rereading of the present manuscript, as well as Cristina Butucea and Rajarshi Mukherjee for enlightening discussions and ideas.\\
The work of A. Carpentier is partially supported by the Deutsche Forschungsgemeinschaft (DFG) Emmy Noether grant MuSyAD (CA 1488/1-1), by the DFG - 314838170, GRK 2297 MathCoRe, by the FG DFG
, by the DFG CRC 1294 'Data Assimilation', Project A03, by the Forschungsgruppe 5381 "Mathematische Statistik im Informationszeitalter – Statistische Effizienz und rechentechnische Durchführbarkeit", and by the UFA-DFH through the French-German Doktorandenkolleg CDFA 01-18 and by the SFI Sachsen-Anhalt for the project RE-BCI.

\bibliographystyle{plain}

\appendix

\section{Relations between the cut-offs}\label{relations_btw_cutofss}

We will also use the following notation:
\begin{equation}\label{def_widetilde_T}
    \widetilde{\mathcal{T}}(\uI) = \left\{x \in \Omega: p_0(x) \leq \uI \right\}.
\end{equation}

Note that in the above definition, the inequality $p_0(x) \leq \uI$ is not strict, whereas the inequalities in the definitions of $\T(\uI)$ and $\T(\uA)$ are strict. Furthermore, we define three different lengths $\hT(\uI), \hT(\uA)$ and $\barhT$ as follows:
\begin{equation}\label{def_htail_uI}
    \hT(\uI) = \Big(n^2 L \; \int_{\T(\uI)} \; p_{0}\Big)^{-\frac{1}{\alpha + d}}.
\end{equation}

\begin{equation}\label{def_htail_uA}
    \hT(\uA) = \Big(n^2 L \; \int_{\T(\uA)} \; p_{0}\Big)^{-\frac{1}{\alpha + d}}.
\end{equation}

\begin{equation}\label{def_barhtail_uI}
    \barhT(\uI) = \Big(n^2 L \; \int_{\widetilde{T}(\uI)} \; p_{0}\Big)^{-\frac{1}{\alpha + d}}.\;
\end{equation}
and prove in Subsection \ref{Technical_results_UB_Tail} that they all differ at most by a multiplicative constant. \\

\begin{lemma}\label{uA>uI_then_=}

Let $\widetilde{\mathcal{T}}(\uI)$ be defined as in \eqref{def_widetilde_T}. \begin{itemize}
    \item\label{T(uI)_leq_cI} We have $$\int_{\T(\uI)} \; p_0^2 \leq \frac{\cI}{n^2 \hT^d(\uI)}.$$
    \item Moreover, if $\max\limits_{\Omega} p_0 \geq \uI$, then it holds:
$$ \frac{\int_{\widetilde{\mathcal{T}}(\uI)} \; p_0^2}{\left(\int_{\widetilde{\mathcal{T}}(\uI)} \; p_0\right)^{d/(\alpha+d)}} \geq \cI \frac{L^{d/(\alpha+d)}}{n^{2\alpha/(\alpha+d)}},$$ implying:
$$\int_{\widetilde{\mathcal{T}}(\uI)} \; p_0^2 \geq \frac{\cI}{n^2 \barhT^d}.$$
\end{itemize}

\end{lemma}

\begin{proof}[Proof of Lemma \ref{uA>uI_then_=}]
By definition of $\uI$, 
there exists a sequence $(u_j)_{j \in \N}$ such that $u_j \uparrow \uI$ and $\frac{\int_{\mathcal{T}(u_j)} \; p_0^2}{\left(\int_{\mathcal{T}(u_j)} \; p_0\right)^{d/(\alpha+d)}} \leq \cI \TL^\frac{1}{\alpha+d}$. The dominated convergence theorem yields the result.

For the second part, when $u \downarrow \uI$, we have by the dominated convergence theorem that $\int_{\mathcal{T}(u)} \; p_0^2 \to \int_{\widetilde{\mathcal{T}}(\uI)} \; p_0^2$ and $\int_{\mathcal{T}(u)} \; p_0 \to \int_{\widetilde{\mathcal{T}}(\uI)} \; p_0$ so that $\frac{\int_{\mathcal{T}(u)} \; p_0^2}{\left(\int_{\mathcal{T}(u)} \; p_0\right)^{d/(\alpha+d)}}  \longrightarrow \frac{\int_{\widetilde{\mathcal{T}}(\uI)} \; p_0^2}{\left(\int_{\widetilde{\mathcal{T}}(\uI)} \; p_0\right)^{d/(\alpha+d)}} $. Moreover, by definition of $\uI$, for all $u > \uI$ we have $\frac{\int_{\mathcal{T}(u)} \; p_0^2}{\left(\int_{\mathcal{T}(u)} \; p_0\right)^{d/(\alpha+d)}} > \cI \frac{L^{d/(\alpha+d)}}{n^{2\alpha/(\alpha+d)}}$ (since $\uI \leq \max\limits_\Omega p_0$), which yields the result. 
\end{proof}

We now define 
\begin{equation}\label{def_rho_bar_uI}
    \large{\overline{ \rhot} =   \left[\TL^\frac{t-1}{\alpha+d}\frac{\left(\int_{ \T(\uA)}p_{0}\right)^\frac{(2-t)\alpha+d}{\alpha}}{\left(\int_{\widetilde \T(\uI)}p_{0}\right)^{\frac{(2-t)\alpha+d}{\alpha }\frac{d}{\alpha + d} }}\right]^{1/t} .}
\end{equation}

\begin{lemma}\label{uA>uI_then_tail}
If $\uA > \uI$ and $\max\limits_{\Omega} p_0 \geq \uI$ then $\overline{ \rhot} \geq \CL{\ref{uA>uI_then_tail}} \; \rhob$ where $\CL{\ref{uA>uI_then_tail}} = \, \cI^\frac{(2-t)\alpha + d}{\alpha t}\,\cA^{-\frac{(4-t)\alpha + d}{\alpha t}}$.
\end{lemma}

\begin{proof}[Proof of Lemma \ref{uA>uI_then_tail}]

If $\uA > \uI$ then by definition of $\uA$ given in \eqref{def_uA}, we have $$\uA = \left[\cA \frac{ L^\frac{d}{4\alpha + d}}{ \left(n^2 \I \right)^\frac{\alpha}{4\alpha + d}}\right]^\frac{(4-t)\alpha +d}{(2-t)\alpha +d}.$$
By Lemma \ref{uA>uI_then_=}, we have: 
\begin{equation}\label{p_0_squared}
    \uA \int_{\widetilde{\mathcal{T}}(\uI)} p_0 \geq \int_{\widetilde{\mathcal{T}}(\uI)} p_0^2 \geq \frac{\cI}{n^2\barhT^d} = \cI \left[\TL \Big(\int_{\widetilde{\mathcal{T}}(\uI)} p_0\Big)^d\right]^\frac{1}{\alpha + d}.
\end{equation}
Therefore:

\vspace{-10mm}

\begin{align*}
     \uA \; \int_{\T(\uA)} p_0 \geq \cI \left[\TL \Big(\int_{\widetilde \T(\uI)} p_0\Big)^d\right]^\frac{1}{\alpha + d}.
\end{align*}
Raising this relation to the power $\frac{(2-t)\alpha + d}{\alpha t}$ and recalling the expressions of $\rhob$ and $\overline{ \rhot}$, we get $\overline{ \rhot} \; \geq \; \CL{\ref{uA>uI_then_tail}} \, \rhob$. 
\end{proof}

\begin{lemma}\label{tail_dominates}
Whenever $\uA>\uI$ and $\max\limits_{\Omega} p_0 \geq \uI$, we have 
$$ \I \uA^{2-r} \leq \frac{\CL{\ref{tail_dominates}}}{n^2} \frac{\barhT^{{d^2/\alpha}}}{\hT^{d(\alpha+d)/\alpha}},$$
where the constant $\CL{\ref{tail_dominates}} = \cA^{2\frac{(4-2t)\alpha + d}{(2-t)\alpha + d}}\CL{\ref{uA>uI_then_tail}}^{-\frac{td}{(2-t)\alpha + d}}$ can be made arbitrarily small by taking $\cA$ small enough.
\end{lemma}

\begin{proof}[Proof of Lemma \ref{tail_dominates}]

We have $2-r = 2 \frac{(4-2t)\alpha + d}{(4-t)\alpha + d}$, so that
\begin{equation}\label{IuA}
     \large{\I \uA^{2-r} = \cA^{2\frac{(4-2t)\alpha + d}{(2-t)\alpha + d}} \TL^{\;\frac{2}{4\alpha + d}\frac{(4-2t)\alpha + d}{(2-t)\alpha + d}} \; \I^{\;\frac{d}{(2-t)\alpha+d}\frac{(4-t)\alpha + d}{4\alpha + d}}.}
\end{equation}

On the other hand, by Lemma \ref{uA>uI_then_=}:
\begin{equation}\label{2ndMoment}
    \int_{\widetilde{\mathcal{T}}(\uI)} p_0^2 \geq \frac{\cI}{n^2\barhT^d} = \cI \left[\TL \Big(\int_{\widetilde \T(\uI)} p_0\Big)^d\right]^\frac{1}{\alpha + d}.
\end{equation}

Moreover,  when $\uA > \uI$, we have by Lemma \ref{uA>uI_then_tail}, $\CL{\ref{uA>uI_then_tail}} \rhob \leq \overline{ \rhot}(\uI)$. We now raise this relation to the power $\frac{td}{(2-t)\alpha + d}$:
\begin{align}
     \large{\TL^{\;\frac{2}{4\alpha + d}\frac{(4-2t)\alpha + d}{(2-t)\alpha + d}} \; \I^{\;\frac{d}{(2-t)\alpha+d}\frac{(4-t)\alpha + d}{4\alpha + d}}} \leq \frac{\TL^\frac{1}{\alpha + d}}{\CL{\ref{uA>uI_then_tail}}^\frac{td}{(2-t)\alpha + d}}  \frac{ \left( \int_{ \T(\uA)} p_0\right)^\frac{d}{\alpha}}{\left(\int_{\widetilde \T(\uI)} p_0\right)^{\frac{d}{\alpha} \frac{d}{\alpha + d}}} = \frac{1}{\CL{\ref{uA>uI_then_tail}}^\frac{td}{(2-t)\alpha + d}} \,\frac{1}{n^2} \frac{\barhT^{{d^2/\alpha}}}{\hT^{d(\alpha+d)/\alpha}}. \label{rhob_leq_rhot}
\end{align}

Equations \eqref{IuA} and \eqref{rhob_leq_rhot} yield the result.
\end{proof}

\begin{lemma}\label{before_int_p0squared_uA_small}
Set $\Cmom = \cI + \CL{\ref{tail_dominates}}$. The constant $\Cmom$ can be made arbitrarily small by choosing successively $\cI$ and $\cA$ small enough. This can be done by taking $\cI^{\alpha+d} = \cA^{4\alpha+d}$. If $\uA > \uI$ and $\max\limits_{\Omega} p_0 \geq \uI$, then it holds: $$ \int_{\T(\uA)} p_0^2 \leq \frac{\Cmom}{n^2 \overline{\hT}^d}. $$ 
\end{lemma}

\begin{proof}[Proof of lemma \ref{before_int_p0squared_uA_small}]

Suppose that $\uA > \uI$. 
\begin{align}
    \int_{\T(\uA)} p_0^2 &= \int_{\T(\uI)} p_0^2 + \int_{\B(\uI) \cap \T(\uA)} p_0^2 \leq \frac{\cI}{n^2h_{tail}^2(\uI)} + \I \uA^{2-r} \nonumber \\
    & \leq \frac{\cI}{n^2h_{tail}^2(\uI)} +  \frac{\CL{\ref{tail_dominates}}}{n^2}\, \frac{\barhT^{{d^2/\alpha}}}{\hT^{d(\alpha+d)/\alpha}(\uI)} \text{ by Lemma \ref{tail_dominates}} \\
    & \leq \Cmom \frac{1}{n^2 \barhT^d}.\label{Finlemme}
\end{align}

\end{proof}

\begin{lemma}\label{int_p0squared_uA_small}
   Regardless of whether $\uI > \max\limits_{\Omega} p_0$ or not, and regardless of whether $\uI=\uA$ or not, it always holds :
$$\int_{\T(\uA)} p_0^2 \leq \frac{\Cmom}{n^2\hT^d(\uA)}.$$
\end{lemma}

\begin{proof}[Proof of Lemma \ref{int_p0squared_uA_small}]
If $\uI = \uA$ then $\hT(\uI) = \hT(\uA)$. Moreover, by definition of $\uI$, we have: $\int_{\B(\uI)}p_0^2 \leq \frac{\cI}{n^2h^d}$ so the result holds by recalling $\Cmom = \cI + \CL{\ref{tail_dominates}}$. Now If $\uA > \uI > \max\limits_{\Omega} p_0$, then $\hT(\uA) \leq \barhT(\uI)$, so the result holds as well. Finally, if $\uI > \max\limits_{\Omega} p_0$, then $\T(\uI) = \T(\uA)$ so the result holds as well by Lemma \ref{uA>uI_then_=} item \ref{T(uI)_leq_cI}.
\end{proof}


We now show that $\hT(\uA) \asymp \barhT(\uI) $ when $\uA > \uI$.

\begin{lemma}\label{htail_all_equal}
If $\uA > \uI$ then $\CL{\ref{htail_all_equal}}\hT(\uA) \geq  \barhT(\uI) $ where $\hT(\uA), \; \barhT(\uI)$ are defined in \eqref{def_htail_uA}, \eqref{def_barhtail_uI} and $\CL{\ref{htail_all_equal}}$ is a constant. Hence it always holds $\hT(\uA) \asymp \barhT(\uI)$.
\end{lemma}

\begin{proof}[Proof of Lemma \ref{htail_all_equal}]
Suppose that $\uA > \uI$. Then by Lemma \ref{before_int_p0squared_uA_small}, we have $\int_{\T(\uA)} p_0^2 \leq \Cmom \frac{1}{n^2 \barhT^d}$. Moreover, by definition of $\uI$, since $\uA > \uI$ we can write: $\int_{\T(\uA)} p_0^2 \geq \cI \frac{1}{n^2 \hT(\uA)^d}$, hence: $(1+\frac{\CL{\ref{tail_dominates}}}{\cI})\hT(\uA) \geq \barhT(\uI)$. Moreover, it directly follows from the definition of $\hT(\uA)$ and $\barhT(\uI)$ that $\hT(\uA) \leq \barhT(\uI)$. Hence $\hT(\uA) \asymp \barhT(\uI)$.
\end{proof}

\begin{lemma}\label{Bulk_asymp_BulkCases}
Assume $\CBT\rhob \geq \rhot$. There exists a constant $\CL{\ref{Bulk_asymp_BulkCases}}>1$ depending only on $\CBT$, $\cI$ and $\cA$, such that: $\int_{\B(\frac{\uA}{2})}p_0^r \leq \CL{\ref{Bulk_asymp_BulkCases}}\, \I$. 
\end{lemma}

\begin{proof}[Proof of Lemma \ref{Bulk_asymp_BulkCases}]
\begin{align}
    \int_{\B(\frac{\uA}{2})\setminus \B(\uA)} p_0^r &\leq \uA^r \big|\B\big(\frac{\uA}{2}\big)\setminus \B(\uA) \big| \leq \uA^{r-2}  \int_{\B(\frac{\uA}{2})\setminus \B(\uA)} 4 p_0^2\nonumber \\
    & \leq 4 \uA^{r-2} \frac{\Cmom}{n^2 \hT(\uA)^d} ~~~~ \text{ by Lemma \ref{int_p0squared_uA_small}.} \label{majorant_BuA_sur_2}
\end{align}
Moreover, the condition $\CBT\rhob \geq \rhot$ exactly rewrites 
\begin{align*}
    \mbox{\large\(\frac{\uA^{r-2}}{n^2 \hT(\uA)^d} \leq \CBT^{\frac{td}{(2-t)\alpha + d} }\cA^{-2 \frac{(4-2t)\alpha + d}{(4-t)\alpha + d}} \cdot \I,\)}
\end{align*}
so that \eqref{majorant_BuA_sur_2} gives:
\begin{align*}
    \int_{\B(\frac{\uA}{2})\setminus \B(\uA)}\mbox{\large\( p_0^r \leq 4\Cmom \CBT^{\frac{td}{(2-t)\alpha + d} }\cA^{-2 \frac{(4-2t)\alpha + d}{(4-t)\alpha + d}} \cdot \I\)}
\end{align*}
so that $$ \int_{\B(\frac{\uA}{2})}\mbox{\large\( p_0^r \leq \Big(1+4\Cmom \CBT^{\frac{td}{(2-t)\alpha + d} }\cA^{-2 \frac{(4-2t)\alpha + d}{(4-t)\alpha + d}} \Big)\cdot \I=: \CL{\ref{Bulk_asymp_BulkCases}} \I.\)}$$
\end{proof}


\begin{lemma}\label{T(2uA)Moment1}
If the tail dominates, i.e. if $\rhot \geq \CBT \rhob$ where $\CBT$ is a large constant, then there exists a small constant $\CL{\ref{T(2uA)Moment1}}$ depending only on $\CBT$ and decreasing with respect to $\CBT$, such that $$\int_{\T(2\uA)}p_0 \leq (1+ \CL{\ref{T(2uA)Moment1}}) \int_{\T(\uA)}p_0.$$
\end{lemma}

\begin{proof}[Proof of Lemma \ref{T(2uA)Moment1}]
We show that $\int_{\T(2\uA)}p_0 - \int_{\T(\uA)}p_0 \leq \CL{\ref{T(2uA)Moment1}} \int_{\T(\uA)}p_0$. Note that $r$ can be greater or smaller than $1$, so that on $\T(2\uA)\setminus \T(\uA)$ we have: $p_0 \leq (2^{1-r}\lor 1)\uA^{1-r} p_0^{r}$. We therefore have
\begin{align*}
    &(2^{r-1}\land 1)\int_{\T(2\uA)\setminus \T(\uA)} p_0 \leq \uA^{1-r}\int_{\T(2\uA)\setminus \T(\uA)} p_0^r \leq \uA^{1-r} \I \\
    & = \I^{\frac{(\alpha+d)((4-t)\alpha + d)}{(4\alpha + d)((2-t)\alpha+d)}} \frac{L^{\frac{d}{4\alpha+d} \frac{(4-3t)\alpha+d}{(2-t)\alpha+d}}}{n^{\frac{2\alpha}{4\alpha+d} \frac{(4-3t)\alpha+d}{(2-t)\alpha+d}}} \leq \CBT^{-\frac{\alpha+d}{(2-t)\alpha+d}} \int_{\T(\uA)}p_0,
\end{align*}
where the last inequality is obtained by using the assumption $\rhot \geq \CBT \rhob$ and the expressions of $\rhob$ and $\rhot$.
\end{proof}

\begin{lemma}\label{T(2uA)Moment2}
If the tail dominates, i.e. if $\rhot \geq \CBT \rhob$ where $\CBT$ is a large enough constant, then there exists a small constant $\CL{\ref{T(2uA)Moment2}}$ depending only on $\CBT$ and decreasing with respect to $\CBT$, such that $$\int_{\T(2\uA)}p_0^2 \leq \frac{\CL{\ref{T(2uA)Moment2}}}{n^2\hT^d}.$$
\end{lemma}

\begin{proof}[Proof of Lemma \ref{T(2uA)Moment2}]
We have, recalling equations \eqref{IuA}, \eqref{rhob_leq_rhot}, \eqref{Finlemme}:
\begin{align*}
    2^{r-2}\int_{T(2\uA) \setminus \T(\uA)} p_0^2 \leq \uA^{2-r} \int_{\B(\uA)} p_0^r \leq \frac{\Cmom}{n^2\barhT^d}.
\end{align*}
Moreover, $\frac{\Cmom}{n^2\barhT^d} \leq \frac{\Cmom}{n^2\hT^d(2\uA)}$ and $\hT(2\uA) \geq (1+\CL{\ref{T(2uA)Moment1}})^{-\frac{1}{\alpha+d}} \hT(\uA)$ by Lemma \ref{T(2uA)Moment1}. Now:
\begin{align*}
    \int_{\T(2\uA)}p_0^2 \leq \frac{\Cmom}{n^2\hT^d} + \frac{\Cmom 2^{2-r}(1+\CL{\ref{T(2uA)Moment1}})^{\frac{d}{\alpha+d}}}{n^2\hT^d}.
\end{align*}

which yields the result. 
\end{proof}

\begin{lemma}\label{htail_leq_hbulk}
In the case $\int_{\T} p_0 \geq  \frac{\ctail}{n}$, there exists constants $\CBT, \CBT^{(2)}$ such that we have $$ \rhot \geq \CBT \rhob \Longleftrightarrow \hm  \geq  \CBT^{(2)} \hT(\uA),$$
where $\hm$ is defined in Equation \eqref{def_hm}. In particular we have 
$$\rhot \geq \CBT \rhob \Longrightarrow \inf_{x \in \B} \hB(x) \geq  \CBT^{(2)} \hT(\uA),$$ 
where $\CBT^{(2)}$ can be made arbitrarily large by choosing $\CBT$ large enough.
\end{lemma}

\begin{proof}[Proof of Lemma \ref{htail_leq_hbulk}]
The result can be proved by direct calculation, recalling the expression of $\hT(\uA)$ from \eqref{def_htail_uA}, the expressions of $\rhob$ and $\rhot$ from \eqref{def_rho_bulk_and_tail} and that $\inf_{x \in \B} \hB(x) \geq \hm$.
\end{proof}
\section{Partitioning algorithm}\label{Algo_appendix}

We now introduce the recursive partitioning scheme, inspired from \cite{balakrishnan2019hypothesis}. 
For any cube $A \subset \Omega$, denote by $e(A)$ its edge length. 
For any function $h: \Omega \rightarrow \R_+$, denoting by $x_A$ the center of $A$, define $h(A) = h(x_A)$. 
The partitionning algorithm takes as input a cubic domain $\widetilde \Omega \subset \Omega$, a parameter $\beta \geq \alpha$, a value $u>0$ and a constant $\cbeta >0$. 
Defining the bandwidth function $h : \widetilde \Omega \to \R_+$ such that $p_0(x) = \cbeta \, h^\beta(x)$ over $\widetilde \Omega$, as well as the set $\D(u) = \{x \in \widetilde \Omega: p_0(x) \geq u\}$ 
to be split into cubes, the algorithm returns a family $P = \{A_1, \dots, A_N\}$ of \textit{disjoints} cubes of $\widetilde \Omega$ \textit{covering} $\D(u)$ (i.e. such that $\D(u) \subset \cup_{j=1}^N A_j$), and such that, for all $j=1, \dots, N$, $h(A_j) \geq e(A_j) \geq \frac{1}{2^{\beta+1}} h(A_j)$. 
Note that the center of $A_j$ need not belong to $\D(u)$. Algorithm \ref{Algo_recursif} corresponds to an auxiliary algorithm called by the actual partitioning algorithm defined in Algorithm \ref{Algo_partitionnement}.\\

\begin{algorithm}[H]\label{Algo_recursif}
\SetAlgoLined
\begin{enumerate}
    \item \textbf{Input:} $A,h, D,P$. 
    \item  \begin{itemize}
            \item \textbf{If} $A \cap D = \emptyset$: \textbf{return} $P$.
            \item  \textbf{If} $e(A) \leq h(A)$: \textbf{return} $P \cup \{A\}$.
            \item\textbf{Else:}
            \begin{enumerate}
                \item Split $A$ into $2^d$ cubes $A_1, \dots, A_{2^d}$ obtained by halving $A$ along each of its axes.
                \item \textbf{return} $\bigcup_{i=1}^{2^d} \text{Algorithm \ref{Algo_recursif}}(A_i,h,D,P)$.
            \end{enumerate}
        \end{itemize}
\end{enumerate}
\caption{Recursive auxiliary algorithm}
\end{algorithm}

\hfill

\begin{algorithm}[H]\label{Algo_partitionnement}
\SetAlgoLined
\begin{enumerate}
    \item \textbf{Input:} $\widetilde \Omega, \beta, u, \cbeta$. 
    \item \textbf{Initialization:} $~P = \emptyset$, $~\D(u) = \{x \in \widetilde \Omega: p_0(x) \geq u\}$, $~h=\left(\frac{p_0}{\cbeta}\right)^\frac{1}{\beta}$.
    \item \textbf{Return} $\text{Algorithm \ref{Algo_recursif}}\big(\widetilde \Omega, h, \D(u), P\big)$.
\end{enumerate}
\caption{Adaptive partition}
\end{algorithm}

\hfill

We have the following guarantees for Algorithm \ref{Algo_partitionnement}.
\begin{proposition}\label{guarantees_algo}
Algorithm \ref{Algo_partitionnement} terminates. Assume moreover that Algorithm \ref{Algo_partitionnement} splits the domain at least once and that there exists a constant $\calpha > 0$ such that $\cstar + \frac{\sqrt{d}^\alpha}{\calpha} (2^{1-\alpha} \lor 1) \leq 1/2$ and such that $\forall x \in \D(u): p_0(x) \geq \calpha L h^\alpha(x)$. Then:
\begin{enumerate}
    \item Denoting by $P$ the output of Algorithm \ref{Algo_partitionnement} with inputs  $\widetilde \Omega, \beta, u, \cbeta$, it holds: $\D(u) \subset \cup_{A \in P} A$;
    \item\label{cases_h} For all cube $A \in P$, it holds: $h(A) \geq e(A) \geq \frac{1}{2^{\beta+1}} h(A)$. 
    \item\label{min_geq_max} For all cell $A\in P$ we have $\min_{A} p_0 \geq \frac{1}{2}\max_{A}p_0$. Consequently, it holds $\cup_{A \in P}\, A \subset \D(\frac{u}{2})$.
\end{enumerate}
\end{proposition}

\begin{proof}[Proof of Proposition \ref{guarantees_algo}]
Fix a cube $\widetilde \Omega \subset \Omega$, $\beta > 0 $, $u > 0$ $\cbeta>0$.\\

\vspace{-4mm}

\begin{enumerate}
\item[] \underline{\textbf{Termination:}}
Suppose that Algorithm \ref{Algo_partitionnement} does not terminate. Then, among the cubes defined at some step by the algorithm, there would exist an infinite sequence $(A_l)_{l \in \N}$ of nested cubes of $\widetilde \Omega$ satisfying:
\begin{enumerate}
    \item[(i)] $A_0 = \widetilde \Omega$,
    \item[(ii)] $\forall l \in \N$: $A_{l+1} \subset A_l$,
    \item[(iii)] $\forall l \in \N: e(A_{l+1}) = \frac{1}{2} e(A_l)$,
    \item[(iv)] $\forall l \in \N: A_l \cap \D(u) \neq \emptyset$.
\end{enumerate}
Denote by $x_l$ the center of $A_l$ for all $l \in \N$. Then by (ii) and (iii), $(x_l)_l$ is a Cauchy sequence of $[0,1]^d$ and thus converges to some $x_\infty \in \Omega$. Moreover, denoting by $d_{\|\cdot\|_2}(x,\D(u))$ the Euclidean distance of $x$ to $\D(u)$, we have: $d_{\|\cdot\|_2}(x_l,\D(u)) \leq e(A_l) \frac{\sqrt{d}}{2} = 2^{-l} e(\widetilde \Omega) \frac{\sqrt{d}}{2} \to 0$. Since $\D(u)$ is closed by continuity of $p_0$, it holds $p_0(x_\infty) \geq u >0$. However, at each step, $A_l$ is split, imposing $e(A_l) \geq h(x_l) \to 0$, hence $h(x_\infty) = 0$, yielding $p_0(x_\infty) = 0$ since $p_0 = \cbeta h^\beta$ over $\widetilde \Omega$. This leads to a contradiction.

    \item It is straightforward to check that when the algorithm terminates, $\D(u) \subset \cup_{A \in P} A$.
    
    \item Let $A \in P$. Denote by $A'$ the parent of $A$ in the hierarchical splitting performed by Algorithm \ref{Algo_recursif}. Since by assumption the domain is split at least once, $A'$ exists. Since $A'$ was split and $A$ was kept we necessarily have: $2h(A) \geq 2 e(A) = e(A') > h(A')$. Denote by $x_A$ and $x_{A'}$ the respective centers of $A$ and $A'$. Since by definition of $A$, $x_{A'}$ is a vertex of $A$, we have $\|x_A - x_{A'}\| = e(A)\frac{\sqrt{d}}{2} \leq h(A)\frac{\sqrt{d}}{2}$. By Assumption \eqref{simplifyingAssumption}:
    \begin{align}
        |p_0(x_A) - p_0(x_{A'})| &\leq \cstar p_0(x_A) + L\|x_A - x_{A'}\|^\alpha = \cstar p_0(x_A) + L h(A)^\alpha \left(\frac{\sqrt{d}}{2} \right)^\alpha. \label{Control_A'_A}
    \end{align}
    There are two cases:
    \begin{itemize}
        \item If $x_A \in \D(u)$ then $|p_0(x_A) - p_0(x_{A'})| \leq \cstar p_0(x_A) + \frac{1}{\calpha} p_0(x_A) \left(\frac{\sqrt{d}}{2} \right)^\alpha \leq p_0(x_A) / 2$,  hence $p_0(x_A') \geq p_0(x_A)/2$.
        \item Otherwise, $x_A \notin \D(u)$. Let $x_u \in A \cap \D(u)$. Since $A$ was kept by Algorithm  \ref{Algo_partitionnement}, $x_u$ exists. By the definition of $D(u)$, it holds $p_0(x_A) < p_0(x_u)$, hence $h(x_A) < h(x_u)$. We have $\|x_A - x_u\| \leq e(A) \frac{\sqrt{d}}{2} \leq h(A) \frac{\sqrt{d}}{2} \leq h(x_u) \frac{\sqrt{d}}{2}$. By Assumption \eqref{simplifyingAssumption}: $|p_0(x_A) - p_0(x_u)| \leq \cstar p_0(x_u) + L h(x_u)^\alpha \big(\frac{\sqrt{d}}{2} \big)^\alpha \leq \frac{p_0(x_u)}{2}$, hence $p_0(x_A) \geq \frac{p_0(x_u)}{2}$. Therefore, it holds:
            \begin{align*}
                L h^\alpha(x_A) \leq L h^\alpha(x_u) \leq \frac{1}{\calpha} p_0(x_u) \leq \frac{2}{\calpha} p_0(x_A).
            \end{align*}
        Injecting this relation into \eqref{Control_A'_A}, we get: $|p_0(x_A) - p_0(x_{A'})| \leq \frac{1}{2} p_0(x_A)$ hence $p_0(x_{A'}) \geq p_0(x_A) /2$.
    \end{itemize}
    In both cases, it holds $p_0(x_{A'}) \geq p_0(x_A) /2$, hence $h(x_{A'}) \geq h(x_A) / 2^{1/\beta}$. We therefore get $h(A) \geq e(A) \geq \frac{h(A')}{2} \geq h(A) / 2^{1+1/\beta}$.
    \item Let $A\in P$ and $x= \arg\max_{A} p_0$, $y= \arg\min_{A} p_0$. We have by Assumption \eqref{simplifyingAssumption}:
    \begin{align*}
        |p_0(x)-p_0(y)| \leq \cstar p_0(x) + L(e(A)\sqrt{d})^\alpha \leq \Big(\cstar + \frac{\sqrt{d}^\alpha}{\calpha} \Big) p_0(x) \leq \frac{p_0(x)}{2}.
    \end{align*}
\end{enumerate}
\end{proof}
\section{Upper bound in the bulk regime}\label{UB_Bulk_appendix}

\subsection{Technical lemmas in the bulk regime}

Throughout the Appendix, we will denote by $B(x,h)$ the Euclidean ball of $\R^d$ centered at $x$ and of radius $h$. When no ambiguity arises, $\| \cdot \|$ will denote the Euclidean norm over $\R^d$. \textbf{In Appendix \ref{UB_Bulk_appendix} only,} we will denote by $h(x)$ the quantity $\cch \hB(x)$ where the constant $\cch$ can be chosen arbitrarily small. We first prove the following result, stating that for any $p$ satisfying Assumption \eqref{simplifyingAssumption}, $p$ can be considered as approximately constant over the balls $B(x, h(x))$ for all $x \in \TB$ where $\TB = \B(\frac{\uA}{2})$ if the bulk dominates and $\TB = \B$ if the tail dominates. 

\begin{lemma}\label{const_balls} Recall that $h(x) = \cch \hB(x) = (\tcA/4)^{1/\alpha} \hB(x)$. Let $x \in \TB$ and $y \in B(x,h(x))$. For all $p$ satisfying Assumption \eqref{simplifyingAssumption}, we have:
$$ \frac{p(y)}{p(x)} \in [\cp, \Cp] ~~~~ \text{ where } ~~~~ \Cp, \cp = 1 \pm \left(\cstar +\frac{\cch^\alpha(1+\delta)}{\tcA}\right).$$ It follows that:
\vspace{-10mm}

\begin{align*}
    \frac{\omega(y)}{ \omega(x)}  \in [\comega, \Comega] ~~~~~~ \text{ and } ~~~~~~  \frac{\hB(y)}{\hB(x)}  \in [\ch, \Ch],
\end{align*}
where $\Ch = {\Cp}^\frac{2}{(4-t) \alpha + d}$, $\ch = {\cp}^\frac{2}{(4-t) \alpha + d}$, $\Comega = {\Cp}^{\frac{2\alpha t - 4 \alpha}{(4-t) \alpha + d}}$ and $\comega = {\Cp}^{\frac{2\alpha t - 4 \alpha}{(4-t) \alpha + d}}$.
\end{lemma}

\hfill

\begin{proof}
Let $x \in \TB$ and $y \in B(x,h(x))$ and assume that $p$ satisfies Assumption \eqref{simplifyingAssumption} with the constants $\cstar'$ and $L'$. We have 
\begin{align*}
    |p(x) - p(y)| &\leq \cstar' p(x) + L'\|x-y\|^\alpha \leq \cstar' p(x) + L' \, \cch^\alpha \, \hB^\alpha(x)\\
    &\leq \cstar' p(x) + \frac{(1+\delta)\cch^\alpha}{\tcA} p(x) ~~~~ \text{ since on } \TB: \tcA L \hB(x)^\alpha \leq p(x).
\end{align*}

We then take $\cch$ small enough to guarantee that $\frac{(1+\delta)\cch^\alpha}{\tcA} \leq \frac{\cstar}{2}$. Therefore, $\Cp\, p_0(x) \geq p_0(y) \geq \cp\, p_0(x)$. The analogous relations for $\hB$ and $\omega$ directly follow from their definitions \eqref{def_hbulk}, \eqref{def_omega}.
\end{proof}

\hfill

In the sequel, we define the following notation:
\vspace{-3mm}
\begin{equation}\label{def_TL}
    \TL = \frac{L^d}{n^{2\alpha}},
\end{equation}
and for all $u\geq 0$:
\begin{equation}\label{def_I_u}
    \I(u) := \int_{\B(u)} p_0^{\frac{2\alpha t}{(4-t) \alpha +d}}.
\end{equation}
so that $\I = \I(\uI)$.

\begin{lemma}\label{lemma_uA}
It holds that
\vspace{-9mm}
\begin{equation}\label{\uA}
    \tuA \geq \left[\tcA \frac{\TL}{\I^\alpha}\right]^{\frac{1}{4\alpha+d}\frac{(4-t)\alpha + d}{(2-t)\alpha + d}}.
\end{equation}
\end{lemma}

The proof follows directly from the definition of $\uA$ in \eqref{def_uA} and \eqref{def_tuA}.

\begin{lemma}\label{control_xy_far}
We have: $\frac{1}{n}\int_{\TB} \omega(x) p_0(x)^2 dx\leq t_n$.
\end{lemma}

\begin{proof}[Proof of Lemma \ref{control_xy_far}]
We have by the Cauchy-Schwarz inequality:
\begin{align*}
    \frac{1}{n}\int_\TB \omega p_0^2 &= \frac{1}{n} \int_\TB p_0 \, \hB^{d/2}\, p_0^{r/2} \times \big(n^2 L^4 \I \big)^{\frac{d/2}{4\alpha + d}} \leq \frac{L^\frac{2d}{4\alpha + d}\I^\frac{d/2}{4\alpha + d}}{n^\frac{4\alpha}{4\alpha + d}}\Big( \int_\TB p_0^2 \hB^d \int_\TB p_0^r \Big)^{1/2} \\
    &\leq \frac{\sqrt{\CL{\ref{Bulk_asymp_BulkCases}}}}{\Ctn} t_n \Big( \int_\TB p_0^2 \hB^d \Big)^{1/2}.
\end{align*}

Moreover by Lemma \ref{const_balls},
\begin{align*}
    p_0(x) \hB^d(x) &= p_0(x) \frac{h(x)^d}{\cch^d} \leq \frac{1}{\cch^d \cp}\int_{B(x,h(x))}p_0 \leq \frac{1}{\cp \cch},
\end{align*}
so that
\begin{align*}
    \int_\TB p_0^2 \hB^d \leq \frac{1}{\cp\cch} \int_\TB p_0 \leq \frac{1}{\cp \cch}.
\end{align*}
Taking $\Ctn \geq \sqrt{\frac{\CL{\ref{Bulk_asymp_BulkCases}}}{\cp \cch}}$ yields the result.
\end{proof}

\begin{lemma}\label{omega_h_2alpha} 
It holds that: $ \int_{\TB} L^2 \omega(x) \hB^{2 \alpha}(x) dx \leq t_n$.
\end{lemma}

\begin{proof}[Proof of Lemma \ref{omega_h_2alpha}]
We have:
$$\int_{\TB} L^2 \omega(x) \hB^{2 \alpha}(x) dx = \frac{L^\frac{2d}{4\alpha + d}\I(\tuA)}{n^\frac{4\alpha}{4 \alpha +d} \I(\uI)^{\frac{2\alpha}{4 \alpha +d}}} \leq \frac{\CL{\ref{Bulk_asymp_BulkCases}}}{\Ctn}t_n,$$
recalling the expressions of $\omega(x)$ from \eqref{def_omega}, $\hB(x)$ from \eqref{def_hbulk}, $t_n$ from \eqref{def_test_bulk}, and using at the last step (if $\tuA = \uA$) $\I(\uA) \leq \I(\uI)$ since $\uA \geq \uI$. Taking $\Ctn \geq \CL{\ref{Bulk_asymp_BulkCases}}$ yields the result.
\end{proof}

\hfill

In the remaining of the analysis of the upper bound in the bulk regime, we fix a family $X_1, \dots, X_n$ of iid random variables with density either $p_0 \in \mathcal P(\alpha, L,\cstar)$ or $p\in \mathcal{P}(\alpha,L',\cstar')$. In the whole analysis of the upper bound, we will use the following notation:
\begin{equation}\label{def_Delta}
    \Delta(x) := p(x) - p_0(x) ~~~~~~~~\hat \Delta(x) := \hat p(x) - p_0(x) ~~~~ \text{ and } ~~~~ \hat \Delta'(x) := \hat p'(x) - p_0(x).
\end{equation} 
We also define 
\begin{equation}\label{def_J}
    J := \int_\TB \omega(x) \Delta(x)^2 dx.
\end{equation}
We will denote by $\CK'$ the constant $(1+\delta)\CK$ and by $\CKprime$ the constant $\int_{\R^d} K^2$.

\begin{lemma}\label{int_omega_p2}
For $p \in \mathcal{P}(\alpha,L',\cstar')$, It holds that $\frac{1}{n} \int_\TB \omega p^2 \leq \AL{\ref{int_omega_p2}} t_n + \frac{\BL{\ref{int_omega_p2}}}{n}J$, where $\AL{\ref{int_omega_p2}}$ and $\BL{\ref{int_omega_p2}}$ are two constants.
\end{lemma}

\begin{proof}[Proof of Lemma \ref{int_omega_p2}]
Using $(a+b)^2 \leq 2a^2 + 2b^2$ and the triangle inequality, we get:
\begin{align*}
    \frac{1}{n} \int_\TB \omega p^2 \leq \frac{1}{n} \int_\TB \omega [2p_0^2 + 2 \Delta^2] \leq 2 t_n + \frac{2}{n}J =: \AL{\ref{int_omega_p2}} t_n + \frac{\BL{\ref{int_omega_p2}}}{n}J ~~ \text{ by Lemma \ref{control_xy_far}.}
\end{align*}
\end{proof}

\begin{lemma}\label{Expectation_Tb_H1}
If $J \geq t_n$ then we have:
$\mathbb{E}\Tb \geq \big(\sqrt{J} - \sqrt{t_n} \big)^2.$
\end{lemma}

\begin{proof}[Proof of Lemma \ref{Expectation_Tb_H1}]
By the Minkowski inequality:
\begin{align*}
    &J= \int_\TB \omega(x) \Delta(x)^2 dx = \int_\TB \omega(x) \big[\Delta(x) + \mathbb{E}[\hat p(x) - p(x)] - \mathbb{E}[\hat p(x) - p(x)]\big]^2 dx \\
    & \leq \Big[\Big(\int_\TB \omega(x) \mathbb{E}^2[ \hat \Delta(x)] dx \Big)^{1/2} + \Big(\int_\TB \omega(x) \mathbb{E}^2[\hat p(x) - p(x)] dx \Big)^{1/2} \Big]^2\\
    &\leq \Big[\sqrt{\E \Tb} + \CK'\sqrt{\frac{t_n}{\Ctn}} \Big]^2 \leq \Big[\sqrt{\E \Tb} + \sqrt{t_n} \Big]^2 \text{ by choosing $\Ctn \geq {\CK'}^2$}.
\end{align*}
At the last step we used $|\E(\hat p(x) - p(x))| \leq \CK L' h^\alpha(x) = \CK' L h^\alpha(x)$, by \cite{tsybakov2008introduction} Proposition 1.2. Moreover, we used Lemma \ref{omega_h_2alpha}. This yields the result, since $J\geq t_n$.
\end{proof}

\begin{lemma}\label{splitting_variance}
We have: $\V(\Tb) \leq \Big[ \big(\sqrt{J} + \sqrt{t_n} \big)^2 + J_2^{1/2} \Big]^2 - \E^2 \Tb$ where 
\begin{equation}\label{def_J2}
    J_2 = \iint_{\TB^2} \omega(x) \omega(y) \frac{1}{k^2} \text{cov}^2\left(K_{h(x)}(x-X),K_{h(y)}(y-X)\right) dx dy.
\end{equation}
\end{lemma}

\begin{proof}[Proof of Lemma \ref{splitting_variance}]
\begin{align}
    \V(\Tb) & = \E(\Tb^2) - \E^2 \Tb = \E \iint_{\TB^2} \omega(x) \omega(y) \; \hat \Delta(x) \hat \Delta(y) \; \hat \Delta'(x) \hat \Delta'(y) \; dx dy - \E^2 \Tb \nonumber \\
    & = \iint_{\TB^2} \omega(x) \omega(y) \E\left[\hat \Delta(x) \hat \Delta(y)\right]^2 dx dy -\E^2 \Tb \label{var_bulk_H0}.
\end{align}

Recall that throughout Appendix \ref{UB_Bulk_appendix}, $h(x) = \cch\hB(x)$ where $\cch = (\tcA/4)^\frac{1}{\alpha}$. We now compute the term $\E\left[\hat \Delta(x) \hat \Delta(y)\right]^2$. We have:

\begin{align}
    \E\left[\hat \Delta(x) \hat \Delta(y)\right] & = \E\left\{\left[ \frac{1}{k} \sum_{i=1}^k \left(K_{h(x)}(x-X_i) - p_0(x)\right) \right]\left[ \frac{1}{k} \sum_{i=1}^k \left(K_{h(y)}(y-X_i) - p_0(y)\right) \right]\right\}\nonumber\\
    &\nonumber\\
    & = \frac{1}{k^2} \sum_{i=1}^k \E\left\{\left[K_{h(x)}(x-X_i) - p_0(x)\right]\left[K_{h(y)}(y-X_i) - p_0(y)\right] \right\}\nonumber\\
    & \;\;\;\;+ \frac{1}{k^2} \sum_{i \neq j} \E\left[K_{h(x)}(x-X_i) - p_0(x)\right]\E\left[K_{h(y)}(y-X_j) - p_0(y)\right]\nonumber\\
    &\nonumber\\
    & = \frac{1}{k} \;  \E\left\{\left[K_{h(x)}(x-X) - p_0(x)\right]\left[K_{h(y)}(y-X) - p_0(y)\right] \right\}\nonumber\\
    & \;\;\;\;+ \frac{k-1}{k}  \E\left[K_{h(x)}(x-X) - p_0(x)\right]\E\left[K_{h(y)}(y-X) - p_0(y)\right]\nonumber\\
    &\nonumber\\
    &= \E\left[K_{h(x)}(x-X) - p_0(x)\right]\E\left[K_{h(y)}(y-X) - p_0(y)\right]\nonumber\\
    & \;\;\;+ \frac{1}{k} \text{cov}\left(K_{h(x)}(x-X),K_{h(y)}(y-X)\right),\label{exp_cross_term}
\end{align}
so that, by the Minkowski inequality:

\begin{align}\label{exp_delta2_delta2}
    &\iint_{\TB^2} \omega(x) \omega(y) \E\left[\hat \Delta(x) \hat \Delta(y)\right]^2 dx dy  \leq  (J_1^{1/2} + J_2^{1/2})^2,
\end{align}

where 
\begin{align}
    J_1^{1/2} &=  \int_{\TB} \omega(x) \E^2\left[K_{h(x)}(x-X) - p_0(x)\right]dx, \label{def_J1}\\
    J_2 & = \iint_{\TB^2} \omega(x) \omega(y) \frac{1}{k^2} \text{cov}^2\left(K_{h(x)}(x-X),K_{h(y)}(y-X)\right) dx dy.\label{def_J22}
\end{align}

Moreover, by triangular inequality and by \cite{tsybakov2008introduction}, Proposition 1.1:
\begin{align*}
    \left|\E\left[K_{h(x)}(x-X) - p_0(x)\right]\right| \; \leq \; |p(x)-p_0(x)| + \big|\E\left(\hat p(x) - p(x)\right)\big| \; \leq \; |\Delta(x)| + \CK' L h(x)^\alpha,
\end{align*}
Therefore, still by the Minkowski inequality:
\begin{align}
    J_1^{1/2} &\leq \int_\TB \omega(x) \big[\; |\Delta(x)| + \CK' L h(x)^\alpha \big]^2 dx\nonumber\\
    & \leq \left[\Big(\int_\TB \omega(x) |\Delta(x)|^2 dx  \Big)^{1/2} + \Big(\int_\TB \omega(x) {\CK'}^2 L^2 h(x)^{2\alpha} dx \Big)^{1/2} \right]^2\nonumber\\
    &\leq \big(\sqrt{J} + \sqrt{t_n} \big)^2 ~~ \text{ by Lemma \ref{omega_h_2alpha}, the definition of $J$ and using $\cch \leq 1$}. \label{truc}
\end{align}
Equations \eqref{var_bulk_H0}, \eqref{exp_delta2_delta2} and \eqref{truc} yield the result.
\end{proof}

\begin{lemma}\label{cov_far}
Let $X$ be a random variable with density $p \in \mathcal P(\alpha, L',\cstar')$. If $~\|x-y\| > \frac{1}{2} [h(x) + h(y)]$, then 
$$\left| \textnormal{cov} \left(K_{h(x)}(x-X),K_{h(y)}(y-X)\right) \right| \leq \left(\CK' L h(x)^\alpha + p(x)\right) \left(\CK' L h(y)^\alpha + p(y)\right).$$
\end{lemma}

\begin{proof}[Proof of Lemma \ref{cov_far}]
We recall that by definition $K$ has bounded support $B(0,\frac{1}{2})$. In this case: $\text{Supp}K_{h(x)}(x-\cdot) \cap \text{Supp}K_{h(y)}(y-\cdot) = \emptyset$. Therefore, $K_{h(x)}(x-X)K_{h(y)}(y-X) = 0$ almost surely. Then by \cite{tsybakov2008introduction}, Proposition 1.1:
\begin{align*}
    \left|\text{cov}\left(K_{h(x)}(x-X),K_{h(y)}(y-X)\right) \right|& =  \left| \E\left[K_{h(x)}(x-X)\right]\E\left[K_{h(y)}(y-X)\right]\right| \nonumber \\
    &\leq \left(\CK' L h(x)^\alpha + p(x)\right) \left(\CK' L h(y)^\alpha + p(y)\right).
\end{align*}
\end{proof}

\begin{lemma}\label{cov_close}
Let $X$ be a random variable with density $p \in \mathcal P(\alpha, L',\cstar')$. If $\|x-y\| \leq \frac{1}{2} [h(x) + h(y)] \leq h(x) \lor h(y)$ then 
\begin{align}
    & \left|\text{cov}\left(K_{h(x)}(x-X),K_{h(y)}(y-X)\right)\right| \leq \CL{\ref{cov_close}} \frac{p(x)}{h^d(x)}. 
\end{align}
where $\CL{\ref{cov_close}}$ is a constant.
\end{lemma}

\hfill

\begin{proof}[Proof of Lemma \ref{cov_close}]
If $\|x-y\| \leq \frac{1}{2} [h(x) + h(y)] \leq h(x) \lor h(y)$, we suppose by symmetry $h(y) = h(x) \lor h(y)$. Then
\begin{align*}
    \left|\text{cov}\left(K_{h(x)}(x-X),K_{h(y)}(y-X)\right)\right| & \leq \sqrt{\V\left(K_{h(x)}(x-X)\right)\V\left(K_{h(y)}(y-X)\right)}.
\end{align*}

Now, by \cite{tsybakov2008introduction}, Proposition 1.1, the variance of the Kernel estimator $K_{h(x)}(x-X)$ is upper bounded as:
\begin{align}
    \V\left(K_{h(x)}(x-X)\right) &\leq \frac{1}{h^d(x)} \Big[\sup_{B(x,h(x))} p\Big] \int_{\R^d} K^2 \leq \frac{1}{h^d(x)} \Cp \CKprime \; p(x) .
    \label{variance_kernel_H1}
\end{align}

In the last inequality, we used Lemma \ref{const_balls}. Hence, since $h(y) \geq h(x)$ and $\|x-y\| \leq h(y)$, we have $x \in B(y,h(x))$ so that $p(y) \leq \frac{1}{\cp} p(x)$. Thus:
\begin{align}
    \left|\text{cov}\left(K_{h(x)}(x-X),K_{h(y)}(y-X)\right)\right| &\leq \frac{\sqrt{\Cp \CKprime \; p(x) \cdot \Cp \CKprime \; p(y)}}{\sqrt{ h^d(x) \cdot h^d(x)}} \nonumber\\
    \leq ~ \frac{\Cp \CKprime}{\sqrt{\cp}} \frac{p(x)}{h^d(x)} ~ &=: ~ \CL{\ref{cov_close}}\frac{p(x)}{h^d(x)}. \label{cov_close_H1} 
\end{align}
\end{proof}

\subsection{Analysis of the upper bound in the bulk regime}

In the bulk regime, we recall that $\psib$ rejects $H_0$ if, and only if: $\Tb > \ThreshB t_n$. We prove the bulk upper bound by showing that $\psib$ has small type-I and type-II errors. To do so, we show that \textit{whp} under $H_0$, $\Tb \leq \ThreshB t_n$, whereas \textit{whp} under $H_1$: $\Tb > \ThreshB t_n$. This will be proved by computing the expectation and variance of $\Tb$ under $H_0$ in Proposition \ref{ExpVar_Bulk_H0} and under $H_1$ in Proposition \ref{ExpVar_Bulk_H1}. In both cases, we then use Chebyshev's inequality to show, in Corollary \ref{Risk_Tb}, that under $H_0$ $\Tb$ is concentrated below $\ThreshB t_n$, while under $H_1$ it is concentrated above $\ThreshB t_n$. 

\hfill

\begin{proposition}\label{ExpVar_Bulk_H0}
Under $H_0$ we have:
\begin{itemize}
    \item $\big|\E(\Tb)\big| \leq  t_n$
    \item $\V(\Tb) \leq \CvarNull t_n^2$,
\end{itemize} where 
$\CvarNull$ is a constant given in the proof. 
\end{proposition}

\hfill

\begin{proposition}\label{ExpVar_Bulk_H1}
There exists a constant $\nbulk$ depending only on $\RadB$ and a constant $\CvarAlt$ such that, whenever $n \geq \nbulk$ and $\left(\int_{\TB}|p-p_0|^t\right)^{1/t} \geq \RadB \rhob$, it holds:
\begin{itemize}
    \item $\E(\Tb) \geq (1 - 1/\sqrt{\RadB})^{\, 2} J $
    \item $\V(\Tb) \leq \frac{\CvarAlt}{\RadB}J^2$.
\end{itemize}
\end{proposition}

\hfill

\begin{corollary}\label{Risk_Tb}
There exist three large constants $\ThreshB$, $\RadB$, $\nbulk$ where $\nbulk$ only depends on $\RadB$, such that whenever $n \geq \nbulk$ and $\int_\TB |p-p_0|^t \geq \RadB {\rhob}^t$, it holds:
\begin{enumerate}
    \item $\mathbb{P}_{p_0} \left(\psib = 1\right) = \mathbb{P}_{p_0} \left(\Tb > \ThreshB t_n\right) \leq \frac{\eta}{4},$
    \item $\mathbb{P}_{p} \left(\psib = 0\right) \;=\; \mathbb{P}_{p} \left(\Tb \leq \ThreshB t_n\right) \,\leq\; \frac{\eta}{4}$.
\end{enumerate}
\end{corollary}

\hfill

\begin{proof}[Proof of Proposition \ref{ExpVar_Bulk_H0}] 
We place ourselves under $H_0$ and bound the expectation and variance of $\Tb$. We recall that $\hat p(x)$ and $\hat p'(x)$ are independent for all $x \in \TB$, and so are $\hat\Delta(x)$ and $\hat\Delta'(x)$.

\hfill

\underline{\textbf{Expectation:}} By the triangle inequality:
\begin{align*}
    \big|\E[\Tb]\big| = \Big|\int_{\TB} \omega(x)\; \E[\hat \Delta(x)] \; \E[\hat \Delta'(x)] dx\Big| \leq \int_{\TB} \omega(x) \big|\E[\hat \Delta(x)]\big| \; \big|\E[\hat \Delta'(x)]\big| dx.
\end{align*}

\vspace{3mm}

Now, recalling \eqref{def_Delta} and \eqref{estimates_of_p}, we have (see e.g. \cite{tsybakov2008introduction}, Prop 1.2):

\begin{equation}\label{control_biais}
    \big|\E[\hat \Delta(x)]\big| \leq \CK L \hB^\alpha(x) ~~~~~~ \text{ and } ~~~~~~ \big|\E[\hat \Delta'(x)]\big| \leq \CK L \hB^\alpha(x),
\end{equation}
\begin{align*}
    \text{so that: }  \big|\E[\Tb]\big| ~ \leq ~ \CK^2 L^2 \int_{\TB} \omega(x) h^{2 \alpha}(x) dx \leq \cch^{2\alpha} \CK^2 t_n \leq t_n,
\end{align*} 
by Lemma \ref{omega_h_2alpha} and taking $\cch$ small enough. 

\hfill\break

\underline{\textbf{Variance:}} By Lemma \ref{splitting_variance}, the variance under $H_0$ of $\Tb$ can be upper bounded as
\begin{align}
    \V(\Tb) \leq & \Big[ \big(\sqrt{J} + \CK\sqrt{t_n} \big)^2 + J_2^{1/2} \Big]^2 - \E^2 \Tb \leq [\CK t_n + J_2^{1/2}]^2.
\end{align}

\hfill 

We now analyse the covariance term in $J_2$. There are two cases. To analyse them, introduce the bulk diagonal:
\begin{equation}\label{def_diagonal}
    \Diag = \left\{(x,y) \in \TB^2 : \|x-y\| \leq \frac{1}{2}[h(x) + h(y)] \right\}.
\end{equation}

\underline{First case}: $\|x-y\|_2 > \frac{1}{2} [ h(x) + h(y)]$ i.e. $(x,y) \notin \Diag$.\\

By Lemma \ref{cov_far}, and recalling that $\forall x \in \TB: \tcA Lh^\alpha(x) \leq p_0(x)$ we have 

\begin{align}
    \frac{1}{k}\left|\text{cov}\left(K_{h(x)}(x-X),K_{h(y)}(y-X)\right) \right|& \leq \frac{1}{k} \left[\CK Lh(x)^\alpha + p_0(x)\right] \left[\CK Lh(y)^\alpha + p_0(y)\right] \nonumber \\
    &\leq \frac{1}{k} \Big(\frac{\CK}{\tcA}+1\Big)^2 p_0(x) p_0(y).\label{xy_close}
\end{align}

\hfill

\underline{Second case}: $\|x-y\| \leq \frac{1}{2}[h(x) + h(y)] \leq h(x) \lor h(y)$, i.e. $(x,y) \in \Diag$. \\

We suppose by symmetry $h(y) = h(x) \vee h(y)$. Then by Lemma \ref{cov_close}:
$$ \frac{1}{k}\left|\text{cov}\left(K_{h(x)}(x-X),K_{h(y)}(y-X)\right)\right| \leq \CL{\ref{cov_close}} \frac{p_0(x)}{k \,h^d(x)}.$$

Putting together the above equation with \eqref{exp_cross_term} and \eqref{xy_close}, we get:
\begin{align}
    \E\left[\hat \Delta(x) \hat \Delta(y)\right]  ~ \leq ~&  \CK^2 L^2h(x)^\alpha h(y)^\alpha ~ + ~  \mathbb{1}_{\Diag^c} \frac{1}{k} \Big(\frac{\CK}{\cA} +1 \Big)^2 p_0(x) p_0(y) + ~ \mathbb{1}_\Diag \CL{\ref{cov_close}} \frac{p_0(x)}{k \; h(x)^d }.\label{UB_cross_term}
\end{align}

We now combine \eqref{UB_cross_term} with \eqref{var_bulk_H0} and use Minkowski's inequality:
\begin{align}
    \V(\Tb) &\leq \left\{\Big(\iint_{\TB^2} \omega(x)\omega(y) \left[\CK^2 L^2h(x)^\alpha h(y)^\alpha \right]^2 dxdy\Big)^{1/2}+\right.\nonumber\\
    &\left. ~~~~~ \Big(\iint_{\TB^2} \omega(x)\omega(y) \left[\mathbb{1}_{\Diag^c} \frac{1}{k} \Big(\frac{\CK}{\tcA} +1 \Big)^2 p_0(x) p_0(y)  + ~ \mathbb{1}_\Diag \CL{\ref{cov_close}}  \frac{p_0(x)}{k \; h(x)^d } \right]^2 dxdy\Big)^{1/2} \right\}^2\nonumber \\
    & \leq \left\{\CK t_n + 2 \Big(\frac{\CK}{\tcA} +1 \Big)^2 t_n + \Big(\iint_{\TB^2} \omega(x)\omega(y)\left[\mathbb{1}_\Diag \CL{\ref{cov_close}}  \frac{p_0(x)}{k \; h(x)^d } \right]^2 dxdy\Big)^{1/2} \right\}^2 \label{UB_variance_H0}.
\end{align}
The last step is obtained by using Lemmas \ref{control_xy_far} and \ref{omega_h_2alpha}. We now analyse the term $ \CL{\ref{cov_close}} ^2 \iint_{D} \omega(x) \omega(y) \left[\frac{p_0(x)}{k \, h(x)^d} \right]^2 dxdy$. By Lemma \ref{variance_term_on_diagonal} (at the end of Appendix \ref{UB_Bulk_appendix}):
\begin{align*}
    \iint_{\Diag} \omega(x) \omega(y) \left[\frac{p_0(x)}{k \, h(x)^d} \right]^2 dxdy \leq  2\Comega \int_\TB \omega(x)^2 h(x)^d \left[\frac{p_0(x)}{k \; h(x)^d }\right]^2 dx = 8 \frac{\Comega}{{\ch}^d} t_n^2,
\end{align*}
by immediate calculation, recalling that $n = 2k$. Therefore, by equation \eqref{UB_variance_H0}, we have
\begin{equation}
    \V \Tb \leq \left[\CK t_n + 2 \Big(\frac{\CK}{\cA} +1 \Big)^2 t_n + \CL{\ref{cov_close}} \sqrt {8 \frac{\Comega}{{\ch}^d}} t_n \right]^2 =: \CvarNull t_n^2.
\end{equation}
\end{proof}

\textbf{Analysis of the test statistic under} $H_1$.

\hfill

\begin{proof}[Proof of Proposition \ref{ExpVar_Bulk_H1}]

\hfill\break

Suppose the data $(X_1, \dots, X_n)$ is drawn from a probability density $p$ satisfying: 
\begin{equation}\label{Assump_H1}
    \RadB^{\, t} {\rhob}^t \leq \int_\TB |p-p_0|^t.
\end{equation}

\hfill

\underline{\textbf{Expectation:}} We first prove $J\geq t_n$ in order to apply Lemma \ref{Expectation_Tb_H1}. We recall that $L' = (1+\delta)L$. From Equation \eqref{Assump_H1} we get:
\begin{align}
    \RadB^t \left[\frac{L^\frac{d}{4\alpha+d} \I^{\frac{1}{t} - \frac{\alpha}{4\alpha + d}}}{n^\frac{2\alpha}{4\alpha + d}}\right]^t &= {\RadB}^{\, t} {\rhob}^t \leq \int_\TB \left|\Delta\right|^t = \int_\TB \left[\omega(x) \Delta(x)^2\right]^\frac{t}{2}\omega(x)^{-\frac{t}{2}} dx \nonumber\\
    & \underset{\text{Hölder}}{\leq} \left[\int_\TB \omega(x) \Delta^2(x) dx\right]^\frac{t}{2} \left[\int_\TB \omega^{-\frac{t}{2-t}}\right]^\frac{2-t}{2},\label{Holder_omega}
\end{align}
where we have applied Hölder's inequality with $u = \frac{2}{t}$ and $v = \frac{2}{2-t}$ satisfying $\frac{1}{u} + \frac{1}{v} = 1$.
Hence:
\begin{align}
    J = \int_\TB \omega(x) \Delta^2(x) dx 
    &\geq \RadB^2 \left(\frac{L^\frac{d}{4\alpha+d} \I^{\frac{1}{t} - \frac{\alpha}{4\alpha + d}}}{n^\frac{2\alpha}{4\alpha + d}}\right)^2 \times \left(\int_\TB \frac{1}{\omega^\frac{t}{2-t}}\right)^{-\frac{2-t}{t}} \geq \RadB^2 \CL{\ref{Bulk_asymp_BulkCases}}^\frac{t-2}{t}\frac{t_n}{\Ctn} \geq \RadB t_n \label{L2_estim_t_norm}.
\end{align}
Taking $\RadB$ large enough yields $J \geq t_n$, hence we can apply Lemma \ref{Expectation_Tb_H1} which yields
\begin{align}\label{minoration_E_Tb}
    \E \Tb \geq \big(\sqrt{J} - \sqrt{t_n} \big)^2 \geq \big(\sqrt{J} - \sqrt{\frac{J}{\RadB}}\big)^2 = (1 - 1/\sqrt{\RadB})^{\, 2} J,
\end{align}
where we recall that $\RadB$ can be taken arbitrarily large.

\hfill

\underline{\textbf{Variance:}}\\

We still have by Lemma \ref{splitting_variance} and by \eqref{minoration_E_Tb}:
\vspace{-2mm}
\begin{align}
    \V(\Tb) & \leq \Big[ \big(\sqrt{J} + \sqrt{t_n} \big)^2 + J_2^{1/2} \Big]^2 - \E^2 \Tb \\
    &\leq  \Big[ \big(\sqrt{J} + \sqrt{t_n} \big)^2 + J_2^{1/2} \Big]^2 - \big(\sqrt{J} - \sqrt{t_n} \big)^4, \label{var_bulk_H1}
\end{align}
where
\vspace{-8mm}

\begin{align*}
    J_2 = \iint_{\TB^2} \omega(x) \omega(y) \frac{1}{k^2} \text{cov}^2\left(K_{h(x)}(x-X),K_{h(y)}(y-X)\right) dx dy.
\end{align*}

We now compute $J_2$. We have 
\begin{equation}\label{split_J}
    J_2 = J_{\Diag} + J_{\Diag^c}
\end{equation}
where 
\begin{align}
    J_{\Diag} &= \iint_{\Diag} \omega(x) \omega(y) \frac{1}{k^2} \text{cov}^2\left(K_{h(x)}(x-X),K_{h(y)}(y-X)\right) dx dy, \label{def_JD}\\
    J_{\Diag^c} &= \iint_{\Diag^c} \omega(x) \omega(y) \frac{1}{k^2} \text{cov}^2\left(K_{h(x)}(x-X),K_{h(y)}(y-X)\right) dx dy, \label{def_JDc}
\end{align}

We examine $J_\Diag$ and $J_{\Diag^c}$ separately.\\

\textbf{\underline{Term $J_{\Diag^c}$}}. We have outside the diagonal $\Diag$: $\|x-y\| > \frac{1}{2}[h(x) + h(y)]$.\\

By Lemma \ref{cov_far}, and using $\frac{\tcA}{\cch^\alpha} L h^\alpha \leq p_0$ on $\TB$: 
\begin{align*}
    \frac{1}{k}\left| \textnormal{cov} \left(K_{h(x)}(x-X),K_{h(y)}(y-X)\right) \right| &\leq \frac{1}{k} \left(\CK' L h(x)^\alpha + p(x)\right) \left(\CK' L h(y)^\alpha + p(y)\right)\\
    & \leq \frac{1}{k} \Big[\Big(\frac{\CK'\cch^\alpha}{\tcA} + 1\Big) p_0(x) + |\Delta(x)| \Big]\Big[\Big(\frac{\CK'\cch^\alpha}{\tcA} + 1\Big) p_0(y) + |\Delta(y)| \Big]\\
    & =: \frac{1}{k} \Big[\Ccov p_0(x) + |\Delta(x)| \Big]\Big[\Ccov p_0(y) + |\Delta(y)| \Big],
\end{align*}
where $\Ccov = \frac{\CK'\cch^\alpha}{\tcA} + 1$. 
Therefore, outside the diagonal \eqref{def_diagonal}, we have:
\begin{align}
    J_{\Diag^c} &= \frac{1}{k^2}\iint_{\Diag^c} \omega(x) \omega(y) \text{cov}^2\left(K_{h(x)}(x-X),K_{h(y)}(y-X)\right)  dx dy\nonumber \\
    & \leq \frac{1}{k^2} \iint_{\Diag^c}\omega(x) \omega(y) \Big[\Ccov p_0(x) + |\Delta(x)| \Big]^2\Big[\Ccov p_0(y) + |\Delta(y)| \Big]^2dx dy\nonumber\\
    & = \left[\frac{1}{k} \int_\TB \omega(x) \Big(\Ccov p_0(x) + |\Delta(x)| \Big)^2dx \right]^2 \leq \left[\frac{1}{k} \int_\TB \omega(x) \Big({\Ccov}^2 p_0^2(x) + |\Delta(x)|^2 \Big)dx \right]^2 \nonumber\\
    & \leq \frac{4}{k^2}\left(\Ccov\, t_n + J \right)^2 ~~ \text{ by Lemma \ref{control_xy_far}}\nonumber\\
    & \leq \frac{16}{n^2} \left(\frac{\Ccov}{\RadB} +1 \right)^2 J^2 =: \frac{C_{\Diag^c}}{n^2} J^2. \label{Majorant_JDc}
\end{align}

\hfill\break

\textbf{\underline{For $J_\Diag$}:} If $\|x-y\| \leq \frac{1}{2} [h(x) + h(y)] \leq h(x) \lor h(y)$, we suppose by symmetry $h(y) = h(x) \lor h(y)$. Then by Lemma \ref{cov_close}, we have:
$$ \frac{1}{k}\left|\text{cov}\left(K_{h(x)}(x-X),K_{h(y)}(y-X)\right)\right| \leq \CL{\ref{cov_close}}   \frac{p(x)}{k \,h^d(x)}.$$
Therefore:
\begin{align*}
    J_\Diag& = \frac{1}{k^2}\iint_{\TB^2} \omega(x) \omega(y)\,  \text{cov}^2\left(K_{h(x)}(x-X),K_{h(y)}(y-X)\right)  dx dy \\
    &\leq \iint_{\Diag^2} \omega(x) \omega(y) \left[\CL{\ref{cov_close}}   \frac{p(x)}{k \,h^d(x)} \right]^2 dx dy.
\end{align*}

By Lemma \ref{variance_term_on_diagonal}, we can upper bound the term as:

\begin{align}
    \frac{J_\Diag}{\CL{\ref{cov_close}} ^2} &\leq \iint_{\Diag} \omega(x) \omega(y) \left[\frac{p(x)}{k \, h(x)^d} \right]^2 dxdy ~~ \leq ~~ 2\Comega \int_\TB \omega(x)^2 h(x)^d \left[\frac{p(x)}{k \; h(x)^d }\right]^2 dx\nonumber \\
    & \leq ~~ 2\frac{\Comega}{\cch^d} \int_\TB \frac{\omega(x)^2}{k^2 h(x)^d} [ 2p_0(x)^2 + 2\Delta^2(x)] dx  ~~ \leq ~~  16 \, \frac{\Comega}{\cch^d} t_n^2 + 16\frac{\Comega}{\cch^d} \underbrace{\int_\TB \frac{\omega(x)^2\Delta^2(x)}{n^2 \hB^d(x)} dx}_{\text{Term I}}. \label{JD_overC1}
\end{align}

\textbf{\underline{Term I}}: We recall that by definition we have $p_0 \geq \tuA$ on $\TB$ and we have $2\alpha t -2d -4\alpha <0$ since $t \in [1,2]$. Therefore:
\begin{align*}
    \frac{1}{k^2} \int_\TB \omega^2(x)  \frac{\Delta(x)^2}{\hB^d(x)} dx & = \frac{1}{n^2} \int_\TB (n^2L^4 \I)^\frac{d}{4\alpha + d} \; p_0^\frac{2\alpha t - 2d - 4 \alpha}{(4-t)\alpha + d}  \omega(x)  \Delta(x)^2 dx\\
    & \leq \frac{(n^2L^4 \I)^\frac{d}{4\alpha + d}}{n^2} \int_\TB \omega(x) \Delta(x)^2 \; \tuA^\frac{2\alpha t - 2d - 4 \alpha}{(4-t)\alpha + d} dx\\
    & \leq \frac{(n^2L^4 \I)^\frac{d}{4\alpha + d}}{n^2} \int_\TB \omega(x) \Delta(x)^2 \left[\tcA \frac{L^d}{n^{2\alpha}\I^\alpha}\right]^{\frac{1}{4\alpha+d}\frac{(4-t)\alpha + d}{(2-t)\alpha + d}\frac{2\alpha t - 2d - 4 \alpha}{(4-t)\alpha + d}}  \\
    & = \ctilde\, t_n \int_\TB \omega(x) \Delta(x)^2 = \ctilde\, t_n J.
\end{align*}
where $\ctilde = \tcA^{-\frac{1}{4\alpha+d} \frac{-2\alpha t + 2d + 4 \alpha}{(2-t)\alpha + d} }$. Therefore, by equation \eqref{JD_overC1}:
\begin{equation}\label{Majorant_JD}
    J_\Diag \leq \CL{\ref{cov_close}}  \left[16\frac{\Comega}{\cch^d}  t_n^2 + 16 \frac{\Comega}{\cch^d} \ctilde t_n J \right] =: A_{\Diag} t_n^2 + B_{\Diag}t_n J,
\end{equation}
for two constants $A_{\Diag}$ and $B_{\Diag}$. By equations \eqref{Majorant_JDc} and \eqref{Majorant_JD}, it holds:
\begin{equation}
    J_2 \leq \AJtwo\, t_n^2 + \BJtwo\, t_n J + \CJtwo \frac{J^2}{n^2},
\end{equation}
for three constants $\AJtwo, \BJtwo, \CJtwo >0$. Recalling Equation \eqref{L2_estim_t_norm}, we can further upper bound $J_2$ as $J_2 \leq \AJtwo J^2/\RadB^2 + \BJtwo\,J^2/\RadB + \CJtwo \frac{J^2}{n^2}$, hence taking $\nbulk := \left\lceil\sqrt{\RadB}\right\rceil$, we get:
\begin{equation}\label{Majorant_J_2}
    J_2 \leq \frac{\AJtwo+\BJtwo+\CJtwo}{\RadB} J^2 =: \frac{\DJtwo^2}{\RadB} J^2.
\end{equation}
\end{proof}
It then follows, from Equation \eqref{var_bulk_H1}:
\begin{align*}
    \V T_b &\leq \Big[ \big(\sqrt{J} + \sqrt{t_n} \big)^2 + J_2^{1/2} \Big]^2 - \big(\sqrt{J} - \sqrt{t_n} \big)^4\\
    & \leq \Big[ \big(\sqrt{J} + \sqrt{\frac{J}{\RadB}} \big)^2 + \frac{\DJtwo}{\sqrt{\RadB}} J \Big]^2 - \big(\sqrt{J} - \sqrt{J/\RadB}\big)^4 ~~ \text{ by Equation \eqref{L2_estim_t_norm} and \eqref{Majorant_J_2}}\\
    & = J^2 \left\{\left[1+ \frac{2+\DJtwo}{\RadB} + \frac{1^{\,2}}{\RadB}\right]^2 - \left[1-4\frac{1}{\sqrt{\RadB}} + O\left(\frac{1}{\RadB}\right) \right] \right\}\\
    & \leq J^2 \left[\frac{8 + 2\DJtwo +1}{\RadB} \right] ~~ \text{ for $\RadB$ large enough}\\
    & =:\frac{\CvarAlt}{\RadB} J^2.
\end{align*}

\hfill

\subsection{Proof of Corollary \ref{Risk_Tb}}

\begin{proof}[Proof of Corollary \ref{Risk_Tb}]

\hfill

 \begin{enumerate}
     \item Set $\ThreshB> 1$. It holds:
     \begin{align*}
         \mathbb{P}_{p_0}\left(\Tb>\ThreshB t_n\right) &\leq \mathbb{P}_{p_0}\left(\left|\Tb-\mathbb{E}\Tb\right|>(\ThreshB-1) t_n\right) ~~ \text{ by Proposition \ref{ExpVar_Bulk_H0}}\\
         & \leq \frac{\CvarNull t_n^2}{(\ThreshB - 1)^2 t_n^2}  ~~ \text{ by Proposition \ref{ExpVar_Bulk_H0} and Chebyshev's inequality}\\
         &\leq \frac{\eta}{4} ~~ \text{ for $\ThreshB$ larger than a suitable constant.}
     \end{align*}
     \item Assume $\RadB$ is large enough to ensure
     \begin{equation}\label{C'large_enough}
         (1-1/\sqrt{\RadB})^2 > \frac{\ThreshB}{\RadB}.
     \end{equation}
     The value of the constant $\RadB$ being given, assume moreover that $n \geq \nbulk$. We then have:
     \begin{align*}
         \mathbb{P}_{p}\left(\Tb\leq \ThreshB t_n\right) &\leq \mathbb{P}_{p}\left(\Tb-\mathbb{E} \Tb\leq \ThreshB t_n - (1-1/\sqrt{\RadB})^2 J\right) ~~ \text{ by Proposition \ref{ExpVar_Bulk_H1}}\\
         &\leq \mathbb{P}_{p}\left(\Tb-\mathbb{E} \Tb\leq \frac{\ThreshB}{\RadB} J - (1-1/\sqrt{\RadB})^2 J\right) ~~ \text{ by Equation \eqref{L2_estim_t_norm}}\\
         & \leq \mathbb{P}_{p}\left(|\Tb-\mathbb{E} \Tb|\leq  (1-1/\sqrt{\RadB})^2 J- \frac{\ThreshB}{\RadB} J\right) ~~ \text{ by Equation \eqref{C'large_enough}}\\
         &\leq \frac{\CvarAlt J^2 / \RadB}{\left((1-1/\sqrt{\RadB})^2- \frac{\ThreshB}{\RadB}\right)^2J^2} ~~ \text{ by Chebyshev's inequality}\\
         & \leq \frac{\eta}{4} ~~ \text{ for $\RadB$ large enough.}
     \end{align*}
 \end{enumerate}
\end{proof}

\begin{lemma}\label{variance_term_on_diagonal}
For any $p \in \mathcal{P}(\alpha,L)$ it holds that: $$\iint_{\Diag} \omega(x) \omega(y) \left[\frac{p(x)}{k \, h(x)^d} \right]^2 dxdy \leq 2\Comega \int_\TB \omega(x)^2 h(x)^d \left[\frac{p(x)}{k \; h(x)^d }\right]^2 dx.$$
\end{lemma}

\begin{proof}[Proof of Lemma \ref{variance_term_on_diagonal}]
 We set:
\begin{equation}\label{def_D+}
    \Diag_+ = \left\{(x,y) \in \Diag: p_0(x) \geq p_0(y) \right\}.
\end{equation}
On $\Diag_+$, we have $\|x-y\| \leq h(x) \lor h(y) = h(x)$ so in particular: $y \in B(x,h(x))$. We therefore have:
\begin{align*}
    & \iint_{\Diag}  \frac{\omega(x) \omega(y) p^2(x)}{k^2 \, h(x)^{2d}}  dxdy = 2 \iint_{\Diag_+} \frac{\omega(x) \omega(y) p^2(x)}{k^2 \, h(x)^{2d}}  dxdy \\
    &\leq 2\int_{\TB} \omega(x)  \left[\frac{p(x)}{k \, h(x)^d} \right]^2 \left[ \int_{ B(x,h(x))} \omega(y) dy\right]dx \leq 2\int_{x \in \TB} \omega(x)  \left[\frac{p(x)}{k \, h(x)^d} \right]^2 \left\{ h(x)^d \Comega \omega(x) \right\}dx \\
    &= 2\Comega \int_\TB \omega(x)^2 h(x)^d \left[\frac{p(x)}{k \; h(x)^d }\right]^2 dx.
\end{align*}
\end{proof}
\section{Lower bound in the bulk regime: Proof of Proposition \ref{Lower_bound_bulk_proposition}}\label{Lower_bound_bulk_appendix}

We define here the bulk prior. 
Fix $c$ a constant, allowed to be arbitrarily small. 
We apply Algorithm \ref{Algo_partitionnement}, with $\widetilde \Omega$, $\beta = \frac{2}{(4-t)\alpha + d}$, $u = \uA$, $\cbeta = c^{-\beta}\left(n^2L^4\I\right)^\frac{\beta}{4\alpha+d}$ and set $\calpha = \cA c^{-\alpha}$. Following the notation from Algorithm \ref{Algo_partitionnement}, let $h=(p_0/\cbeta)^{1/\beta}$. 
The choice of the constants ensures $h \leq c \hB$ and $p_0 \geq \calpha Lh^\alpha$ over $\B(\uA)$. 
Moreover, $c$ is chosen small enough to ensure $\calpha \geq \frac{\sqrt{d}^\alpha(2^{1-\alpha}\lor1)}{1/2-\cstar}$. 
\unboundedcase{Since}\boundedcase{If} $\Omega$ is unbounded, we can moreover choose a subset $\widetilde \Omega$ large enough to ensure that it is split at least once by Algorithm \ref{Algo_partitionnement}.
Therefore, the guarantees of Proposition \ref{guarantees_algo} are ensured.  Let $j\in \{1, \dots, N\}$ and consider the cell $B_j$. Its center is denoted by $x_j$ and we also set $h_j = c \hB(x_j)/4$ where $c$ is the constant used to define the constants $\calpha, \cbeta$ taken as inputs for Algorithms \ref{Algo_partitionnement}. We also set $\overset{\rightarrow}{\mathbb{1}} = (1,\dots, 1)$. 
Define the perturbation function $f\geq 0$ over $\R^d$, such that $f \in H(\alpha,1)\cap C^\infty$, $f$ is supported over $\left\{x \in \R^d : \|x\|<1/2 \right\}$.
We define the perturbations $(\phi_j)_{j=1}^N$ as follows:
\begin{equation}\label{perturbation}
     \phi_j(x) = \Cphi L h_j^\alpha~ f\bigg(\frac{x-x_j-\frac{h_j}{\sqrt{d}}\overset{\rightarrow}{\mathbb{1}}}{h_j} \bigg) -  \Cphi L h_j^\alpha ~ f\bigg(\frac{x-x_j+\frac{h_j}{\sqrt{d}}\overset{\rightarrow}{\mathbb{1}}}{h_j} \bigg),
\end{equation}
where $\Cphi$ is a small enough constant. For $\epsilon = (\epsilon_1, \dots, \epsilon_N)$ where $\epsilon_j \overset{iid}{\sim} Rad(\frac{1}{2})$, the prior is defined as follows:
\begin{equation}\label{prior_bulk}
    p_\epsilon^{(n)} = p_0 + \sum_{j=1}^N \epsilon_j \phi_j.
\end{equation}

For clarity, we give the probability density over $\Omega^n$ corresponding to data drawn from this prior distribution. Assume that $(X'_1, \dots, X'_n)$ are drawn from \eqref{prior_bulk}. The data is therefore $iid$ with the \textit{same} density $q$, itself uniformly drawn in the set $\{p_\epsilon ~|~  \epsilon \in \{\pm 1\}^n\}$. In other words, the density of $(X'_1, \dots, X'_n)$ corresponds to the mixture 
$$ \widetilde p = \frac{1}{2^N} \sum_{\epsilon \in \{\pm 1\}^N} \Big( p_0 + \sum_{j=1}^N \epsilon_j \phi_j \Big)^{\otimes n}, $$
where, for any $q \in \mathcal{P}(\alpha, L)$, $q^{\otimes n}$ is defined by $q^{\otimes n}(x_1, \dots, x_n) = q(x_1) \dots q(x_n)$ and represents the density of $(Y_1, \dots, Y_n)$ when $Y_i \overset{iid}{\sim} q$. 
\\

The lower bound will be proved by showing that there exist no test with risk smaller than $\eta$ for the testing problem $H_0 : (X'_1, \dots, X'_n) \sim p_0^{\otimes n}$ vs $H'_1 : (X'_1, \dots, X'_n) \sim \widetilde p$. Whenever no ambiguity arises, we will just write $p_\epsilon$ instead of $p_\epsilon^{(n)}$. Recalling that $L' = (1+\delta)L$ and $\cstar' = (1+\delta)\cstar$, the following proposition states that this prior is admissible, \textit{i.e.} that each one of these densities belongs to $\mathcal{P}(\alpha,L',\cstar')$.

\begin{proposition}\label{prior_bulk_admissible}
For all $\epsilon = (\epsilon_1, \dots, \epsilon_N) \in \{\pm 1\}$: $p_\epsilon \in \mathcal{P}(\alpha,L', \cstar')$.
\end{proposition}

We now prove that this prior distribution gives a lower bound on $\rhob$. This lower bound will be denoted by $\rho_{bulk}^{LB}$, defined as the $L_t$ norm of the perturbation:
\begin{equation}
    \rho_{bulk}^{LB} = \big\|\sum_{j=1}^N \phi_j\big\|_t.
\end{equation}
Then by definition, $\forall \epsilon \in \{\pm 1\}^N: p_\epsilon \in H_1(\rho_{bulk}^{LB})$. Moreover, the following Proposition states that the prior we consider yields a lower bound of order $ \rhob$:

\begin{proposition}\label{separation_prior_bulk}
There exists a constant $\CbulkLB$ given in the Appendix, such that 
$$\rho_{bulk}^{LB} = \CbulkLB \,  \rhob.$$
\end{proposition}

We now introduce the \textit{Bayes risk} associated with the prior distribution \eqref{prior_bulk}:
\begin{definition}\label{bayesian_risk_bulk_prior}
Define
$$ R_B^{\; bulk} = \inf_{\psi \text{ test}} \Big\{\mathbb{P}_{p_0}(\psi = 1) + \mathbb{E}_{\epsilon} \Big[ \mathbb{P}_{p_\epsilon}(\psi = 0)\Big]\Big\},$$
where the expectation is taken with respect to the realizations of $\epsilon$ and $\mathbb{P}_{p_e}$ denotes the probability distribution when the data is drawn with density \eqref{prior_bulk}.
\end{definition}

As classical in the minimax framework, we have $R^*(\rho_{bulk}^{LB}) \geq R_B^{\; bulk}$ (indeed, the supremum in \eqref{def_minimax_risk} can be lower bounded by the expectation over $\epsilon$). The following proposition states that $\rho_{bulk}^{LB}$ is indeed a lower bound on $\rhob$:

\begin{proposition}\label{lower_bound_bulk}
It holds $R_B^{\; bulk} > \eta.$ 
\end{proposition}

Indeed, Proposition \ref{lower_bound_bulk} proves that $R^*(\rho_{bulk}^{LB}) > \eta$. Since $R^*(\rho)$ is a decreasing function of $\rho$, we therefore have $\rho^* > \rho_{bulk}^{LB}$ by the definition of $\rho^*$ in equation \eqref{separation_radius}. This ends the proof of Proposition \ref{Lower_bound_bulk_proposition}\\

\subsection{Proof of Proposition \ref{prior_bulk_admissible}}

\begin{proof}[Proof of Proposition \ref{prior_bulk_admissible}]
First, for each $j=1, \dots, N$ the functions $\Cphi L \Big(h_j - \big\|x-x_j \pm \frac{h_j}{\sqrt{d}}\overset{\rightarrow}{\mathbb{1}} \big\| \Big)_+^\alpha$ are in $H(\alpha, \Cphi L)$ and have disjoint support so that their sum also belongs to $H(\alpha, \Cphi L)$. Hence for all $\epsilon \in \{\pm 1\}^N$, $\sum\limits_{j=1}^N \epsilon_j \phi_j \in H(\alpha, \Cphi L)$, proving that $p_\epsilon \in H(\alpha, (1 + \Cphi )L)$. \\

Now, note that we have:
\begin{equation}\label{hb_leq_p0_B(uA/2)}
    \forall x \in \B\left(\frac{\uA}{2}\right): ~~ L\hB(x)^\alpha \leq \frac{2^{\frac{(2-t)\alpha+d}{(4-t)\alpha + d}}}{\cA}p_0(x).
\end{equation}
 

Let $\epsilon \in \{\pm 1\}^N$ and $x \in B_j$, for some $j \in \{1, \dots, N\}$. Recalling that $h_j = c\,\hB(x_j)$ we have:
\begin{align*}
    \hspace{2cm} |\phi_j(x)| &\leq \Cphi \|f\|_1 L h_j^\alpha = \Cphi \|f\|_1 c^\alpha L \hB(x_j)^\alpha \\ 
    &\leq \Cphi \|f\|_1 c^\alpha \CL{\ref{p0_const_B(uA/2)}}^{(p_0)} L \hB(x)^\alpha \, \hspace{25mm} \text{ by Lemma \ref{p0_const_B(uA/2)}} \\
    &\leq \frac{2^{\frac{(2-t)\alpha+d}{(4-t)\alpha + d}}}{\cA}\Cphi \|f\|_1 c^\alpha \CL{\ref{p0_const_B(uA/2)}}^{(p_0)} p_0(x)  \hspace{2cm} \text{ by equation \eqref{hb_leq_p0_B(uA/2)}}\\
    & =: \lambda \, p_0(x).
\end{align*}


Therefore, by Lemma \ref{simplifying_assumpt_satisfied}, we have $p+\sum \epsilon_j \phi_j \in \mathcal{P}(\alpha, (1+\lambda) L, \frac{\cstar + 2\lambda + \lambda\cstar}{1-\lambda})$. Choosing the constant $\lambda$ small enough (by adjusting $\Cphi$), we can ensure $p_\epsilon \in \mathcal{P}(\alpha, L', \cstar')$ where $L'= (1+\delta)L$ and $\cstar' = (1+\delta)\cstar$. 
\end{proof}

\hfill

\subsection{Proof of Proposition \ref{separation_prior_bulk}}

\begin{proof}[Proof of Proposition \ref{separation_prior_bulk}]
Since the $\phi_j$ have disjoint support, we get that: $\big\|\sum_{j=1}^N \phi_j\big\|_t = \sum_{j=1}^N \big\| \phi_j\big\|_t$. 
Now, let $j\in \{1, \dots, N\}$. We have:
\begin{align*}
    \|\phi_j\|_t^t &= 2\int_{B_j} \Big\{\Cphi L h_j^\alpha~ f\bigg(\frac{x-x_j-\frac{h_j}{\sqrt{d}}\overset{\rightarrow}{\mathbb{1}}}{h_j} \bigg)\Big\}^t dx\\
    & = 2 \big(\Cphi L \|f\|_t\big)^t ~ h_j^{\alpha t + d} \\
    & \geq  2 \big(\Cphi L \|f\|_t\big)^t \; \frac{c^{\alpha t}}{{{\CL{\ref{p0_const_B(uA/2)}}^{(h)}}}^{\alpha t}} \int_{B_j} \hB^{\alpha t} ~~ \text{ by Lemma \ref{p0_const_B(uA/2)}}\\
    & =: \CbulkLB L^t\int_{B_j} \hB^{\alpha t} 
\end{align*}
where $\CbulkLB = \frac{2\, c^{\alpha t} \,{\Cphi }^t}{{{\CL{\ref{p0_const_B(uA/2)}}^{(h)}}}^{\alpha t}} \|f\|_t^t$, so that
\begin{align*}
    \sum_{j=1}^N \|\phi_j\|_t^t= \CbulkLB L^t \int_{\cup_{j=1}^N B_j} \hB^{\alpha t} \geq \CbulkLB L^t \int_{\B(\uA)} \hB^{\alpha t} = \CbulkLB {\overline{\rhob}}^t,
\end{align*}
where 
\begin{equation}\label{def_rhob_bar}
    \overline{\rhob} = \left(\frac{L^d}{n^{2\alpha }\I^\alpha}\right)^\frac{t}{4\alpha + d} \int_{\B(\uA)} p_0^r.
\end{equation}

It follows that ${\rho_{Bulk}^{LB}}^t \geq \CbulkLB {\overline{\rhob}}^t$. 
\\

Now, if $\uA = \uI$ then $\overline{\rhob} = \rhob$. Otherwise, if $\uA > \uI$ ,we have by Lemma \ref{uA>uI_then_tail} and Lemma \ref{htail_all_equal}: $\rhot \asymp \overline{\rhot} \geq \CL{\ref{uA>uI_then_tail}} \rhob \geq \overline{\rhob}$. Therefore: $\overline{\rhob} + \rhot \asymp \rhob + \rhot$.

In particular, we have proved that, if the bulk dominates, then 
\begin{equation}
    \|\sum_j \phi_j\|_t \gtrsim \rhob. \label{eq:separation_LB_bulk_phij}
\end{equation}
\end{proof}

\hfill

\subsection{Proof of Proposition \ref{lower_bound_bulk}}

\begin{proof}[Proof of Proposition \ref{lower_bound_bulk}]

As classical in the minimax literature we always have $R^*(\rho_{bulk}^{LB}) \geq R_B^{\; bulk} = 1 - d_{TV}(p_0^{\otimes n}, p_\epsilon^{(n)})$. Moreover, by Pinsker's inequality (see e.g. \cite{tsybakov2008introduction}) we have $d_{TV}(p_0^{\otimes n}, p_\epsilon^{(n)}) \leq \frac{1}{2} \sqrt{\chi^2\big(p_\epsilon^{(n)} ||  p_0^{\otimes n} \big)}$, therefore: $R_B^{\; bulk} \geq 1 - \frac{1}{2} \sqrt{\chi^2\big(p_\epsilon^{(n)} ||  p_0^{\otimes n} \big)}$. To prove that $R_B^{\; bulk} > \eta$, it therefore suffices to prove that $\chi^2\big(p_\epsilon^{(n)} ||  p_0^{\otimes n} \big) < 4(1-\eta)^2$. We recall that by Proposition \ref{guarantees_algo} item \ref{min_geq_max}, we have:
\begin{equation} \label{cup_Bj_sub_B(uA/2)}
    \B(\uA) ~ \subset~ \bigcup_{j=1}^N B_j ~\subset~ \B(\frac{\uA}{2}).
\end{equation}

We now compute $1+\chi^2\big(p_\epsilon^{(n)} || p_0^{\otimes n} \big)$.

\begin{align}
    &1+\chi^2\big(p_\epsilon^{(n)} ||  p_0^{\otimes n} \big) = \int_{\Omega^n} \frac{\big(\frac{1}{2^N} \sum_{\epsilon \in \{\pm 1\}^N} \prod_{i=1}^n p_\epsilon(y_i) \big)^2}{\prod_{i=1}^n p_0(y_i)} dy_1 \dots dy_n\nonumber\\
    & = \frac{1}{4^N} \int \sum_{\epsilon, \epsilon' \in \{\pm 1\}^N} \prod_{i=1}^n \frac{p_\epsilon(y_i)p_{\epsilon'}(y_i)}{p_0(y_i)}dy_1 \dots dy_n\nonumber\\
    & = \frac{1}{4^N} \int_{\Omega^n} \sum_{\epsilon, \epsilon' \in \{\pm 1\}^N} \prod_{i=1}^n \frac{\big( p_0(y_i) + \sum_{j=1}^N \epsilon_j \phi_j(y_i) \big)\big( p_0(y_i) + \sum_{j=1}^N \epsilon'_j \phi_j(y_i) \big)}{p_0(y_i)}dy_1 \dots dy_n\nonumber\\
    & = \frac{1}{4^N} \sum_{\epsilon, \epsilon' \in \{\pm 1\}^n} \left( \int_{\Omega} p_0(x) +\sum_{j=1}^N (\epsilon_j + \epsilon'_j) \phi_j(x) +  \sum_{j=1}^N \epsilon_j \epsilon'_j \frac{\phi_j^2(x)}{p_0(x)} dx \right)^n\nonumber\\
    & = \frac{1}{4^N} \sum_{\epsilon, \epsilon' \in \{\pm 1\}^n} \left( 1 + \sum_{j=1}^N \epsilon_j \epsilon'_j \int_\Omega\frac{\phi_j^2(x)}{p_0(x)} dx \right)^n\nonumber\\
    & \leq \frac{1}{4^N} \sum_{\epsilon, \epsilon' \in \{\pm 1\}^n} \exp\big(n \sum_{j=1}^N \epsilon_j \epsilon'_j \int_\Omega\frac{\phi_j^2(x)}{p_0(x)} dx\big)\nonumber\\
    & = \prod_{j=1}^N\left(\frac{1}{4} \sum_{\epsilon_j, \epsilon'_j \in \{\pm 1\}}\exp\big(n  \epsilon_j \epsilon'_j \int_\Omega\frac{\phi_j^2(x)}{p_0(x)} dx\big) \right) = \prod_{j=1}^N \cosh \big(n\int_\Omega\frac{\phi_j^2(x)}{p_0(x)} dx \big) \nonumber\\
    & \leq \exp\Big( \frac{1}{2}\sum_{j=1}^N n^2 \big(\int_\Omega\frac{\phi_j^2(x)}{p_0(x)} dx \big)^2 \Big) \label{quantite_chi2}\\
    & \leq \exp\Big( \frac{1}{2}\sum_{j=1}^N n^2 {\Cphi}^4 \Ctilde L^4 \int_{B_j} \frac{\hB^{~~4\alpha + d}}{p_0^2} \Big) ~~~~~ \text{ by Lemma \ref{int_phij4_hd_p02}}\nonumber\\
    & \leq \exp\Big( \frac{1}{2} n^2 {\Cphi}^4 \Ctilde L^4 \frac{1}{n^2L^4\I} \int_{\B(\frac{\uA}{2})}p_0^r \Big) ~~~~~ \text{ by equation \eqref{cup_Bj_sub_B(uA/2)}}\nonumber\\
    & \leq \exp\Big( \frac{1}{2} n^2 {\Cphi}^4 \Ctilde L^4\CL{\ref{Bulk_asymp_BulkCases}} \frac{\I}{n^2L^4\I} \Big) ~~~~~~~~~~~~~~ \text{ by Lemma \ref{Bulk_asymp_BulkCases}}\nonumber\\
    &= \exp\big({\Cphi}^4 \Ctilde\big) ~\leq~ 1 + 4(1 - \eta)^2 ~~ \text{ for } \Cphi \leq \Big(\frac{1}{\Ctilde} \log(1+4(1-\eta)^2) \Big)^\frac{1}{4}.\nonumber
\end{align}
\end{proof}

\subsection{Technical results for the LB in the bulk regime}

We recall that \textbf{in this section}, $\calpha = \cA c^{-\alpha}$, where $c$ is a constant chosen small enough to ensure $\calpha \geq \frac{\sqrt{d}^\alpha(2^{1-\alpha}\lor1)}{1/2-\cstar}$.

\begin{lemma}\label{p0_const_B(uA/2)}
Let $j \in \{1,\dots, M\}$ and $x \in B_j$. Denote by $x_j$ the center of $B_j$. Then $\frac{p_0(x)}{p_0(x_j)} \in \left[\cL{\ref{p0_const_B(uA/2)}}^{(p_0)}, \CL{\ref{p0_const_B(uA/2)}}^{(p_0)}\right]$ where $\cL{\ref{p0_const_B(uA/2)}}^{(p_0)} = \frac{1}{2}$ and $\CL{\ref{p0_const_B(uA/2)}}^{(p_0)} = \frac{3}{2}$ are two  constants. It follows that $\frac{\hB(x)}{\hB(x_j)} \in \left[\cL{\ref{p0_const_B(uA/2)}}^{(h)}, \CL{\ref{p0_const_B(uA/2)}}^{(h)}\right]$ where $\cL{\ref{p0_const_B(uA/2)}}^{(h)} = \left(\frac{1}{2}\right)^\frac{2}{(4-t)\alpha+d}$ and $\CL{\ref{p0_const_B(uA/2)}}^{(h)} = \left(\frac{3}{2}\right)^\frac{2}{(4-t)\alpha+d}$ are two  constants.
\end{lemma}

\begin{proof}[Proof of Lemma \ref{p0_const_B(uA/2)}]
The proof follows from Assumption \eqref{simplifyingAssumption}:
\begin{align*}
    |p_0(x)-p_0(x_j)| &\leq \cstar p_0(x_j) + L(e(B_j)\sqrt{d})^\alpha \leq  \cstar p_0(x_j) + Lh^\alpha(x_j)\sqrt{d}^\alpha\\
    &\leq \Big(\cstar + \frac{\sqrt{d}^\alpha}{\calpha}\Big)p_0(x_j) ~ \leq ~  \frac{p_0(x_j)}{2}.
\end{align*}
\end{proof}

\begin{lemma}\label{simplifying_assumpt_satisfied}
Let $p : \Omega \longrightarrow \R_+$ satisfying Assumption \eqref{simplifyingAssumption}. Let $\phi : \Omega \rightarrow \R$ in $H(\alpha,\mu L)$ for some constant $\mu>0$ and such that $|\phi| \leq \lambda p$ over $\Omega$ for some constant $\lambda > 0$. Then $$p+\phi \in \mathcal{P}\big(\alpha, \, (1+\lambda \lor \mu)L, \, \frac{\cstar + 2\lambda + \lambda\cstar}{1-\lambda}\big).$$
\end{lemma}

\begin{proof}[Proof of Lemma \ref{simplifying_assumpt_satisfied}]
Clearly, $p+\phi \in H(\alpha, (1+\mu)L) \subset H(\alpha, (1+\mu \lor \lambda)L)$. Now, let $x,y \in \Omega$. By Assumption \eqref{simplifyingAssumption} and the triangular inequality, we have:
\begin{align*}
    |p(x) + \phi(x) - p(y) - \phi(y)| & \leq |p(x)-p(y)| + |\phi(x)| + |\phi(y)|\\
    & \leq \cstar p(x) + L\|x-y\|^\alpha + 2\lambda p(x) + \lambda [p(y)-p(x)]\\
    & \leq (\cstar + 2\lambda) p(x) + L\|x-y\|^\alpha + \lambda(\cstar p(x) + L\|x-y\|^\alpha)\\
    & \leq (\cstar+2\lambda+\lambda\cstar)p(x) + (1+\lambda) L\|x-y\|^\alpha\\
    & \leq \frac{\cstar+2\lambda+\lambda\cstar}{1-\lambda}[p(x) + \phi(x)] + (1+\lambda \lor \mu) L\|x-y\|^\alpha.
\end{align*}
\end{proof}

\begin{lemma}\label{int_phij4_hd_p02}
There exist two  constants $\cphi, \Cphi >0$ such that for all $j=1, \dots, N$:
$$\Big(\int_\Omega \frac{\phi_j^2}{p_0}\Big)^2 \leq \Ctilde \, {\Cphi}^4 L^4 \int_{B_j} \frac{\hB^{~~4\alpha + d}}{p_0^2},$$
where $\Ctilde$ is a constant given in the proof.
\end{lemma}

\begin{proof}[Proof of Lemma \ref{int_phij4_hd_p02}]
Recall that $\phi_j$ is supported on $B_j$. By Lemma \ref{p0_const_B(uA/2)}, 
\begin{align}
    \Big(\int_\Omega \frac{\phi_j^2}{p_0}\Big)^2 &= \Big(\int_{B_j} \frac{\phi_j^2}{p_0}\Big)^2 \leq \int_{B_j} \frac{\phi_j^4}{p_0^2} \hB^d \cdot \int_{B_j} \frac{1}{\hB^d} ~~ \text{ by Cauchy-Schwarz' inequality}\nonumber \\
    & \leq \frac{h_j^d}{p_0(x_j)^2} \frac{{\CL{\ref{p0_const_B(uA/2)}}^{(h)}}^d}{c^{\,d}\,{\cL{\ref{p0_const_B(uA/2)}}^{(p_0)}}^2} \int_{B_j}\phi_j^4 \times \frac{1}{{\cL{\ref{p0_const_B(uA/2)}}^{(h)}}^d} \frac{c^d}{h_j^d} |B_j|\nonumber \\
    & \leq \frac{h_j^d}{p_0(x_j)^2} \frac{{\CL{\ref{p0_const_B(uA/2)}}^{(h)}}^d}{{\cL{\ref{p0_const_B(uA/2)}}^{(p_0)}}^2 {\cL{\ref{p0_const_B(uA/2)}}^{(h)}}^d} \int_{B_j}\phi_j^4. \label{int_phi4}
\end{align}
Moreover, by the change of variable $y = (x-x_j)/h_j$ we have
\begin{align*}
    \int_{B_j}\phi_j^4 &= 2 \int_{\R^d} \Big\{\Cphi L h_j^\alpha~ f(y) \Big\}^4  h_j^d dy = 2 \big(\Cphi L\big)^4 \|f\|_4^4 ~ h_j^{4\alpha + d}.
\end{align*}

Injecting into \eqref{int_phi4} we get:
\begin{align*}
    \int_\Omega \frac{\phi_j^2}{p_0} &\leq 2 {\Cphi}^4 \frac{{\CL{\ref{p0_const_B(uA/2)}}^{(h)}}^d}{{\cL{\ref{p0_const_B(uA/2)}}^{(p_0)}}^2 {\cL{\ref{p0_const_B(uA/2)}}^{(h)}}^d} \cdot \|f\|_4^4 ~ L^4  \frac{h_j^{4\alpha + d}}{p_0(x_j)^2} h_j^d\\
    & \leq 2 {\Cphi}^4 \frac{{\CL{\ref{p0_const_B(uA/2)}}^{(h)}}^d}{{\cL{\ref{p0_const_B(uA/2)}}^{(p_0)}}^2 {\cL{\ref{p0_const_B(uA/2)}}^{(h)}}^d} \cdot \|f\|_4^4 ~ L^4 \int_{B_j} \frac{\hB(x)^{4\alpha + d}}{p_0^2(x)}dx \; \frac{c^{4\alpha + d}{\CL{\ref{p0_const_B(uA/2)}}^{(p_0)}}^2}{{\cL{\ref{p0_const_B(uA/2)}}^{(h)}}^{4\alpha + d}}\\
    &=: {\Cphi}^4 \Ctilde L^4 \int_{B_j} \frac{\hB^{~~4\alpha + d}}{p_0^2}.
\end{align*}
\end{proof}

\section{Upper bound in the tail regime}\label{UB_tail_regime}

In the tail regime, we show that the combination of the tests $\psi_1$ and $\psi_2$ has both type-I and type-II errors upper bounded by $\eta/4$ when $\left(\int_{\T}|p-p_0|^t\right)^{1/t} \geq \ThreshT \rho^*$ for some constant $\ThreshT$. We defer to Subsection \ref{Technical_results_UB_Tail} the technical results needed for proving this upper bound. We recall that $\mathcal{T} = \T(\uA)$ but that we place ourselves over a covering of $\T(\tuA)$. \textbf{Until the end of the proof, whenever no ambiguity arises, we drop the indexation in $\bigcup_{j=1}^M \widetilde C_j$ and only write $\|p_0\|_1, \|p\|_1, \|\Delta\|_1$ to denote $\int_{\bigcup_{j=1}^M \widetilde C_j} p_0,\; \int_{\bigcup_{j=1}^M \widetilde C_j} p$, and $ \int_{\bigcup_{j=1}^M \widetilde C_j} |\Delta|$ respectively. Moreover, in Appendix \ref{UB_tail_regime} only, we will write $h$ for $\hT(\uA)$ when the tail dominates and $h = \hm$ when the bulk dominates.}\\

\subsection{Under $H_0$}

We here prove that $\psi_1 \lor \psi_2$ has a type-I error upper bounded by $\eta/2$, no matter whether the bulk or the tail dominates. \\

By Lemma \ref{proba_observed_twice}, the type-I error of $\psi_2$ is upper bounded by 
\begin{equation}\label{ErrI_psi2}
    \mathbb{P}_{p_0}(\psi_2 = 1) \leq n^2 h^d \int_{\bigcup_{j \in \Ibar} \widetilde C_j} p_0^2 \leq \CL{\ref{proba_observed_twice}} \leq \frac{\eta}{8} ~~ \text{ taking  $\CL{\ref{proba_observed_twice}}$ small enough}.
\end{equation}

\hfill

As to the type-I error of $\psi_1$, we have under $H_0$:
\begin{align*}
    &\mathbb{E}\Big[\sum_{j \in \Ibar} \frac{N_j}{n}\Big] = \|p_0\|_1 ~~~~ \text{ and } ~~~~ \mathbb{V}\Big[\sum_{j \in \Ibar} \frac{N_j}{n}\Big] \leq \frac{\|p_0\|_1}{n}.
\end{align*}

Recalling that we write $\|p_0\|_1$ for $\int_{\bigcup_{j \in \Ibar} \widetilde C_j} p_0$, we therefore have by Chebyshev's inequality:
\begin{equation}\label{ErrI_psi1}
    \mathbb{P}_{p_0}\left[\Big|\sum_{j \in \Ibar} \frac{N_j}{n} - \|p_0\|_1 \Big| > \CPsiOne \sqrt{\frac{\|p_0\|_1}{n}}\right] \leq \frac{\eta}{4}.
\end{equation}
for $\CPsiOne = 2\sqrt{2}/ \sqrt{\eta}$. Combining \eqref{ErrI_psi2} and \eqref{ErrI_psi1}, we conclude that $\psi_1 \lor \psi_2$ has type-I error upper bounded by $\eta/4$.\\

\boundedcase{\textbf{Remark:} Note that by Equation \eqref{ErrI_psi2}, we cannot have $\hT>\homega$ when the tail dominates. Otherwise, the family of cells $(\widetilde C_j)_j$ would only consist of one cube $\widetilde C_1$, covering $\Omega$. Even under $H_0$, this cell would clearly contain $n$ observations, contradicting \eqref{ErrI_psi2}.}

\subsection{Under the alternative when the tail dominates}

We now prove that when the tail dominates, $\rhot + \rhor$ is an upper bound on the minimax separation radius. To do so, we show that when $p$ is such that $\int_{\T} |p-p_0|^t \geq \ThreshT \left( {\rhot}^t + \rhor^t \right)$, one of the two tests $\psi_1$ or $\psi_2$ rejects $H_0$, \textit{whp}. Fix a density $p$ satisfying:
\begin{align}
    \int_{\T} |p-p_0|^t& \geq \ThreshT {\rhot}^t \geq \widetilde \ThreshT \Big(\frac{L^d}{n^{2\alpha}}\Big)^\frac{t-1}{\alpha + d} \left(\int_{\T}p_{0} + \frac{1}{n}\right)^\frac{(2-t)\alpha+d}{\alpha + d}\nonumber \\
    & \geq \widetilde \ThreshT \left(\int_{\T} p_0 + \frac{1}{n}\right)^{2-t} \left[\frac{L^d}{n^{2\alpha}} \Big(\int_{\T} p_0 \Big)^d \right]^\frac{t-1}{\alpha+d}. \label{Assumption_H1_tail}
\end{align}
where $\widetilde \ThreshT = \ThreshT/2^\frac{(2-t)\alpha+d}{\alpha + d}$ and $\ThreshT$ is a large enough constant.

Setting $u = 2-t$ and $v = t-1$ satisfying $u+v = 1$ and $u+2v = t$, we have by Hölder's inequality:
\begin{align*}
     \int_{\T} |p-p_0|^t &= \int_{\T} |p-p_0|^{u+2v} \leq \Big[\int_{\T} |p-p_0|\; \Big]^u \Big[\int_{\T} |p-p_0|^2\; \Big]^v \\
     &\leq \Big[\int_{\bigcup_{j=1}^M \widetilde C_j} |p-p_0|\; \Big]^u \Big[\int_{\bigcup_{j=1}^M \widetilde C_j} |p-p_0|^2\; \Big]^v.
\end{align*}

Then by \eqref{Assumption_H1_tail}, one of the following two inequalities must hold:
\begin{align*}
    &(\text{i})~~~ \int_{\bigcup_{j=1}^M \widetilde C_j} |p-p_0| ~\geq ~ \ThreshT_1 \left(\int_{\T} p_0 + \frac{1}{n}\right)\\
    &(\text{ii})~~ \int_{\bigcup_{j=1}^M \widetilde C_j} |p-p_0|^2 \geq ~ \ThreshT_2 \left[\frac{L^d}{n^{2\alpha}} \Big(\int_{\T} p_0 \Big)^d \right]^\frac{1}{\alpha+d} = \; \frac{\ThreshT_2}{n^2 \hT^d(\uA)}.
\end{align*}
where $\ThreshT_1$ and $\ThreshT_2$ are two constants given in the proof, such that $\ThreshT_1\ThreshT_2 = \widetilde \ThreshT$.\\

\underline{\textbf{First case:}} Suppose that
$ (\text{i})$ holds, i.e. $ \|\Delta\|_1 \geq  \ThreshT_1 \left(\int_{\T} p_0 + \frac{1}{n}\right)$. We then have \\ $\|\Delta\|_1 \geq \ThreshT_1\left(\frac{1}{n}+ \frac{1}{1+\CL{\ref{T(2uA)Moment1}}} \int_{\T(2\uA)} p_0 \right)\geq \CL{\ref{Erreur_II_psi1}}(\|p_0\|_1 + 1/n)$ for $\ThreshT_1$ large enough. Therefore, by Lemma \ref{Erreur_II_psi1}, we have $\mathbb{P}_p(\psi_1 = 0) \leq \frac{\eta}{8}$.\\

\hfill

\underline{\textbf{Second case:}} Suppose (i) does not hold. Then (ii) holds. By Lemma \ref{link_2ndMoment__sumSquares}, we can write:
$$ \int_{\bigcup_{j=1}^M \widetilde C_j} (p-p_0)^2 \leq \frac{\AL{\ref{link_2ndMoment__sumSquares}}}{h^d}\sum_{j=1}^M \Big( \int_{\widetilde C_j} p  \Big)^2 +  \frac{\BL{\ref{link_2ndMoment__sumSquares}}}{n^2h^d} + \CL{\ref{link_2ndMoment__sumSquares}} L h^\alpha \int_{\bigcup_{j=1}^M \widetilde C_j} |p-p_0|,$$ 
Since (i) does not hold we can further upper bound this expression as:
\begin{align*}
    \int_{\bigcup_{j=1}^M \widetilde C_j} (p-p_0)^2 & \leq \frac{\AL{\ref{link_2ndMoment__sumSquares}}}{h^d}\sum_{j=1}^M \Big( \int_{\widetilde C_j} p \Big)^2 + \frac{\BL{\ref{link_2ndMoment__sumSquares}}}{n^2h^d} + \CL{\ref{link_2ndMoment__sumSquares}} L h^\alpha \cdot \ThreshT_1 \int_{\T} p_0\\
    & \leq \frac{\AL{\ref{link_2ndMoment__sumSquares}}}{h^d}\sum_{j=1}^M \Big( \int_{\widetilde C_j} p  \Big)^2 + \frac{\BL{\ref{link_2ndMoment__sumSquares}}}{n^2h^d} + \CL{\ref{link_2ndMoment__sumSquares}} \frac{\ThreshT_1}{n^2 h^d}.
\end{align*}
By (ii) we therefore have:
\begin{align*}
    \frac{\AL{\ref{link_2ndMoment__sumSquares}}}{h^d}\sum_{j=1}^M \Big( \int_{\widetilde C_j} p \Big)^2 +  \frac{\BL{\ref{link_2ndMoment__sumSquares}}}{n^2h^d} + \frac{\CL{\ref{link_2ndMoment__sumSquares}} \ThreshT_1 }{n^2 h^d} \geq \frac{\ThreshT_2}{n^2 h^d}
\end{align*}
hence:
\vspace{-9mm}

\begin{align}
    &\frac{\AL{\ref{link_2ndMoment__sumSquares}}}{h^d}\sum_{j=1}^M \Big( \int_{\widetilde C_j} p  \Big)^2  \geq \frac{\ThreshT_2 - \CL{\ref{link_2ndMoment__sumSquares}} \ThreshT_1 - \BL{\ref{link_2ndMoment__sumSquares}}}{n^2 h^d} =:  \frac{\AL{\ref{link_2ndMoment__sumSquares}}\ThreshT_3}{n^2 h^d} \nonumber \\
    & \nonumber \\
    \text{i.e. } ~~~~ & \sum_{j=1}^M \Big( \int_{\widetilde C_j} p  \Big)^2 \geq \frac{\ThreshT_3}{n^2} \label{sum_observed_twice}
\end{align}

Taking $\ThreshT_2$ large enough ensures that $\ThreshT_3 \geq \CL{\ref{Erreur_II_psi2}}$, so that by Lemma \ref{Erreur_II_psi2}, we have $\mathbb{P}_p(\psi_2=0) \leq \frac{\eta}{8}$.



\subsection{Under $H_1(\ThreshT \rhob)$ when the bulk dominates}

We now suppose that $\CBT\rhob \geq \rhot$. Moreover, we suppose $\int_{\Omega} |p-p_0|^t\geq \ThreshT\left(\rhob^t + \rhor^t \right)$ for $\ThreshT$ large enough. If $\int_{\B(\uA/2)} |p-p_0|^t\geq \frac{\ThreshT}{2} \left(\rhob^t + \rhor^t \right)$, then for $\ThreshT$ large enough, $\mathbb{P}_p(\psib =0)\leq \frac{\eta}{4}$ by the analysis of the upper bound. Therefore, \textit{wlog}, suppose that $\int_{\B(\uA/2)} |p-p_0|^t\leq \frac{\ThreshT}{2}\left(\rhob^t + \rhor^t \right)$, hence that $\int_{\T(\uA/2)} |p-p_0|^t\geq \frac{\ThreshT}{2} \left(\rhob^t + \rhor^t \right)$.\\

Assume first that $\|p_0\|_1 \leq \frac{1}{n}$. Then we have $\|p\|_t \geq \left( \frac{\ThreshT}{2}\right)^{1/t} \rhor - \|p_0\|_t \geq \left(\frac{\ThreshT}{2}-\AL{\ref{control_norm_t_Rd}}(1)\right) \rhor$ by Lemma \ref{control_norm_t_Rd}. Taking $\ThreshT$ large enough imposes $\|p\|_1 \geq 2  \frac{\CL{\ref{Erreur_II_psi1}}}{n} \geq \CL{\ref{Erreur_II_psi1}} \left(\frac{1}{n} + \|p_0\|_1\right)$ hence $\mathbb{P}_p(\psi_1 = 0) \leq \frac{\eta}{8}$ by Lemma \ref{Erreur_II_psi1}.\\

Now, in the remaining of the proof, assume $\|p_0\|_1 > \frac{1}{n}$. Again, there are two cases.\\

\underline{\textbf{First case:}} If $\|\Delta\|_1 \geq 2 \CL{\ref{Erreur_II_psi1}}\|p_0\|_1 \geq \CL{\ref{Erreur_II_psi1}}(\|p_0\|_1+\frac{1}{n})$, then by Lemma \ref{Erreur_II_psi1}: $\mathbb{P}_p(\psi_1 = 0) \leq \frac{\eta}{8}$. \\

\underline{\textbf{Second case:}} Assume now that $\|\Delta\|_1 \leq 2 \CL{\ref{Erreur_II_psi1}}\|p_0\|_1$, hence that $\|p\|_1 \leq (2\CL{\ref{Erreur_II_psi1}} + 1)\|p_0\|_1$. By Assumption \eqref{simplifyingAssumption}, the definition of $\hm$ from \eqref{def_hm} and the choice of $\cm$, we can immediately check that $\bigcup\limits_{j \in \Ibar} \widetilde C_j \subset \T(\uA)$. Hence $\|p\|_1 \leq (2\CL{\ref{Erreur_II_psi1}} + 1)\int_{\T(\uA)} p_0 = (2\CL{\ref{Erreur_II_psi1}} + 1)\frac{1}{L\hT^{\alpha+d}}$. We can now lower bound $\sum\limits_{j \in \Ibar} n^2 q_j^2$ using Lemma \ref{link_2ndMoment__sumSquares}:
\begin{align}
    \int_{\bigcup_{j=1}^M \widetilde C_j} (p-p_0)^2 &\leq \frac{\AL{\ref{link_2ndMoment__sumSquares}}}{\hm^d}\sum_{j \in \Ibar} \Big( \int_{\widetilde C_j} p  \Big)^2 +  \frac{\BL{\ref{link_2ndMoment__sumSquares}}}{n^2\hm^d} + \CL{\ref{link_2ndMoment__sumSquares}}(2\CL{\ref{Erreur_II_psi1}} + 1)  \frac{L \hm^\alpha}{L\hT^{\alpha+d}}\noindent\\
    & \leq \frac{\AL{\ref{link_2ndMoment__sumSquares}}}{\hm^d}\sum_{j \in \Ibar} \Big( \int_{\widetilde C_j} p  \Big)^2 +  \left(\BL{\ref{link_2ndMoment__sumSquares}} + \frac{ \CL{\ref{link_2ndMoment__sumSquares}}(2\CL{\ref{Erreur_II_psi1}} + 1) }{{\CBT^{(2)}}^{\alpha + d}}\right)\frac{1}{n^2\hm^d} ~~ \text{ by Lemma \ref{htail_leq_hbulk}}\noindent \\
    & =: \frac{\Apsitwo}{\hm^d}\sum_{j \in \Ibar} \Big( \int_{\widetilde C_j} p  \Big)^2 +  \frac{\Bpsitwo}{n^2\hm^d} \label{sum_qj2_geq_moment_2}
\end{align}
where $\Apsitwo$ and $\Bpsitwo$ are two constants. We recall that \textbf{in this section,} we respectively denote by $\|\Delta\|_2^2$ and $\|\Delta\|_t^t$ the quantities $\int_{\bigcup_{j=1}^M \widetilde C_j} (p-p_0)^2$ and $\int_{\bigcup_{j=1}^M \widetilde C_j} |p-p_0|^t$. We now lower bound the term $\|\Delta\|_2^2$. By Hölder's inequality:
\begin{align*}
    \|\Delta\|_2^2 &\geq \left(\|\Delta\|_t^t\|\Delta\|_1^{t-2}\right)^\frac{1}{t-1} \geq \left(\frac{\ThreshT}{2} \rhob^t \big\{(2\CL{\ref{Erreur_II_psi1}}+1)\|p_0\|_1\big\}^{t-2}\right)^{\frac{1}{t-1}}\\
    & = \left(\frac{\ThreshT}{2} \rhob^t \left\{(2\CL{\ref{Erreur_II_psi1}}+1)\left(\frac{\rhot^t}{\TL^\frac{t-1}{\alpha+d}}\right)^\frac{\alpha+d}{(2-t)\alpha+d}\right\}^{t-2}\right)^{\frac{1}{t-1}}\\
    &\geq \left(\frac{\ThreshT}{2} \rhob^t \left\{(2\CL{\ref{Erreur_II_psi1}}+1)\left(\frac{\CBT^t\rhob^t}{\TL^\frac{t-1}{\alpha+d}}\right)^\frac{\alpha+d}{(2-t)\alpha+d}\right\}^{t-2}\right)^{\frac{1}{t-1}}  \text{ recalling $t-2 \leq 0$}\\
    &=: \CDelta \rhob^\frac{td}{(2-t)\alpha+d}\TL^\frac{2-t}{(2-t)\alpha+d},
\end{align*}
where $\CDelta$ is a constant \llarge{$\ThreshT$}. We therefore have $\hm^d\|\Delta\|_2^2 \geq \cm \CDelta$, hence combining with equation \eqref{sum_qj2_geq_moment_2}, we get 
\begin{align*}
    n^2\sum_{j \in \Ibar} q_j^2 \geq \frac{1}{\Apsitwo} \left(\cm \CDelta - \Bpsitwo\right) \geq \CL{\ref{Erreur_II_psi2}}, 
\end{align*}
by choosing $\ThreshT$ large enough, which yields $\mathbb{P}_p(\psi_2 =0) \leq \frac{\eta}{8}$.

\subsection{Technical results}\label{Technical_results_UB_Tail}

\begin{lemma}\label{proba_observed_twice}
The following result holds no matter whether the bulk or the tail dominates. Under $H_0$, the probability that at least one of the cells $(\widetilde C_j)_{j=1,\dots,M}$ contains at least two observations is upper bounded as
\begin{align*}
    \mathbb{P}_{p_0}[\, \exists j \in \Ibar: N_j \geq 2\, ] \leq n^2 h^d \int_{\bigcup_{j \in \Ibar} \widetilde C_j} p_0^2 \leq \CL{\ref{proba_observed_twice}},
\end{align*}
where $\CL{\ref{proba_observed_twice}}$ is a constant which \slarge{$\CBT$}.  
\end{lemma}

\begin{proof}[Proof of Lemma \ref{proba_observed_twice}]
We place ourselves under $H_0$. For all ${j \in \Ibar}$, let $p_j = \int_{\widetilde C_j} p_0$. By the definition of $N_j = \sum_{i=1}^n \mathbb{1}\{X_i \in \widetilde C_j\}$, we have $N_j \sim Bin(p_j,n)$ for all $j=1, \dots, M$. Therefore the probability that for a fixed $j$ we have $N_j\geq 2$ is upper bounded as:
\begin{align*}
    1 - (1-p_j)^n - np_j(1-p_j)^{n-1} \leq 1 - (1-np_j) - np_j[1-(n-1)p_j] \leq n^2p_j^2.
\end{align*}
The probability that at least one of the $N_j$ is at least $2$ is therefore upper bounded by $\sum_{j \in \Ibar} n^2p_j^2$.

Now, by the Cauchy-Schwarz inequality:
\begin{align*}
    \sum_{j \in \Ibar} n^2p_j^2 = \sum_{j \in \Ibar} n^2 \left(\int_{\widetilde C_j} p_0 \right)^2 \leq \sum_{j \in \Ibar} n^2 h^d \int_{\widetilde C_j} p_0^2 = n^2 h^d \int_{\bigcup_{j \in \Ibar} \widetilde C_j} p_0^2. 
\end{align*}
If the tail dominates, then Lemma \ref{T(2uA)Moment2} proves that the last quantity is at most $\CL{\ref{T(2uA)Moment2}}$. Otherwise, since $\bigcup_{j \in \Ibar} \widetilde C_j \subset \T(\uA)$, the RHS can be further upper bounded by Lemma \ref{int_p0squared_uA_small} as $ \frac{\Cmom n^2 h^d}{n^2 \hT^d(\uA)} \leq \Cmom (\CBT^{(2)})^d$ by Lemma \ref{htail_leq_hbulk}. In both cases, the constant upper bounding the RHS \ssmall{$\Cmom$}.
\end{proof}

\begin{lemma}\label{T(2uA)}
If the tail dominates, i.e. if $\rhot \geq \CBT \rhob$ where $\CBT$ is defined in Lemma \ref{htail_leq_hbulk}, then it holds that $\bigcup_{j \in \Ibar}\widetilde C_j \subset \T(2\uA)$, provided that $\CBT$ is larger than a constant.
\end{lemma}

\begin{proof}[Proof of Lemma \ref{T(2uA)}]
Let $y\in \bigcup_{j \in \Ibar}\widetilde C_j$ and $x \in \T(\uA)$ such that $x$ and $y$ belong to the same cell $\widetilde C_j$. By Assumption \eqref{simplifyingAssumption} and Lemma \ref{htail_leq_hbulk} we have:
\begin{align*}
   p_0(y) \leq (1+\cstar)p_0(x) + L(h \sqrt{d})^\alpha \leq (1+\cstar)p_0(x) + L\sqrt{d}^\alpha \inf_{x \in \B} \hB(x)^\alpha \leq 2\uA.
\end{align*}
\end{proof}

\begin{lemma}\label{sum_to_norm}
Recall that $\|\Delta\|_1 = \int_{\T(\uA)} |\Delta|$ and $\|p_0\|_1 = \int_{\T(\uA)} |p_0|$. If $\|\Delta\|_1 \geq 3\|p_0\|_1$, then: $$\left| \int_{\T(\uA)} \Delta \right| \geq \frac{1}{2} \|\Delta\|_1.$$
\end{lemma}

\hfill

\begin{proof}[Proof of Lemma \ref{sum_to_norm}]
Define $J_+ = \{x \in \mathcal{T}: p(x) \geq p_0(x)\}$ and $J_- = \{x \in \mathcal{T}: p(x) < p_0(x)\}$. Define also:
$$ s = \frac{\int_{\mathcal{T}} \Delta }{\int_{\mathcal{T}} p_0}, ~~~~~~ s_+ = \frac{\int_{J_+} \Delta}{\int_{\mathcal{T}} p_0}, ~~~~~~ s_- = - \frac{\int_{J_-} \Delta }{\int_{\mathcal{T}} p_0} $$

Then by assumption: $s_+-s_- = s\geq 3$. Moreover, $s_- =  \frac{\int_{J_-} p_0 - p}{\int_{\mathcal{T}} p_0}  \leq 1$. Thus, $s_+ \geq3 \geq 3s_-$ so that $2(s_+-s_-) \geq s_++s_-$, which yields the result.
\end{proof}


\begin{lemma}\label{Erreur_II_psi1}
The  following result holds no matter whether the bulk or the tail dominates. There exists a constant $\CL{\ref{Erreur_II_psi1}}$ such that, whenever $\|\Delta\|_1 \geq \CL{\ref{Erreur_II_psi1}} (\|p_0\|_1+1/n)$, then $\mathbb{P}_p(\psi_1 = 0) \leq \frac{\eta}{8}$.
\end{lemma}

\begin{proof}[Proof of Lemma \ref{Erreur_II_psi1}]

Choose $\CL{\ref{Erreur_II_psi1}} \geq 10$ and $\ctail \geq 1$ so that by 
the triangular inequality and recalling $\int_{\T(\uA)} p_0 \geq \frac{\ctail}{n}$ we have: $\|p\|_1 + \|p_0\|_1 \geq 5\|p_0\|_1$, hence $\|\Delta\|_1 \geq \int p - \int p_0 \geq 3 \|p_0\|_1$. Therefore, the assumptions of Lemma \ref{sum_to_norm} are met.
    \begin{align*}
        \mathbb{P}_p(\psi_1 = 0) &= \mathbb{P}_p\Big(\big|\sum_{j \in \Ibar} \frac{N_j}{n} - \|p_0\|_1\big| \leq \CPsiOne \sqrt{\frac{\|p_0\|_1}{n}}\Big)\\
        &\leq \mathbb{P}_p\Big(\big|\int_{\bigcup_{j \in \Ibar} \widetilde C_j} p-p_0\big| - \big|\sum_{j \in \Ibar} \frac{N_j}{n} - \|p\|_1\big| \leq \CPsiOne \sqrt{\frac{\|p_0\|_1}{n}}\Big) \text{  by the triangular inequality}\\
        & \leq \mathbb{P}_p\Big(\frac{1}{2}\|\Delta\|_1 -   \CPsiOne \sqrt{\frac{\|p_0\|_1}{n}}\leq  \big|\sum_{j \in \Ibar} \frac{N_j}{n} - \|p\|_1\big| \Big) ~~ \text{ by Lemma \ref{sum_to_norm}}\\
        & \leq \frac{\frac{1}{n} \|p\|_1}{\left(\frac{1}{2}\|\Delta\|_1 -  \CPsiOne \sqrt{\frac{\|p_0\|_1}{n}}\right)^2} ~~ \text{ by Chebyshev's inequality}\\
        & \leq \frac{ \|p\|_1/n}{\left(\frac{1}{2}\|p\|_1 - \frac{1}{2}\|p_0\|_1 -   \CPsiOne \sqrt{\frac{\|p_0\|_1}{n}}\right)^2} \text{ by the triangular inequality}\\
        & \leq \frac{ \|p\|_1/n}{\left(\frac{1}{2}\|p\|_1 - \frac{1}{2}\|p_0\|_1 -  \CPsiOne (\|p_0\|_1 + 1/n)\right)^2} \text{ using } \sqrt{xy} \leq x + y\\
        & \leq \frac{ \|p\|_1/n}{\left(\frac{1}{2}\|p\|_1 -  (\CPsiOne + 1) (\|p_0\|_1 + 1/n)\right)^2}.
    \end{align*}
    Choose $\CL{\ref{Erreur_II_psi1}}  \geq 4(\CPsiOne + 1) +1$, so that the quantity $\frac{1}{2}\|p\|_1 -  (\CPsiOne + 1) (\|p_0\|_1 + 1/n)$ is strictly positive. This ensures that all of the above operations are valid. Now set $z = (\CPsiOne + 1) (\|p\|_1 + 1/n)$. The function $f: x \mapsto \frac{ x}{n\left(x/2 -  z\right)^2}$ is decreasing over $(2z, \infty)$. 
    For $x \geq 20z/ \eta$, since $nz > 1$ and $\eta \leq 1 $, we have:
    \begin{align*}
        f(x) \leq \frac{20z/ \eta}{n (10z / \eta - z)^2} = \frac{20 \eta}{nz (10-\eta)^2} \leq \frac{20\eta}{81} \leq \eta/4.
    \end{align*}
    
    which proves that, whenever $\|p\|_1 \geq \frac{20}{\eta}(\CPsiOne + 1) (\|p_0\|_1 + 1/n)$, we have $\mathbb{P}_p(\psi_1=0) \leq \eta/4$. This condition is guaranteed whenever $\|\Delta\|_1 \geq \big(1+\frac{20}{\eta}(\CPsiOne + 1) \big)(\|p_0\|_1 + 1/n) = \CL{\ref{Erreur_II_psi1}} (\|p_0\|_1 + 1/n)$ for $\CL{\ref{Erreur_II_psi1}}  = 1+ \frac{20}{\eta}(\CPsiOne + 1)$.
\end{proof}

\begin{lemma}\label{link_2ndMoment__sumSquares}
We have: 
$$ \int_{\bigcup_{j \in \Ibar} \widetilde C_j} (p-p_0)^2 \; \leq \;\; \frac{\AL{\ref{link_2ndMoment__sumSquares}}}{h^d}\sum_{j \in \Ibar} \left(\int_{C_j} p \right)^2  +\;  \frac{\BL{\ref{link_2ndMoment__sumSquares}}}{n^2h^d} \; + \; \CL{\ref{link_2ndMoment__sumSquares}} L h^\alpha \int_{\bigcup_{j \in \Ibar} \widetilde C_j} |p-p_0|,$$
where $\AL{\ref{link_2ndMoment__sumSquares}}, \BL{\ref{link_2ndMoment__sumSquares}}, \CL{\ref{link_2ndMoment__sumSquares}}$ are constants given in the proof.
\end{lemma}

\begin{proof}[Proof of Lemma \ref{link_2ndMoment__sumSquares}]

Let $j\in\{1, \dots, M\}$. Assume that each cube $\widetilde C_j$ is centered at $x_j$. By Assumption \eqref{simplifyingAssumption} we have for all $j=1, \dots, M$ and $x \in \widetilde C_j$:
\begin{align}
    p(x) &\leq (1+\cstar)p(x_j) + L(h\sqrt{d})^\alpha, \label{p(x)_leq_p(xj)}  
\end{align}
hence by exchanging $x$ and $x_j$ and integrating:
\begin{align}
    p(x_j) & \leq \frac{1+\cstar}{h^d} \int_{\widetilde C_j}p + L(h\sqrt{d})^\alpha \label{p(xj)_leq_p(x)},
\end{align}
and by equations \eqref{p(x)_leq_p(xj)} and  \eqref{p(xj)_leq_p(x)}, we have:
\begin{align}
    p(x) &\leq \frac{(1+\cstar)^2}{h^d} \int_{\widetilde C_j}p + (2+ \cstar)L(h\sqrt{d})^\alpha. \label{p(x)_leq_int_p}
\end{align}

Therefore, fixing any ${j \in \Ibar}$ it holds that:
\begin{align}
    \int_{\widetilde C_j} p^2 & \leq \int_{\widetilde C_j} p(x)dx \left[\frac{(1+\cstar)^2}{h^d} \int_{\widetilde C_j}p + (2+ \cstar)L(h\sqrt{d})^\alpha \right]\nonumber \\
    & = \frac{(1+\cstar)^2}{h^d} \left(\int_{\widetilde C_j} p \right)^2 + (2+ \cstar)L(h\sqrt{d})^\alpha \int_{\widetilde C_j}p. \label{int_p2_leq_int_p_2}
\end{align}

Now, we have for all ${j \in \Ibar}$:
\begin{align*}
    \int_{\widetilde C_j} (p-p_0)^2 \leq 2 \int_{\widetilde C_j} p^2 + 2 \int_{\widetilde C_j} p_0^2 \leq 2\frac{(1+\cstar)^2}{h^d} \left(\int_{\widetilde C_j} p \right)^2 + 2(2+ \cstar)L(h\sqrt{d})^\alpha \int_{\widetilde C_j}p + 2 \int_{\widetilde C_j} p_0^2,
\end{align*}
and summing for ${j \in \Ibar}$:
\begin{align*}
    \int_{\bigcup_{j \in \Ibar} \widetilde C_j} (p-p_0)^2 \leq \AL{\ref{link_2ndMoment__sumSquares}} \sum_{j \in \Ibar} \left(\int_{\widetilde C_j} p\right)^2 + \CL{\ref{link_2ndMoment__sumSquares}} Lh^\alpha \|p\|_1 + \frac{\CL{\ref{T(2uA)Moment2}}}{n^2h^d}.
\end{align*}

Now, 
\begin{align*}
    Lh^\alpha \|p\|_1 \leq Lh^\alpha \left( \|\Delta\|_1 + \|p_0\|_1\right) \leq Lh^\alpha \|\Delta\|_1 + \frac{\CL{\ref{T(2uA)Moment1}}}{n^2h^d}.
\end{align*}
Therefore, setting $\BL{\ref{link_2ndMoment__sumSquares}} = \CL{\ref{T(2uA)Moment1}} + \CL{\ref{T(2uA)Moment2}}$ yields the result.
\end{proof}

\begin{lemma}\label{Erreur_II_psi2}
The following result holds no matter whether the bulk or the tail dominates. Assume that $\sum_j n^2 q_j^2 \geq \CL{\ref{Erreur_II_psi2}}$ where $\CL{\ref{Erreur_II_psi2}}$ is a large constant and $q_j = \int_{\widetilde C_j} p$ for all $j$. Then $\mathbb{P}_p(\psi_2 = 0) \leq \frac{\eta}{8}$.
\end{lemma}

\begin{proof}
We draw $\widetilde k \sim \Poi(k)$ where we recall that $k = \frac{n}{2}$. We consider the setting where we observe $\widetilde X_1, \dots, \widetilde X_{\widetilde k}$ iid drawn from the density $p$ and define $\forall j \in \Ibar, N_j' = \sum_{i=1}^{\widetilde k} \mathbb{1}_{\widetilde X_i = j}$ the histogram of the tail in this modified setting. We recall that by the classical poissonization trick, the random variables $(N'_j)_j$ are independent and distributed as $\Poi(kq_j)$ respectively. We first notice that
\begin{align}
    \mathbb{P}_{p^{\otimes \widetilde k}}(\forall j \in \Ibar: N_j' = 0 \text{ or } 1) &\geq \mathbb{P}_{p^{\otimes \widetilde k}}(\forall j \in \Ibar: N_j' = 0 \text{ or } 1| \widetilde k \leq n)\mathbb{P}(\widetilde k \leq n)\nonumber \\
    & \geq \mathbb{P}_{p^{\otimes n}}(\forall j \in \Ibar: N_j' = 0 \text{ or } 1)\mathbb{P}(\widetilde k \leq n)\label{histogram_poissonized}
\end{align}
Moreover, 
\begin{align*}
    &\mathbb{P}_{p^{\otimes \widetilde k}}(\forall j\in \Ibar : N_j = 0 \text{ or } N_j = 1)  \; = \prod_{j \in \Ibar} e^{-kq_j}\left(1+kq_j\right).
\end{align*}
    Let $I_{-} = \{j\in \Ibar: kq_j \leq \frac{1}{2}\}$ and $I_{+} = \{j\in \Ibar: kq_j > \frac{1}{2} \}$. Recall that for $x\in (0,1/2], \; \log(1+x) \leq x - x^2/3$. Then, for $j \in I_-$: 
    \begin{align*}
        e^{-kq_j}\left(1+kq_j\right) &= \exp\left\{-kq_j +\log(1+kq_j)\right\} \leq \exp\left(- \frac{k^2q_j^2}{3}\right)
    \end{align*}
    
    Now, for $j \in I_+$, we have: $  -kq_j + \log(1+kq_j) \leq -kq_j + \log(1+kq_j) \leq -\frac{1}{10} kq_j
$ using the inequality $-0.9x + \log(1+x) \leq 0$ true for all $x \geq \frac{1}{2}$. 
Therefore, we have upper bounded the type-II error of $\psi_2$ by:
\begin{align*}
    \mathbb{P}_{p^{\otimes \widetilde k}}(\forall j\in& \Ibar : N_j = 0 \text{ or } N_j = 1) \leq \exp\Big(-\frac{1}{3} \sum_{j \in I_-} k^2q_j^2 - \frac{1}{10} \sum_{j\in I_+} kq_j\Big)\\
    & \leq \exp\Big(-\frac{1}{3} \sum_{j \in I_-} k^2q_j^2 - \frac{1}{10}\big(\sum_{j\in I_+} k^2q_j^2\big)^{1/2}\Big)\\
    & = \exp\Big(-\frac{1}{3} (S-S_+) - \frac{1}{10}\left(S_+\right)^{1/2}\Big) \text{ for $S = \sum_{j\in \Ibar} k^2q_j^2$ and $S_+ = \sum_{j\in I_+} k^2q_j^2$.} 
\end{align*}
Now, $S_+ \mapsto -\frac{S}{3} + \frac{1}{3}S_+ - \frac{\sqrt{S_+}}{10}$ is convex over $[0,S]$ so its maximum is reached on the boundaries of the domain and is therefore equal to $(-\frac{\sqrt{S}}{10}) \vee -\frac{S}{3} = -\frac{\sqrt{S}}{3}$ for $S \geq 9/100$. Now, since $\|q\|_2^2 \geq 4\CL{\ref{Erreur_II_psi2}}/k^2 \geq \CL{\ref{Erreur_II_psi2}}/k^2$, we have $S = k^2 \|q\|_2^2 \geq \log(16/\eta)^2 \vee 9/100$ which ensures $\mathbb{P}_{p^{\otimes \widetilde k}}(\forall j\in \Ibar : N_j = 0 \text{ or } N_j = 1) \leq \eta/16$, hence, by equation \eqref{histogram_poissonized}, $\mathbb{P}_p(\psi_2 = 0) \leq \frac{\eta}{16}/\mathbb{P}(\widetilde k \leq n) \leq \frac{\eta}{8}$ if $n$ is larger than a constant.
\end{proof}

\section{Lower bound in the tail regime}\label{LB_tail_appendix}





We now define $(\widetilde C_j)_{j \in \Ibar}$ as the covering of $\T(\uA)$ given by Algorithm \ref{splitting_tail} with inputs $u = \uA$ and $h = \cch \hT(\uA)$ for a small constant $\cch$. 
For all $j\in \Ibar$, define $p_j = \int_{\widetilde C_j }p_0$. We recall that \boundedcase{$\Ibar = \N^*$ or $\Ibar = \{1,\dots,M\}$ for some $M \in \N$ and that }the cells $(\widetilde C_j)_j$ are ordered such that the nonnegative real numbers $(p_j)_{j \in \Ibar}$ are sorted in decreasing order. Now, we can set:
\begin{equation}\label{def_U}
    U = \min \Big\{j \in \Ibar ~\big|~ n^2\, p_j  \sum_{l\geq j} p_l \leq \cu \Big\}.
\end{equation}

Lemma \ref{sum_pj_geq_int} proves that, when (a) and (b) hold, \boundedcase{$\{j \in \Ibar: j\geq U\} \neq \emptyset$. Therefore, }the union $D(U) := \bigcup\limits_{j \geq U} \widetilde C_j$ is not empty and that $S_U = \sum\limits_{j\geq U} p_j > 0$. For $j \geq U$ and for a sufficiently small constant $\cu>0$, we now set
\begin{equation}\label{def_pi}
    \pi_j = \frac{p_j}{\bar \pi} ~~ \text{ and } ~~ \bar \pi = \frac{2\cu}{n^2 \sum\limits_{j\geq U} p_j}.
\end{equation}

Index $U$ has no further meaning than to guarantee that $\pi_j \in [0,\frac{1}{2}]$ for all $j \geq U$. In particular, $\pi_j$ is a Bernoulli parameter.
\\

We now give high-level explanations regarding the construction of the tail prior. To start with, this prior will be supported over $D(U)$ rather than $\T(\uA)$. First, one sparse subset of indices $J_S \subset \{U,\dots,M\}$ is drawn by setting $J_S = \{j: b_j=1\}$ where for each $j \geq U$, $b_j \sim Ber(\pi_j)$ are independent Bernoulli random variables with parameter $\pi_j$. 
The random elements of $J_S$ represent the indices of the cubes $\widetilde C_j$ denoted here as the \textit{selected cubes}. On each selected cube $\big(\widetilde C_j\big)_{j \in J_S}$ one large (deterministic) perturbation $\gammaup_j\in H(\alpha, \delta' L)$ is added to $p_0$. Conversely, on each non-selected cube $\big(\widetilde C_j\big)_{j \notin J_S}$, one small perturbation $\gammadown_j \in H(\alpha, \delta' L)$ is removed from $p_0$. We consider the random function defined by
\begin{equation}\label{def_qbn}
    q_b:= p_0 + \sum_{j \geq U} \Big[b_j \gammaup_j - (1-b_j)\gammadown_j\Big].
\end{equation} Since $q_b$ may not necessarily be a probability density, we rescale $q_b$ to define the prior as follows:
\begin{equation}\label{def_prior_tail}
   p_b^{(n)} = \frac{q_b~}{\big\| q_b\big\|_1}.
\end{equation} 
The definitions of $\gammadown_j$ and $\gammaup_j$ are given in Equations \eqref{def_gammadown} and \eqref{def_gamma}. 
We show in Proposition \ref{separation_prior_tail} that with high probability, our prior satisfies $\|p_0 - p_b\|_t \geq \CtailLB \rhot$ for a constant $\CtailLB$. 



We now give the precise definitions of $(\gammaup_j)_j$ and $(\gammadown_j)_j$. The perturbations $(\gammadown_j)_j$ are designed to guarantee the following condition: $\forall j \geq U, \int_{\widetilde C_j} \gammadown_j \geq c\, p_j$ for some small constant $c>0$. 
To do this, we split $\widetilde C_j$ into smaller cells $(\Elj)_{l=1}^{M_j}$ on which $p_0$ can be considered as "approximately constant", in the sense that $\max\limits_{\Elj}p_0 / \max\limits_{\Elj}p_0 \in [c',c'']$ for two constants $c',c''>0$. By Assumption \eqref{simplifyingAssumption}, this condition is satisfied if all $\Elj$ have edge length $\asymp \left(\frac{p_0(x)}{L}\right)^{1/\alpha}$. 
We now remove on each $\Elj$ a small deterministic function $\philj$ whose total mass is at least $c\int_{\Elj} p_0$. The role of the $\gammadown_j$ is therefore to remove a small fraction of the mass of $p_0$ on each cell where $b_j=0$.  To formally define $\gammadown_j$, we first let for all $j\geq U$:
\begin{align}
    u_j &= \inf \Big\{u>0: \int_{\widetilde C_j} \mathbb{1}_{p_0(x)\geq u}\, p_0 \; \geq \frac{1}{2}p_j\Big\}, \label{def_uj}\\    
    D_j &= \; \big\{x \in \widetilde C_j: p_0(x) \geq u_j \big\}  \label{Def_Dj}.
\end{align}
We therefore apply Algorithm \ref{Algo_partitionnement} with inputs $\widetilde \Omega = \widetilde C_j$, $\beta = \alpha$, $u = u_j$ and $\cbeta = \cbeta'L$ for some large constant $\cbeta'$, and we set $\calpha = \cbeta'$. Taking $\cbeta'$ large enough ensures $\cstar + \frac{\sqrt{d}^\alpha}{\calpha} (2^{1-\alpha} \lor 1) \leq 1/2$, hence, the guarantees of Proposition \ref{guarantees_algo} are satisfied. 
For each cube $\widetilde C_j$, Algorithm \ref{Algo_partitionnement} defines the family of smaller cells $(E_1^{(j)}, \dots, E_{M_j}^{(j)})$ for some $M_j \in \N$. We denote the center of each cube $E_l^{j}$ by $z_l^{(j)} \in \widetilde C_j$ and its edge length by $h_l^{(j)}\asymp \left(\frac{1}{L \; \cbeta'} p_0(z_l^{(j)})\right)^{1/\alpha}$. Moreover, each cube has non empty intersection with $D_j$ and  $D_j \subset \bigcup\limits_{l=1}^{M_j} \Elj$. For some constant $\cdown$ small enough, define on each cell $\Elj$: 
\begin{equation}\label{def_phi_lj}
    \philj(x) = \cdown L \Big({\hlj}\Big)^\alpha f\Big(\frac{x-\zlj}{\hlj}\Big),
\end{equation}
where we recall that $f\geq 0$ over $\R^d$, $f \in H(\alpha,1) \cap C^\infty$, and $f$ is supported over $\left\{x \in \R^d : \|x\|<1/2 \right\}$. We here moreover assume that $f$ satisfies
\begin{equation}\label{simplifying_assymption_f}
    \forall x,y \in \R^d: |f(x)-f(y)| \leq \cstar f(x) + \|x-y\|^\alpha.
\end{equation}
The perturbation $\gammadown_j$ is defined as:
\begin{equation}\label{def_gammadown}
    \gammadown_j = \sum_{l=1}^{M_j} \philj.
\end{equation}

We now move to the definition of $\gammaup_j$. Assuming that each cube $\widetilde C_j$ is centered at $z_j$, $\gammaup_j$ is defined as:
\begin{equation}\label{def_gamma}
    \gammaup_j(x) = \cupward_j L h^\alpha ~ f\left(\frac{x-z_j}{h}\right), ~~~~ j \geq U,
\end{equation}
where $h:= \hT(\uA)$ is defined in \eqref{def_htail} and $\cupward_j$ is chosen so as to ensure that $\pi_j \int \gammaup_j = (1-\pi_j) \int \gammadown_j$. 
In other words, $\cupward_j$ is chosen so that the total mass of the prior is equal to $\int_{D(U)} p_0$ in expectation over the $(b_j)_{j\geq U}$. 
Noticeably, when setting $\cupward := \cu \cdown$, Proposition \ref{Properties_gammadown} shows that $\forall j \geq U, \cupward_j \in \left[\cupward, \; 2\cupward \right]$ i.e. $\cupward_j$ is lower- and upper bounded by two strictly positive constants. \\

The functions $(\gammadown_j)_{j \geq U}$ and $(\gammaup_j)_{j \geq U}$ are chosen to ensure the following properties:

\begin{proposition}\label{Properties_gammadown}
\begin{enumerate}
    \item For all $(b_j)_{j\geq U}: \; p_b^{(n)} \in \mathcal{P}(\alpha, L',\cstar')$ over the whole domain $[0,1]^d$. 
    \item There exists a constant $\Cdown>0$ independent of $p_0$ such that for all $j \geq U: \; \Cdown\int_{\widetilde C_j} \gammadown_j \; \geq \;  \int_{\widetilde C_j} p_0$ where $\Cdown = \cdown (1-\cstar) \Cint$.
\end{enumerate}
\end{proposition}

For clarity, we now give the the probability density over the space $\Omega^n$ of the data when they are generated from prior \eqref{def_prior_tail}. 
Assume that we observe $(X''_1, \dots, X''_n)$ generated from $p_b^{(n)}$. 
Then $X''_1, \dots, X''_n $ are all $iid$ with the \textit{same} density $q$, which is itself (not uniformly) drawn in the set $\{p_b ~|~  b_j \in \{0,1\} ~ \forall j \geq U\}$. 
In other words, the density of $(X''_1, \dots, X''_n)$ corresponds to the mixture 
\begin{equation}\label{def_pbar}
     \overline p^{(n)} = \sum_{\substack{b_j \in \{0,1\}\\ j \geq U}}~ \prod_{j \geq U} \pi_j^{b_j} (1 - \pi_j)^{1-b_j} ~ \left(\frac{q_b}{\|q_b\|_1}\right)^{\otimes n},
\end{equation}
where $q_b$ is defined in \eqref{def_qbn}. The lower bound will be proved by showing that there exists no test with risk $\leq \eta$ for the testing problem $H_0' : (X''_1, \dots, X''_n) \sim p_0^{\otimes n}$ vs $H'_1 : (X''_1, \dots, X''_n) \sim \overline p^{(n)}$. The following Proposition states that the prior concentrates \textit{whp} on a zone separated away from $p_0$ by an $L_t$ distance of order $\rhot$.

\begin{proposition}\label{separation_prior_tail} 
There exists a constant $\CtailLB$ such that, when $\int_{\T} p_0 \geq \ctail/n$, we have with probability at least $1-\frac{\eta}{4}$ (over the realizations of $b=(b_U, \dots, b_M)$):
$$ \|p_b^{(n)} - p_0\|_t \geq \CtailLB\, \rhot.$$
\end{proposition}


We now introduce the \textit{Bayes risk} associated with the prior distribution \eqref{def_prior_tail}: 
\begin{definition}\label{bayesian_risk_tail_prior}
Define
$$ R_B^{\; tail} = \inf_{\psi \text{ test}} \Big\{\mathbb{P}_{p_0}(\psi = 1) + \mathbb{E}_{b} \Big[ \mathbb{P}_{p_b}(\psi = 0)\Big]\Big\},$$
where the expectation is taken with respect to the realizations of $(b_j)_{j \geq U}$ and $\mathbb{P}_{p_b}$ denotes the probability distribution when the data is drawn with density \eqref{def_pbar}.
\end{definition}

The Proposition below states that when $\int_{\mathcal{T}(\uA)} p_0 \geq \ctail/n$ and when the tail dominates, the prior \eqref{def_prior_tail} is indistinguishable from $p_0^{\otimes n}$, in the sense that there exists no test with risk $\leq \eta$ for the testing problem $H'_0 : (X''_1, \dots, X''_n) \sim p_0^{\otimes n}$ vs $H'_1 : (X''_1, \dots, X''_n) \sim \overline p^{(n)}$.\\

\begin{proposition}\label{lower_bound_tail}
$R_B^{\; tail} > \eta.$
\end{proposition}

\hfill

\textbf{Remark:} Our prior concentrates only with high probability on the zone $\|p_0 - p_b\|_t \geq \CtailLB \rhot$. We can here justify that this is not restrictive. Indeed, we can \textit{wlog} modify Proposition \ref{lower_bound_tail} to get $R_B^{\; tail} > \eta-2\epsilon$ for any $\epsilon>0$ small enough. We moreover show in Lemma \ref{LB_prior_restricted} that if instead of our prior $\mathbb{E}\big(p_b^{(n)}\big)$, we considered as prior $p_{b,cond} = \mathbb{E}(p_b^{(n)} | \Asep)$ where $\Asep = \{b \text{ is such that } \|p_0 - p_b\|_t \geq \CtailLB \rhot\}$ and where the expectation is taken according to the realizations of $b$, then we would have $d_{TV}(p_0^{\otimes n}, p_{b,cond}) < d_{TV}(p_0^{\otimes n}, \mathbb{E}_b(p_b^{(n)})) + 2 \epsilon \leq 1-\eta-2\epsilon+2\epsilon = 1-\eta$. Now, $p_{b,cond}$ satisfies almost surely $\|p_0 - p_{b,cond}\|_t \geq \CtailLB \rhot$.



\subsection{Proof of Proposition \ref{remainder_term}}

\begin{proof}[Proof of Proposition \ref{remainder_term}]
Assume that $\int_{\T(\uA)} p_0 < \frac{\ctail}{n}$ and that $n> \ctail$. \boundedcase{Recall that $Ln \homega^{\alpha+d} > \csmall$.} Since $\int_{\T(\uA)} p_0 \leq \frac{\ctail}{n} <1$, we necessarily have $\T(\uA) \subsetneq \Omega$. Set $u = \sup \{v> 0: \int_{\T(v)} p_0 \leq \frac{\ctail}{n}\}$. We therefore necessarily have $u \leq \max_{\Omega} p_0$ (since $n > \ctail$) and $\int_{\overline{T}(u)}p_0 \geq \frac{\ctail}{n}$. Choose $D(u) \subset p_0^{-1}(\{u\})$ a subset such that $\int_{D(u) \cup \T(u)} p_0 = \frac{\ctail}{n}$ and define $T'(u) = D(u) \cup \T(u)$. By the definition of $\uI$, we have $u > \uI$ so that
\begin{align}
    &\left(\max_{\Omega} p_0\right) \int_{\T'(u)} p_0 \geq \int_{\T(u)}p_0^2 \geq \cI \left[\frac{L^d}{n^{2\alpha} \left(\int_{\T'(u)} p_0 \right)^d} \right]^\frac{1}{\alpha + d}\nonumber \\
    \text{hence } ~~ & \max_{\Omega} p_0 \geq \cI \left[\frac{L^d}{n^{2\alpha } \left(\int_{T'(u)}p_0\right)^\alpha}  \right]^\frac{1}{\alpha + d} = \frac{\cI}{\ctail^\frac{\alpha}{\alpha+d}} \left[\frac{L^d}{n^\alpha} \right]^\frac{1}{\alpha + d} =: m. \label{max_p0_gros}
\end{align}
Define $h_r =  (nL/\csmall)^{-\frac{1}{\alpha+d}}\boundedcase{< \homega}$ \unboundedcase{ for some small enough constant $\csmall$} and $x_0 = \arg \max\limits_\Omega p_0$. We note that $m = \cI' Lh_r^\alpha $ where $\cI' = \csmall^\frac{\alpha}{\alpha+d} \cI/\ctail^\frac{\alpha}{\alpha+d}$. Set $B_1$ and $B_2$ two disjoint balls included in $\Omega \cap B\left(x_0, \left(\cI' \cstar \right)^{1/ \alpha}h_r\right)$ with radius  $R := \frac{1}{4\sqrt{d}}\left(\cI' \cstar \right)^{1/ \alpha}h_r$. $B_1$ and $B_2$ exist no matter how close $x_0$ is to the boundary of $\Omega$. Denote by $x_1^{(r)}$ and $x_2^{(r)}$ the respective centers of $B_1$ and $B_2$. By Assumption \eqref{simplifyingAssumption}, we have $p_0 \geq m(1-2\cstar)$ over $\Omega \cap B\left(x_0, \left(\cI' \cstar \right)^{1/ \alpha}h_r\right)$ so that it is possible to set the following prior: 
\begin{equation}\label{prior_remainder}
    p_r(x) = p_0(x) + \cR Lh_r^\alpha \; f\Big(\frac{x_1^{(r)}-x}{h_r}\Big)  - \cR Lh_r^\alpha \; f\Big(\frac{x_2^{(r)}-x}{h_r}\Big),
\end{equation}
where 
$\cR$ is a small enough constant. This prior satisfies $\int_\Omega p_r = 1$, $p_r \geq 0$, $p_r \in H(\alpha, L(1+\cR))$ and satisfies Assumption \eqref{simplifyingAssumption} by Lemma \ref{simplifying_assumpt_satisfied} if we choose $\cR$ small enough. Moreover, the $L_t$ discrepancy between $p_0$ and $p_r$ is given by
\begin{align*}
    \|p_0 - p_r\|_t^t &= 2 \int_{\R^d} \Big\{\cR Lh_r^\alpha \; f\Big(\frac{x_1^{(r)}-x}{h_r}\Big) \Big\}^{\alpha t} dx =  2 \left(\cR L \right)^{\alpha t}\Cintt \; h_r^{\alpha t + d}  \asymp L^{\frac{d(t-1)}{t(\alpha+d)}}n^{-\frac{\alpha t + d}{t(\alpha+d)}}.
\end{align*}
Now, the total variation between $p_r$ and $p_0$ is given by:
\begin{align*}
    d_{TV}(p_0, p_r) =  L h_r^{\alpha + d} \Cint = \cR\Cint \frac{1}{n} < 1-\eta,
\end{align*}
for $\cR$ small enough, which proves the desired lower bound.
\end{proof}

\subsection{Proof of Proposition \ref{Properties_gammadown}}

\begin{lemma}\label{norm_qb} 
It holds $\mathbb{E}[\|q_b\|_1] = 1$ and $\mathbb{V}[\|q_b\|_1] \leq \CL{\ref{norm_qb}}/n^2$, where $\CL{\ref{norm_qb}}$ is a constant. 
\end{lemma}

\begin{proof}[Proof of Lemma \ref{norm_qb}]
First, $\mathbb{E}[\|q_b\|_1] = 1$ is true by the definition of $\cupward_j$ and $\cdown$. As to the variance, we recall that for all $j\geq U : \Gammaup_j = \int_{\widetilde C_j} \gammaup_j$ and $\Gammadown = \int_{\widetilde C_j} \gammadown_j$. We have:
\begin{align*}
    \mathbb{V}[\|q_b\|_1] &= \sum_{j \geq U} \mathbb{V}\left(b_j(\Gammaup_j + \Gammadown_j) \right) \leq  \sum_{j\geq U} \pi_j \left(\frac{\AL{\ref{gammadown_lesssim_intp0}}}{1+\BL{\ref{gammadown_lesssim_intp0}}}\right)^2{\Gammaup_j}^2 ~~ \text{ by Lemma \ref{gammadown_lesssim_intp0}}\\
    & \leq \big(2{\cupward}\big)^2 \left(\frac{\AL{\ref{gammadown_lesssim_intp0}}}{1+\BL{\ref{gammadown_lesssim_intp0}}}\right)^2\Big(\sum_{j\geq U} \pi_j\Big) \left(Lh^{\alpha+d}\right)^2.
\end{align*}
Moreover:
\begin{align*}
    \sum_{j\geq U} \pi_j = \frac{\left(n \sum_{j\geq U} p_j\right)^2}{2\cu} \leq \frac{\left(n \int_{\T(\uA)} p_0\right)^2}{2\cu} {\CL{\ref{Toute_la_case_OK_moment1}}}^2 ~~\text{ by Lemma \ref{Toute_la_case_OK_moment1}},
\end{align*}
and $ \left(Lh^{\alpha+d}\right)^2 = n^{-4} \left(\int_{\T(\uA)} p_0\right)^{-2}$. Hence: $\mathbb{V}[\|q_b\|_1] \leq \CL{\ref{norm_qb}}/n^2$, for some constant $\CL{\ref{norm_qb}}$.
\end{proof}

\begin{lemma}\label{gammaup_simplifying_assumption_cases_diff}
Let $I$ be a countable set of indices and $(x_l)_{l \in I} \in \Omega$ and $(h_l)_{l \in I} >0$ such that the balls $(B(x_l, h_l))_l$ are disjoint. Set moreover $(\epsilon_l)_{l \in I} \in \{\pm 1\}^I$ and let $\Calpha =1 \lor 2^{1-\alpha}$ and $\gamma(x) = \sum\limits_{l \in I} \epsilon_l a_l\, Lh_l^\alpha \;f\big(\frac{x-x_l}{h_l}\big)$ where $(a_l)_l \geq 0$. Then $\forall x,y \in \Omega, |\gamma(x) - \gamma(y)| \leq \cstar |\gamma(x)| +  \bar a \, \Calpha L \|x-y\|^\alpha$ where $\bar a = \sup\limits_{l \in I} a_l$.
\end{lemma}

\begin{proof}[Proof of Lemma \ref{gammaup_simplifying_assumption_cases_diff}]
Set for all $l \in I: A_l = B(x_l, h_l)$ and $A_{0} = \Omega \setminus \left(\bigcup_{l \in I} A_l \right)$. Let $x,y \in \Omega$. The result is direct if $x,y \in A_{0}$. If $x,y$ are in the same set $A_l$ where $l \neq 0$ then by equation \eqref{simplifying_assymption_f} we have:
\begin{align*}
    |\gamma(x) - \gamma(y)| &=  a_l L h_l^\alpha \; \Big|f\Big(\frac{x-x_l}{h_l}\Big)-f\Big(\frac{y-x_l}{h_l}\Big)\Big| \leq  a_l L h_l^\alpha \;\Big[\cstar f\Big(\frac{x-x_l}{h_l}\Big) + \Big\|\frac{y-x}{h_l}\Big\|^\alpha\Big]\\
    & = \cstar\gamma(x) +  a_l L \|y-x\|^\alpha.
\end{align*}
Assume now there exist $i \neq l$ such that $x \in A_i$ and $y \in A_l$. For $x'\in A_i$ and $y' \in A_l$ such that $d_{\|\cdot\|}(x',A_{0}) =0$ and $d_{\|\cdot\|}(y',A_{0}) =0$ we have by equation \eqref{simplifying_assymption_f}: $$\big|\gamma(x)\big| = \big|\gamma(x) - \gamma(x')\big| \leq \cstar \big|\gamma(x')\big| + a_i \, Lh_i^\alpha \big\|\frac{x-x'}{h_i}\big\|^\alpha =  a_i L\|x-x'\|^\alpha,$$ $$\big|\gamma(y)\big| = \big|\gamma(y) - \gamma(y')\big| \leq \cstar \big|\gamma(y')\big| +  a_l \, Lh_l^\alpha \big\|\frac{y-y'}{h_l}\big\|^\alpha =  a_l L\|y-y'\|^\alpha.$$
\hfill

Moreover, we have $\|x-y\| \geq \|x-x'\| + \|y-y'\|$ since $x$ and $y$ are in two different sets, and $\Calpha = \max \{\lambda^\alpha + (1-\lambda)^\alpha: \lambda \in [0,1]\}$ so that: $\Calpha\|x-y\|^\alpha \geq \|x-x'\|^\alpha + \|y-y'\|^\alpha$. This yields the result.
\end{proof}

\begin{proof}[Proof of Proposition \ref{Properties_gammadown}]
\begin{enumerate}
    \item We first show that $p_0 - \gammadown_j \geq 0$ for all $j \in \Ibar$. By Lemma \ref{pU_leq_Lh^alpha+d}, we have $p_U \leq \CL{\ref{pU_leq_Lh^alpha+d}}Lh^{\alpha+d}$ where $\CL{\ref{pU_leq_Lh^alpha+d}}$ is a constant, so that by Assumption \eqref{simplifyingAssumption}:
    \begin{equation}\label{p0_upper_bounded_j_geq_U}
        \forall j \geq U,~ \forall x \in \widetilde C_j,~ p_0(x) \leq \left(\CL{\ref{pU_leq_Lh^alpha+d}}(1+\cstar) + \sqrt{d}^\alpha\right)Lh^\alpha =: CLh^\alpha.
    \end{equation}
    Recall that $e(\widetilde C_j) = \hT$. Therefore, the condition that Algorithm \ref{Algo_partitionnement} splits $\widetilde C_j$ at least once rewrites:
    \begin{align*}
        e(\widetilde C_j) = \hT > \left(\frac{p_0(x_j)}{\cbeta' L}\right)^{1/\alpha} \Longleftarrow \hT > \left(\frac{C}{\cbeta'}\right)^{1/\alpha}\hT
    \end{align*}
    by equation \eqref{p0_upper_bounded_j_geq_U}, which is true if we choose $\cbeta'$ large enough. This ensures that for all cell $\Elj$ and for all $x \in \Elj$, we have by the properties of the partitioning scheme (Proposition \ref{guarantees_algo} item \ref{min_geq_max}), that $p_0(x)\geq \frac{1}{2}p_0\big(z_l^{(j)}\big) \geq L (\hlj)^\alpha$ by taking $\cbeta = \calpha \geq 2$. 
    Therefore, 
    $$p_0 - \philj \geq p_0 - \cdown L\big(\hlj\big)^\alpha \geq \frac{1}{2} p_0\big(z_l^{(j)} \big) - \cdown L\big(\hlj\big)^\alpha \geq \frac{1\,-\,\cdown}{2} p_0\big(z_l^{(j)} \big) \geq 0.$$ 
    Moreover, it is clear that $\widetilde C_j$, $p_0 - \gammadown_j \in H(\alpha, L(1+ \cdown)) \subset H(\alpha, L')$ and $\widetilde C_j$, $p_0 + \gammaup_j \in H(\alpha, L(1+ \cupward_j)) \subset H(\alpha, L')$ for $\cdown$ small enough. 
    To finish, by Lemma \ref{gammaup_simplifying_assumption_cases_diff}, we have that for all $(b_j)_{j \geq U}$, $p_b^{(n)}$ satisfies Assumption \eqref{simplifyingAssumption} with the constants $\cstar'$ and $L'$ by choosing $\cdown$ small enough. Indeed, set $\gamma = \sum\limits_{l \geq U} b_j \gammaup_j - (1-b_j) \gammadown_j$ and $\bar a = \sup ~ \left( \{\cdown\}\cup \{\cupward_j: j \geq U\}\right)$.
    \begin{align*}
        |q_b(x) - q_b(y)| &\leq |p_0(x) - p_0(y)| + |\gamma(x) - \gamma(y)|\\
        &\leq \cstar p_0(x) + \cstar|\gamma(x)| + (1+\bar a \, \Calpha) L \|x-y\|^\alpha.
    \end{align*}
    Let $j \geq U$ such that $x \in \widetilde C_j$. If $b_j = 1$ then $\gamma(x) = \gammaup_j(x)\geq 0$ hence $p_0(x) + |\gamma(x)| = p_0(x) + \gamma(x)$ which proves that \eqref{simplifyingAssumption} is satisfied. Otherwise, $\gamma(x) = - \gammadown(x)$. We have already shown that $\gammadown_j \leq p_0$ over $\widetilde C_j$. Taking $\cdown$ small enough, we can therefore impose, for any $\lambda >0 : \gammadown_j \leq \lambda p_0$ over $\widetilde C_j$. Therefore,
    \begin{align*}
        p_0(x) + |\gamma(x)| = p_0(x) + \gammadown_j(x) \leq \frac{1+\lambda}{1-\lambda} \big(p_0(x) - \gammadown_j(x)\big). 
    \end{align*}
    Taking $\lambda$ and $\bar a$ small enough (which can be done by taking $\cdown$ small enough), we get in both cases that $q_b$ satisfies Assumption \eqref{simplifyingAssumption} with the constants $\cstar(1+\delta/2)$ and $L(1+\delta/2)$ instead of $\cstar$ and $L$. 
    Now, by Lemma \ref{norm_qb} and the Chebyshev inequality, the event $\left\{\left|\, \|q_b\|_1-1 \,\right| \leq \CL{\ref{gammadown_lesssim_intp0}}\right\}$ can have arbitrarily high probability when $n$ is larger than a suitably chosen constant.
    Taking $n$ large enough ensures that with probability arbitrarily close to $1$, $p_b^{(n)}$ satisfies \eqref{simplifyingAssumption} with the constants $\cstar'$ and $L'$.
    \item By Proposition \ref{guarantees_algo}, we have $h_l^{(j)} \geq \frac{1}{2^{\beta+1}} \left(\frac{1}{L \; \cbeta'} p_0(z_l^{(j)})\right)^{1/\alpha}$ and that for all $l \in \{1,\dots, M_j\}: p_0  \geq \frac{1}{2}\; p_0(z_l^{(j)})$ over $\Elj$.
    Now, for all $j\geq U$ and $l \in \{1, \dots, M_j\}$ we have
    \begin{align*}
        \int_{\widetilde C_j} \gammadown_j &= \sum_{j=1}^{M_j} \int_{\Elj} \philj = \sum_{j=1}^{M_j} \cdown L \big(\hlj\big)^{\alpha+d} \|f\|_1\\
        & \geq \cdown \|f\|_1 \sum_{j=1}^{M_j} L \frac{1}{2^{\beta+1}}\frac{1}{L \; \cbeta'} p_0(z_l^{(j)}) \big(\hlj\big)^d\\
        & \geq \cdown \|f\|_1 ~\frac{1}{2^{\beta+2}\cbeta'}  \int_{D_j} p_0  ~ \geq ~ \frac{1}{\Cdown} \int_{D_j} p_0,
    \end{align*}
    where $\Cdown$ is a constant, which ends the proof.
\end{enumerate}
\end{proof}

\subsection{Proof of Proposition \ref{separation_prior_tail}}

In what follows we set for all $j\geq U$:
\begin{equation}\label{Def_Gamma_up_down}
     \Gammaup_j = \int_{\widetilde C_j} \gammaup_j ~~ \text{ and } ~~ \Gammadown = \int_{\widetilde C_j} \gammadown_j.
\end{equation}

\begin{lemma}\label{gammadown_lesssim_intp0}
There exist three constants  $\AL{\ref{gammadown_lesssim_intp0}}, \BL{\ref{gammadown_lesssim_intp0}}$ and $\CL{\ref{gammadown_lesssim_intp0}}$ such that for all $j \geq U$, it holds:
\begin{enumerate}
    \item $\Gammadown_j \leq \AL{\ref{gammadown_lesssim_intp0}}\; p_j$,
    \item $\Gammaup_j \geq \BL{\ref{gammadown_lesssim_intp0}}\; p_j$ 
    \item $\|\gammaup_j\|_t^t \geq \CL{\ref{gammadown_lesssim_intp0}} \; \int_{\widetilde C_j} p_0^t$ ~ where $\CL{\ref{gammadown_lesssim_intp0}} < 1$.
\end{enumerate}
\end{lemma}
\begin{proof}[Proof of Lemma \ref{gammadown_lesssim_intp0}] 
Fix $j \geq U$.
\begin{enumerate}
    \item We have:
    \begin{align*}
    \int_{\widetilde C_j} \gammadown_j &= \sum_{l=1}^{M_j} \cdown  L {\hlj}^{\alpha+d} \|f\|_1 \leq \|f\|_1 \sum_{l=1}^{M_j}  \frac{\cdown}{\cbeta'}p_0\big(\zlj\big) {\hlj}^{d} \\
    &\leq 2\|f\|_1 \frac{\cdown}{\cbeta'} \sum_{l=1}^{M_j} \int_{\Elj} p_0 \leq 2\|f\|_1 \frac{\cdown}{\cbeta'}  p_j =: \AL{\ref{gammadown_lesssim_intp0}}\; p_j.
\end{align*}
\item By definition of $U$:
\begin{align}
    p_j &\leq \frac{\cu}{n^2\sum_{j\geq U} p_j} \leq \frac{\cu}{n^2D \int_{\T(\uA)} p_0} = \frac{\cu}{D} L h^{\alpha+d}\label{pour_item2}\\
    &\leq \frac{1}{\BL{\ref{gammadown_lesssim_intp0} 
    }}\Gammaup_j, \hspace{1cm} \text{ for some constant $\BL{\ref{gammadown_lesssim_intp0}}$.}\nonumber
\end{align}
\item Let $x \in \widetilde C_j$ and $y \in \widetilde C_j$ such that $p_0(y) = \frac{p_j}{h^d}$, which exists by the intermediate value theorem. We have by Assumption \eqref{simplifyingAssumption}:
\begin{align*}
    p_0(x) &\leq (1+\cstar)p_0(y) + L(h\sqrt{d})^\alpha \\
    & \leq \left[(1+\cstar)\frac{\cu}{D} + \sqrt{d}^\alpha\right]Lh^\alpha \hspace{1cm} \text{by Equation \eqref{pour_item2}},
\end{align*}
so that:
\begin{align*}
    \int_{\widetilde C_j} p_0^t \leq \left[(1+\cstar)\frac{\cu}{D} + \sqrt{d}^\alpha\right]^t L^t h^{\alpha t + d} \leq \frac{1}{ \CL{\ref{gammadown_lesssim_intp0}}} \; \|\gammaup_j\|_t^t \;, \hspace{3mm} \text{for some constant $\CL{\ref{gammadown_lesssim_intp0}}$}.
\end{align*}
\end{enumerate}
\end{proof}

\begin{proof}[Proof of Proposition \ref{separation_prior_tail}]
Assume throughout the proof that $\int_{\B^c} p_0 > \frac{\ctail}{n}$. We show that our prior concentrates with high probability on a zone separated away from $p_0$ by an $L_t$ distance of order $\rhot$, up to a constant. To lower bound the $L_t$ separation between our prior and the null distribution, we will only consider the discrepancy accounted for by the perturbations $(\gammaup_j)_j$. We recall that $\widetilde C_0 = D(U)^c$. For all $j\geq U$, fix $b_j \in \{0,1\}$ as well as $n$ large enough, such that
\begin{equation}\label{Evenement_gde_pba_tail}
    \left|\, \|q_b\|_1-1 \,\right| \leq \CL{\ref{gammadown_lesssim_intp0}}.
\end{equation}
By Lemma \ref{norm_qb} and the Chebyshev inequality, the event corresponding to Equation \eqref{Evenement_gde_pba_tail} can have arbitrarily high probability when $n$ is larger than a suitably chosen constant. Now, write $\Ibarzero = \{0\} \cup \{j \in \Ibar : j\geq U\}$.

\begin{align*}
    \|p_0 - p_b\|_t^t & = \sum_{j\in \Ibarzero} b_j\int_{\widetilde C_j} \Big|p_0 - \frac{p_0 + \gammaup_j}{\|q_b\|_1}\Big|^t +\sum_{j\in \Ibarzero} (1-b_j) \int_{\widetilde C_j} \Big|p_0 - \frac{p_0 + \gammadown}{\|q_b\|_1}\Big|^t\\
    &\geq \sum_{j\in \Ibarzero} b_j\int_{\widetilde C_j} \Big|p_0 - \frac{p_0 + \gammaup_j}{\|q_b\|_1}\Big|^t = \sum_{j\in \Ibarzero} b_j \Big\|p_0 - \frac{p_0 + \gammaup_j}{\|q_b\|_1}\Big\|_{t, \widetilde C_j}^t\\
    &\geq \sum_{j\in \Ibarzero} b_j \left|\frac{\|\gammaup_j\|_{t, \widetilde C_j}}{\|q_b\|_1} - \Big\|p_0\Big(1-\frac{1}{\|q_b\|_1}\Big)\Big\|_{t, \widetilde C_j}\right|^t \text{ by the reverse triangle inequality}\\
    & \geq \sum_{j\in \Ibarzero} b_j \left(\frac{\|\gammaup_j\|_{t, \widetilde C_j}}{\|q_b\|_1} - \Big(1-\frac{1}{\|q_b\|_1}\Big)\, \frac{1}{ \CL{\ref{gammadown_lesssim_intp0}}} \, \|\gammaup_j\|_{t,\widetilde C_j}\right)^t \text{ by Lemma \ref{gammadown_lesssim_intp0} and Equation \eqref{Evenement_gde_pba_tail}}\\
    & = \sum_{j\in \Ibarzero} b_j \|\gammaup_j\|_{t, \widetilde C_j}^t \left(\frac{1}{\|q_b\|_1}\left(1+\frac{1}{\CL{\ref{gammadown_lesssim_intp0}}}\right) - \frac{1}{\CL{\ref{gammadown_lesssim_intp0}}} \right)^t \\
    & \geq \sum_{j\in \Ibarzero} b_j \|\gammaup_j\|_{t, \widetilde C_j}^t \left(\frac{1}{2+\CL{\ref{gammadown_lesssim_intp0}}}\right)^t ~~ \text{ by Equation \eqref{Evenement_gde_pba_tail}}\\
    & \geq \sum_{j\geq U} b_j \cu \cdown L^t h^{\alpha t + d}\left(\frac{1}{2+\CL{\ref{gammadown_lesssim_intp0}}}\right)^t := \Cgap L^t h^{\alpha t + d} \sum_{j \geq U} b_j.
\end{align*}
It now remains to prove that, \textit{whp}, $L^t h^{\alpha t + d} \sum\limits_{j\geq U} b_j \gtrsim {\rhot}^t$. 
\begin{align*}
    \mathbb{E}\Big[\sum_{j\geq U} b_j\Big] = \sum_{j\geq U} \pi_j = \frac{n^2 \left(\sum\limits_{j\geq U}p_j\right)^2}{2\cu} \geq \frac{D^2 \ctail^2}{2\cu} ~ \text{ by Lemma \ref{sum_pj_geq_int}}.
\end{align*}
Moreover,
\begin{align*}
    \mathbb{V}\Big[\sum_{j\geq U} b_j\Big]\leq \sum_{j\geq U} \pi_j = \mathbb{E}\Big[\sum_{j\geq U} b_j\Big]
\end{align*}
We now consider the event
\begin{equation}\label{Evenement_gde_pba_tail_2}
    \sum_{j\geq U} b_j \geq \frac{1}{2} \mathbb{E}\Big[\sum_{j\geq U} b_j\Big].
\end{equation}
By the Chebyshev inequality, the probability of this event \lsmall{the constant $\ctail$}, since $\mathbb{V}\Big[\sum\limits_{j\geq U} b_j\Big] = o\Big(\mathbb{E}^2\Big[\sum\limits_{j\geq U} b_j\Big]\Big)$ as $\ctail \to +\infty$. 
Therefore, on the intersection of the events defined in Equations \eqref{Evenement_gde_pba_tail} and \eqref{Evenement_gde_pba_tail_2}, we have that: 
\begin{align*}
    \|p_0 - p_b\|_t^t &\geq \Cgap L^t h^{\alpha t + d} \sum_{j\geq U} b_j \geq \frac{\Cgap}{2} L^t h^{\alpha t + d} \sum_{j\geq U} \pi_j\\
    & \geq \frac{\Cgap}{2} L^t h^{\alpha t + d} D\int_{\T(\uA)}p_0 \hspace{5mm} \text{ by Lemma \ref{sum_pj_geq_int}}\\
    & \asymp {\rhot}^t.
\end{align*}
\end{proof}

\subsection{Proof of Proposition \ref{lower_bound_tail}}

\begin{proof}[Proof of Proposition \ref{lower_bound_tail}]

We draw $\widetilde n \sim \Poi(2n)$ and $\widetilde n' | b \sim \Poi\big(2n \int_{\Omega }q_b\big)$ independent of $\widetilde n$, and we let $\An = \{\widetilde n \geq n\}$ and $\An' = \{\widetilde n' \geq n\}$. By Lemma \ref{An_high_pba}, we can ensure $\mathbb{P}(\An), \mathbb{P}(\An') \geq 1-\eta/100$ for $n$ larger than a constant. This condition will be assumed throughout the proof of Proposition \ref{lower_bound_tail}. We will also slightly abuse notation and identify the probability densities with their associated probability measures. Moreover, we will use the notation $\widetilde C_0 = D(U)^c$ where we recall that $D(U) = \bigcup\limits_{j \geq U} \widetilde C_j$. We recall the definition of $\bar p^{(n)}$ in \eqref{def_pbar} and introduce the \textit{poissonized} probability measures $\bar p_b^{(\widetilde n')}$, $\qtens$ and $\ptens$, defined over $\Xtilde = \bigcup\limits_{n \in \N} \Omega^n$. The core of the proof is to link our target quantity $d_{TV}(\bar p^{(n)}, p_0^{\otimes n}) = d_{TV}(\mathbb{P}_{p_0^{\otimes n}}, \mathbb{P}_{\bar p^{(n)}})$ (which we want to upper bound by a small constant), to the quantity $d_{TV}(\qtens, \ptens)$ which is easier to work with. For clarity, we give the densities associated to each of the poissonized probability measures. For any $x \in \Xtilde$, we denote by $\widetilde n(x) \in \N$ the unique integer such that $x=(x_1, \dots, x_{\widetilde n(x)})$.

\begin{align*}
    \ptens(x) \big| \{\widetilde n = \nu\} \; &= \begin{cases} \text{ if } \widetilde n(x) \neq \nu: & 0 \\\text{ otherwise: } &
    p_0^{\otimes \nu}(x). \end{cases}\\
    &\\
    \bar p_b^{(\widetilde n')}(x) \big| \{\widetilde n' = \nu\} &= \begin{cases} \text{ if } \widetilde n'(x) \neq \nu: & 0 \\\text{ otherwise: } & \hspace{-5mm}
    \sum\limits_{(\beta_j)_{j\geq U} \in \{0,1\}}~ \prod\limits_{j \geq U} ~ \pi_j^{\beta_j} (1 - \pi_j)^{1-\beta_j} ~ \left(\frac{q_\beta}{\|q_\beta\|_1}\right)^{\otimes \nu} \hspace{-2mm}(x). \end{cases} \\
    &\\
    \qtens(x) \big| \{\widetilde n = \nu\} \; &= \begin{cases} \text{ if } \widetilde n(x) \neq \nu: & 0 \\\text{ otherwise: } & \hspace{-5mm}
    \sum\limits_{(\beta_j)_{j\geq U} \in \{0,1\}}~ \prod\limits_{j\geq U} ~ \pi_j^{\beta_j} (1 - \pi_j)^{1-\beta_j} ~ q_\beta^{\otimes \nu}(x). \end{cases} 
\end{align*}

Note that $q_b$ is not a density. Therefore the term $\qtens$ with $\widetilde n \sim \Poi(2n)$ denotes the mixture of inhomogeneous spatial Poisson processes with intensity functions $\left(2n\, q_b\right)_b$, over the realizations of $(b_j)_{j\in \Ibarzero}$ where we recall that $\Ibarzero = \{0\} \cup \{j\in \Ibar: j\geq U\}$. We define the histogram of $x$ over the domain $D(U)$ by setting for all $j\geq U: \widetilde N_j(x) = \sum_{i=1}^{\widetilde n(x)} \mathbb{1}_{x_i \in \widetilde C_j}$. \\

We have $R_B^{\; tail} = 1 - d_{TV}(\mathbb{P}_{p_0}, \mathbb{P}_{\bar p})$, where $\mathbb{P}_{p_0}$ and $\mathbb{P}_{\bar p}$ are respectively the probability measures of the densities $p_0^{\otimes n}$ and $\bar p$. We therefore aim at proving $d_{TV}(\mathbb{P}_{p_0}, \mathbb{P}_{\bar p}) < 1-\eta$. We will denote by $\Poi(f)$ the inhomogeneous spatial Poisson process with nonnegative intensity function $f$, by $f_{|\widetilde C_j}$ the restriction of $f$ to the cell $\widetilde C_j$. Moreover, for any two probability measures $P, Q$ over the same measurable space $(\mathcal{Y}, \mathcal{C})$, and for any event $A_0 \in \mathcal{C}$ we will denote by $d_{TV}^{A_0}(P,Q)$ the total variation restricted to $A_0$, defined as the quantity:
\begin{equation}\label{def_TV_event}
    d_{TV}^{(A_0)}(P,Q) = \sup_{A \in \mathcal{C}}\big|P(A \cap A_0) - Q(A \cap A_0) \big|.
\end{equation}
We prove the following lemmas concerning the total variation restricted to $A_0$:
\begin{lemma}\label{TV_restricted_abs}
For any two probability measures $P, Q$ over the same measurable space $(\mathcal{Y}, \mathcal{C})$, and for any event $A_0 \in \mathcal{C}$, if $P,Q \ll \mu$ over $(\mathcal{Y}, \mathcal{C})$, i.e. $dP = p d\mu$ and $dQ = q d\mu$, it holds: $$d_{TV}^{(A_0)}(P,Q) = \frac{1}{2}\left[|P(A_0) - Q(A_0)| + \int_{A_0} |p-q|d\mu \right].$$
\end{lemma}
\begin{lemma}\label{TV_restricted_tensorization}
For any two probability measures $P_1, Q_1$ (resp. $P_2, Q_2$) over the same measurable space $(\mathcal{Y}_1, \mathcal{C}_1)$(resp. $(\mathcal{Y}_2, \mathcal{C}_2)$), for any event $A_0 = \Azeroone \times \Azerotwo $ such that $\Azeroone \in \mathcal{C}_1$ and $\Azerotwo \in \mathcal{C}_2$, if $P_1,Q_1 \ll \mu_1$ (resp. $P_2,Q_2 \ll \mu_2$) over $(\mathcal{Y}_1, \mathcal{C}_1)$ (resp. $(\mathcal{Y}_2, \mathcal{C}_2)$), i.e. $dP_j = p_j d\mu_j$ and $dQ_j = q_j d\mu_j$ for $j=1, 2$, then it holds:
$$ d_{TV}^{(A_0)}(P_1 \otimes P_2,Q_1 \otimes Q_2) \leq  d_{TV}^{(\Azeroone)}(P_1,Q_1)+ d_{TV}^{(\Azerotwo)}(P_2,Q_2).$$
\end{lemma}

We now come back to the proof of Proposition \ref{lower_bound_tail}. We define $\forall j \geq U$:
\begin{align}
    \Aoncej ~~ &= \big\{x \in \Xtilde ~ \big| ~  \widetilde N_j(x)\leq 1\big\},\label{def_Aonce_j}\\
    \Aonce^{(0)} &= \Xtilde.\label{def_Aonce_M+1}
\end{align}

We also introduce
\begin{equation}\label{def_Aonce}
    \Aonce = \big\{x \in \Xtilde ~ \big| ~ \forall j \geq U: \widetilde N_j\leq 1\big\} = \bigcap_{j \geq U} \Aoncej,
\end{equation}
the subset of $\Xtilde$ of all the vectors of observations such that any cube $(\widetilde C_j)_{j \geq U}$ contains at most one observation. $\Aonce$ will play an essential role, as it is the high probability event on which we will place ourselves to approximate the total variation $d_{TV}(p_0^{\otimes n}, \bar p^{(n)})$.

In Lemmas \ref{TVvraie_TVpoisson}-\ref{TV_on_each_cell}, we will successively use equivalence of models to formalize the following (informal) chain of approximations: 
\begin{align*}
    d_{TV}(\mathbb{P}_{p_0}, \mathbb{P}_{\bar p}) & \lesssim d_{TV}(\bar p^{(\widetilde n')}, \ptens) = d_{TV}(\qtens, \ptens) \lesssim d_{TV}^{(\Aonce)}\big(\qtens, \ptens\big) \\
    & = d_{TV}^{(\Aonce)}\left( \bigotimes_{j \in \Ibarzero} \Poi\big(2n\, p_{0 | \widetilde C_j}\big), \bigotimes_{j \in \Ibarzero} \Poi\big(2n\, q_{b_j| \widetilde C_j}\big) \right) \\
    & \leq \sum_{j \geq U} d_{TV}^{(\Aoncej)}\left\{\Poi\Big(2n\, q_{b| \widetilde C_j}\Big), \; \Poi\Big(2n\, p_{0 | \widetilde C_j}\Big)\right\},
\end{align*} 
At each step, we will control the approximation errors. We recall that $\mathbb{P}_{p_0^{\otimes n}}$ and $ \mathbb{P}_{\bar p^{(n)}}$ are defined over $\Omega^n$ whereas $\ptens, \bar p_b^{(\widetilde n')}$ and $\qtens$ are defined on $\Xtilde = \bigcup\limits_{n \in \N} \Omega^n$.
More precisely we will prove the following lemmas:
\begin{lemma}\label{TVvraie_TVpoisson}
It holds $d_{TV}(\bar p^{(n)}, p_0^{\otimes n}) \leq d_{TV}(\bar p^{(\widetilde n')}, \ptens)/ \mathbb{P}(\An \cap \An')$.
\end{lemma}
Note that in the right hand side of Lemma \ref{TVvraie_TVpoisson}, we have $\widetilde n'$ observations for $ \bar p^{(\widetilde n')}$ and $\widetilde n $ for $p_0^{\otimes \widetilde n}$. 
\begin{lemma}\label{pbar=qbn'}
It holds $\overline p^{(\widetilde n')} = \qtens$.
\end{lemma}

\begin{lemma}\label{TV_over_A}
It holds: $d_{TV}(\qtens, \ptens) \leq d_{TV}^{(\Aonce)}\big(\qtens, \ptens\big)  + \qtens(\Aonce^c) + p_0^{\otimes \widetilde n}(\Aonce^c). $ 
\end{lemma}

\begin{lemma}\label{Tensorization_Poisson}
The following tensorization of the spatial Poisson processes holds:\\

$\Poi(2np_0)= \bigotimes\limits_{j \in \Ibarzero} \Poi(2np_{0 | \widetilde C_j})$ and $\Poi(2nq_{b}) = \bigotimes\limits_{j \in \Ibarzero} \Poi(2nq_{b_j| \widetilde C_j})$. Hence we have:
\begin{align*}
    d_{TV}^{(\Aonce)}\big(\qtens, \ptens\big) \leq \sum_{j \geq U} d_{TV}^{(\Aoncej)}\left(\Poi\big(2n\, q_{b| \widetilde C_j}\big), \; \Poi\big(2n\, p_{0 | \widetilde C_j}\big)\right).
\end{align*}
\end{lemma}

\vspace{3mm}

Furthermore, we will prove the following lemmas controlling the error at each step:

\vspace{3mm}

\begin{lemma}\label{An_high_pba}
There exists a constant $n_0 \in \N$ such that whenever $n \geq n_0$, it holds \\ $\mathbb{P}(\An), \mathbb{P}(\An') \geq 1- \eta/100$. 
\end{lemma}

\begin{lemma}\label{Control_Ac}
It holds $p_0^{\otimes n}(\Aonce^c) \leq \AL{\ref{Control_Ac}}$ and $q_b^{(n)}(\Aonce^c) \leq \BL{\ref{Control_Ac}}$ where $\AL{\ref{Control_Ac}}$ and $\BL{\ref{Control_Ac}}$ are two constants which \ssmall{successively $\cI, \cu$ and $\cdown$}. 
\end{lemma}
Finally we compute each of the terms in the last sum from Lemma \ref{Tensorization_Poisson}:
\begin{lemma}\label{TV_on_each_cell}
For all $j\geq U$ it holds $d_{TV}^{(\Aoncej)}\left(\Poi\big(2n\, q_{b| \widetilde C_j}\big), \; \Poi\big(2n\, p_{0 | \widetilde C_j}\big)\right) \leq \AL{\ref{TV_on_each_cell}} n^2p_j^2 + \BL{\ref{TV_on_each_cell}} \frac{p_j}{\sum\limits_{l \geq U} p_l}$ where $\AL{\ref{TV_on_each_cell}}$ and $\BL{\ref{TV_on_each_cell}}$ are two constants and $\BL{\ref{TV_on_each_cell}}$ \ssmall{$\cu$}. 
\end{lemma}
Bringing together Lemmas \ref{TVvraie_TVpoisson} - \ref{TV_on_each_cell}, we get: 
\begin{align}
    d_{TV}(\bar p^{(n)}, p_0^{\otimes n}) &\leq \Big[\AL{\ref{Control_Ac}} + \BL{\ref{Control_Ac}} + \AL{\ref{TV_on_each_cell}}\sum_{j \geq U} n^2p_j^2 + \BL{\ref{TV_on_each_cell}} \Big]/\mathbb{P}(\An \cap \An')\nonumber\\
    & \leq \Big[\AL{\ref{Control_Ac}} + \BL{\ref{Control_Ac}} + \AL{\ref{TV_on_each_cell}}\cu + \BL{\ref{TV_on_each_cell}} \Big]/\mathbb{P}(\An \cap \An'),\label{eq:TV_prior_tail}
\end{align}
where the right-hand side \ssmall{successively $\cI, \cu$ and $\cdown$}, which ends the proof of Proposition \ref{lower_bound_tail}.\\

We now prove Lemmas \ref{TV_restricted_abs} - \ref{TV_on_each_cell}.

\begin{proof}[Proof of Lemma \ref{TV_restricted_abs}]
Suppose by symmetry $P(A_0) \geq Q(A_0)$ and set $B_0 = \{x \in \mathcal{Y}: p(x) \geq q(x)\}$. Then we have:
\begin{align*}
    |P(B_0 \cap A_0) - Q(B_0 \cap A_0)| &= P(B_0 \cap A_0)-Q(B_0 \cap A_0) = \int_{B_0 \cap A_0} |p-q| d\mu\\ 
    & = P(A_0) - Q(A_0) + \int_{A_0 \setminus B_0} |p-q| d\mu,\\
    \text{so that } ~~  |P(B_0 \cap A_0) - Q(B_0 \cap A_0)| &= \frac{1}{2}\left[|P(A_0) - Q(A_0)| + \int_{A_0} |p-q|d\mu \right],\\
    \text{which yields: }~~ d_{TV}^{(A_0)}(P,Q) & \geq \frac{1}{2}\left[|P(A_0) - Q(A_0)| + \int_{A_0} |p-q|d\mu \right].
\end{align*}
Moreover, for any $B \in \mathcal{C}$, we consider $|P(B \cap A_0) - Q(B \cap A_0)|$. There are two cases.

\underline{First case:} $P(B \cap A_0) \geq  Q(B \cap A_0)$. Then we have:
\begin{align*}
    |P(B\cap A_0) - Q(B \cap A_0)| & = P(B \cap A_0) - Q(B \cap A_0) \\
    &= P(B \cap A_0 \cap B_0) - Q(B \cap A_0 \cap B_0) + \underbrace{ P(B \cap A_0 \setminus B_0) - Q(B \cap A_0 \setminus B_0)}_{\leq 0}\\
    & \leq P(B \cap A_0 \cap B_0) - Q(B \cap A_0 \cap B_0) \leq P( A_0 \cap B_0) - Q( A_0 \cap B_0)\\
    & = |P( A_0 \cap B_0) - Q(A_0 \cap B_0)|.
\end{align*}
\underline{Second case:} $Q(B \cap A_0) \geq  P(B \cap A_0)$. Then we have:
\begin{align*}
    |P(B\cap A_0) - Q(B \cap A_0)| & = Q(B \cap A_0) - P(B \cap A_0)| \\
    &=Q(B \cap A_0 \setminus B_0) - P(B \cap A_0 \setminus B_0)  + \underbrace{ Q(B \cap A_0 \cap B_0) - P(B \cap A_0 \cap B_0)}_{\leq 0}\\
    & \leq Q(B \cap A_0 \setminus B_0) - P(B \cap A_0 \setminus B_0) \leq Q( A_0 \setminus B_0) - P( A_0 \setminus B_0)\\
    & = \underbrace{Q(A_0) - P(A_0)}_{\leq 0} + P( A_0 \cap B_0) - Q(A_0 \cap B_0)\\
    & \leq |P( A_0 \cap B_0) - Q(A_0 \cap B_0)|.
\end{align*}
In both cases, the result is proven.
\end{proof}
\begin{proof}[Proof of Lemma \ref{TV_restricted_tensorization}]
We have by Lemma \ref{TV_restricted_abs}:
\begin{align*}
    2d_{TV}^{(A_0)}(P_1 \otimes P_2,Q_1 \otimes Q_2)  =& \; |P_1 \otimes P_2(A_0) - Q_1 \otimes Q_2(A_0)|+ \int_{A_0} |p_1(x) p_2(y) - q_1(x) q_2(y)| d\mu_1(x) d\mu_2(y)\\
    \leq & \; P_1(\Azeroone)\big|P_2(\Azerotwo) - Q_2(\Azerotwo)\big| + Q_2(\Azerotwo)\big|P_1(\Azeroone) - Q_1(\Azeroone)\big| \\
    & ~~ + P_1(\Azeroone) \int_{\Azerotwo}|p_2(y) - q_2(y)|d\mu_2(y) \\
    & ~~ + Q_2(\Azerotwo) \int_{\Azeroone}|p_1(x) - q_1(x)|d\mu_1(x)\\
    = & \; 2P_1(\Azeroone) d_{TV}^{(\Azerotwo)}(P_2,Q_2) + 2Q_2(\Azerotwo) d_{TV}^{(\Azeroone)}(P_1,Q_1)\\
    \leq & \; 2d_{TV}^{(\Azeroone)}(P_1,Q_1)+2d_{TV}^{(\Azerotwo)}(P_2,Q_2).
\end{align*}
\end{proof}

\begin{proof}[Proof of Lemma \ref{TVvraie_TVpoisson}]
We have:
\begin{align*}
    d_{TV}\big(\bar p_b^{(\widetilde n')}, \ptens\big) &= \sup_{A \in \Xtilde} \big|\bar p_b^{(\widetilde n')}(A) - \ptens(A) \big| \geq \sup_{A \in \Xtilde \cap \An \cap \An'} \big|\bar p_b^{(\widetilde n')}(A) - \ptens(A) \big|\\
    & = \mathbb{P}(\An\cap \An') \sup_{A \in \Xtilde} \big|\bar p_{b,\An'}^{(\widetilde n')}(A) - p_{0,\An}^{\otimes \widetilde n}(A) \big| ~~ \text{ where } \begin{cases}\bar p_{b,\An'}^{\, (\widetilde n')} =\bar p_b^{\, (\widetilde n')}(\cdot | \An')\\ \\p_{0,\An}^{\otimes \widetilde n} = \ptens(\cdot | \An) \end{cases}\\
    & = \mathbb{P}(\An \cap \An') \; d_{TV}\big(\bar p_{b,\An'}^{(\widetilde n')}, p_{0, \An}^{\otimes \widetilde n}\big).
\end{align*}
Furthermore, over $\An \cap \An'$ it holds $\widetilde n \geq n$ and $\widetilde n' \geq n$, so that $d_{TV}\big(\bar p_{b,\An}^{(\widetilde n)}, p_{0, \An}^{\otimes \widetilde n}\big) \geq d_{TV}\big(\bar p_b^{( n)}, p_0^{\otimes n}\big)$, which yields the result.
\end{proof}

\begin{proof}[Proof of Lemma \ref{pbar=qbn'}]
For fixed $b = (b_j)_{j \geq U}$ we have $\widetilde n' \sim \Poi(\|q_b\|_1)$. 
Moreover, we also have $\overline p_b^{(\widetilde n')} = \Poi \big(n \|q_b\|_1 \frac{q_b}{\|q_b\|_1}\big) = \Poi(nq_b) = q_b^{\otimes \widetilde n}$ so that taking the mixture over all realizations of $b$ yields that, unconditionally on $b$:  $\overline p_b^{(\widetilde n')} = \qtens$.
\end{proof}

\hfill

\begin{proof}[Proof of Lemma \ref{TV_over_A}]
We have:
\begin{align*}
    d_{TV}\big(\qtens, \ptens \big) &\leq \sup_{A \in \widetilde X \cap \Aonce} \big| \qtens(A) - \ptens(A)\big| + \qtens(\Aonce^c)+ \ptens(\Aonce^c) \\
    & = d_{TV}^{(\Aonce)}\big(\qtens, \ptens \big) + \qtens(\Aonce^c)+ \ptens(\Aonce^c).
\end{align*}

Hence the result.
\end{proof}

\begin{proof}[Proof of Lemma \ref{Tensorization_Poisson}]
To further transform the last quantity $d_{TV}^{(\Aonce)}\big(\qtens, \ptens\big)$ from Lemma \ref{TV_over_A}, we recall that 
drawing an inhomogeneous spatial Poisson process with intensity function $f$, defined on $\bigcup_{j \in \Ibarzero} \widetilde C_j$, is equivalent to drawing independently for each cell $\widetilde C_j, j \in \Ibarzero$, one inhomogeneous spatial Poisson process with intensity $f_{| \widetilde C_j}$. 
For any non-negative function $g$, denote by $\Poi(g)$ the spatial Poisson process with intensity function $g$. We can therefore re-index the data generated from $ p_0^{\otimes \widetilde n} = \Poi(np_0)$, as data generated by $\bigotimes_{j \in \Ibarzero} \Poi\big(np_{0| \widetilde C_j}\big)$ - and respectively $ \qtens$ as $\bigotimes_{j \in \Ibarzero} \Poi\big(nq_{b| \widetilde C_j}\big)$. Moreover, by independence of $(b_j)_{{j \in \Ibarzero}}$, the events $(\Aoncej)_{{j \in \Ibarzero}}$ defined in \eqref{def_Aonce_j} and \eqref{def_Aonce_M+1} are independent under both $\Poi(2np_0)$ and $\Poi(2n q_b)$ so that Lemma \ref{TV_restricted_tensorization} yields:
\begin{align*}
    d_{TV}^{(\Aonce)}\big(\qtens, \ptens\big) &\leq \sum_{j \in \Ibarzero} d_{TV}^{(\Aoncej)}\left(\Poi\big(2n\, q_{b| \widetilde C_j}\big), \; \Poi\big(2n\, p_{0 | \widetilde C_j}\big)\right)\\
    & = \sum_{j \geq U} d_{TV}^{(\Aoncej)}\left(\Poi\big(2n\, q_{b| \widetilde C_j}\big), \; \Poi\big(2n\, p_{0 | \widetilde C_j}\big)\right) ~~ \text{ since on $\widetilde C_0: \; p_0 = q_b$ for all $b$.}
\end{align*}

\end{proof}

\begin{proof}[Proof of Lemma \ref{An_high_pba}]
By Chebyshev's inequality:
\begin{align*}
    \mathbb{P}(\An^c) = \mathbb{P}(\Poi(2n) < n) \leq \mathbb{P}(|\Poi(2n)-2n| > n) \leq \frac{2n}{n^2} = \frac{2}{n}
\end{align*}
Moreover,
\begin{align*}
    \mathbb{P}({\An'}^c) &= \mathbb{P}(\Poi(2n\|q_b\|_1) < n) \leq \mathbb{P}\left(\Poi(2n\|q_b\|_1) < n \; \Big| \; \|q_b\|_1 \geq \frac{2}{3}\right) + \mathbb{P}\left(\|q_b\|_1 < \frac{2}{3} \right)\\
    & \leq \mathbb{P}\left(\Poi\Big(\frac{4}{3}n\Big) < n\right) +\mathbb{P}\left(\big|\|q_b\|_1 - 1\big| > \frac{1}{3} \right) \leq \frac{12}{n} + \frac{9\CL{\ref{norm_qb}}}{n^2} ~~ \text{ by Lemma \ref{norm_qb}}.
\end{align*}
Choosing $n_0$ such that $\frac{12}{n_0} + \frac{9\CL{\ref{norm_qb}}}{n_0^2} \leq \eta/100$, 
we get the result.
\end{proof}

\begin{proof}[Proof of Lemma \ref{Control_Ac}]
\begin{itemize}
    \item For the first quantity: 
    \vspace{-3mm}
    \begin{align*}
         p_0^{\otimes n}(\Aonce^c) &\leq \frac{1}{\mathbb{P}(\An)} \mathbb{E}_{\widetilde n}\left[\mathbb{P}_{\ptens}\left(\exists j \geq U: \widetilde N_j \geq 2 \;|\; \widetilde n \right) ~\Big|~ \widetilde n \geq n \right]\\
         & \leq \frac{1}{\mathbb{P}(\An)} \mathbb{E}_{\widetilde n}\left[\sum_{j \geq U} {\widetilde n}^2 p_j^2 ~\Big|~ \widetilde n \geq n \right] =  \frac{2n^2}{\mathbb{P}(\An)}\sum_{j \geq U} p_j^2 \\
         & \leq \frac{2n^2}{\mathbb{P}(\An)} \sum_{j \geq U} h^d \int_{\widetilde C_j}p_0^2 \leq 3 \CL{\ref{Toute_la_case_OK_moment2}},
    \end{align*}
    by the Cauchy-Schwarz inequality and Lemma \ref{Toute_la_case_OK_moment2}, and taking $n \geq n_0$. Setting $ \AL{\ref{Control_Ac}} = 3 \CL{\ref{Toute_la_case_OK_moment2}}$ yields the result.
    \item For the second quantity:
    \begin{align*}
        \qtens(\Aonce^c) &= \mathbb{E}_{b,\widetilde n}\left[\mathbb{P}_{\qtens}\left(\exists j \geq U: \widetilde N_j \geq 2 \;|\; \widetilde n \right) \right] \\
        & \leq \sum_{j \geq U} \mathbb{E}_{b_j,\widetilde n}\left[\mathbb{P}_{q_{b_j | \widetilde C_j}^{(\widetilde n)}}\left(\widetilde N_j \geq 2 \;|\; \widetilde n \right) \right] \leq \sum_{j \geq U} \mathbb{E}_{b_j,\widetilde n}\left[ {\widetilde n}^2 \left(\int_{\widetilde C_j} q_b\right)^2\right]\\
        & = 2n^2 \Big[\mathbb{V}\left[\|q_b\|_1  \right]+ \sum_{j \geq U} p_j^2 \Big] \leq 2\CL{\ref{norm_qb}} + 2 \CL{\ref{Toute_la_case_OK_moment2}} =: \BL{\ref{Control_Ac}}.
    \end{align*}
    Setting $\BL{\ref{Control_Ac}}:= 2\CL{\ref{norm_qb}} + 2 \CL{\ref{Toute_la_case_OK_moment2}} $ yields the result.
\end{itemize}
\end{proof}

\begin{proof}[Proof of Lemma \ref{TV_on_each_cell}]
We will use the notation 
\begin{align}
    &\pup_j = p_j + \Gammaup_j, \label{Def_pup}\\
    &\pdown_j = p_j - \Gammadown_j. \label{Def_down}
\end{align}
First, we show the following two facts:\\

\underline{\textbf{Fact 1}}: For all $a,b \geq 0$ such that $a+b=1$, and for all $x,y,z \in \R_+$ such that $x= ay +bz$, it holds $$ \big|e^{-x} - a e^{-y} - b e^{-z} \big| \leq \frac{x^2}{2} + a \frac{y^2}{2} + b \frac{z^2}{2}.$$
The proof of Fact 1 is straightforward by the relation $1-u \leq e^{-u} \leq 1-u+ \frac{u^2}{2}$ for all $u\geq 0$. \\

\underline{\textbf{Fact 2}}: We have $\mathbb{P}_{\ptens}\left(\widetilde N_j =1\right) = 2np_j e^{-2np_j}$ and moreover it holds: $$0 \; \leq \; \mathbb{P}_{\qtens}\left(\widetilde N_j =1\right) - 2np_j e^{-2np_j} \leq \Cq n^2p_j^2,$$ where 
$\Cq$ is a constant. \\

We now prove Fact 2. Under $\ptens$ we have that the number of observations is distributed as $\widetilde n \sim \Poi(2n)$ so that $\widetilde N_j \sim \Poi(2n \, p_j)$, hence $\mathbb{P}_{\ptens}\left(\widetilde N_j =1\right) = 2np_j e^{-2np_j}$.\\

Now, under $\qtens$, it holds: $\widetilde N_j \sim \pi_j \Poi\big(2n \pup_j\big) + (1-\pi_j) \Poi\big(2n\pdown_j\big)$. Therefore, 
\begin{align}
    \mathbb{P}_{\qtens}\left(\widetilde N_j =1\right) &= \pi_j 2n \pup_j e^{-2n\pup_j} + (1-\pi_j) 2n\pdown_j e^{-2n\pdown_j}\\
    & = e^{-2n p_j}\left[\pi_j 2n \pup_j e^{-2n\Gammaup_j} + (1-\pi_j) 2n\pdown_j e^{2n\Gammadown_j} \right], \label{qtens_1_observation}
\end{align}
hence $\mathbb{P}_{\qtens}\left(\widetilde N_j =1\right) \geq 2np_j e^{-2np_j}$ using the inequality $e^{x} \geq 1+x$. We now prove $\mathbb{P}_{\qtens}\left(\widetilde N_j =1\right) \leq 2np_j e^{-2np_j} + 8n^2p_j^2$. First, we have $2n\Gammaup_j \leq 1$ since $2n\Gammaup_j = 2n \frac{\Cint \cupward_j}{n^2 \int_{\T(\uA)}p_0} \leq 2\frac{\Cint \cu \cdown}{\ctail} \leq 1$ by choosing $\ctail$ large enough. Therefore, using the inequality $e^x \leq 1+2x$ for $0 \leq x \leq 1$, we get from Equation \eqref{qtens_1_observation}:
\begin{align*}
    \mathbb{P}_{\qtens}\left(\widetilde N_j =1\right) & \leq e^{-2np_j}\left( \pi_j 2n \pup_j + (1-\pi_j) 2n\pdown_j \big(1+ 4n\Gammadown_j\big)\right)\\
    & \leq 2np_j e^{-2np_j} + 8n^2 p_j \Gammadown_j \leq 2np_j e^{-2np_j} + 8 n^2p_j^2\\
    & =:2np_j e^{-2np_j} + 8n^2 p_j \Gammadown_j \leq 2np_j e^{-2np_j} + 8 n^2p_j^2,
\end{align*}
which ends the proof of Fact 2.\\

Facts $1$ and $2$ being established, we can now compute $d_{TV}^{(\Aoncej)}\left(\Poi\big(2n\, q_{b| \widetilde C_j}\big), \; \Poi\big(2n\, p_{0 | \widetilde C_j}\big)\right)$ for fixed $j \geq U$. By Lemma  \ref{TV_restricted_abs} we have:
\begin{align}
    &d_{TV}^{(\Aoncej)}\left(\Poi\big(2n\, q_{b| \widetilde C_j}\big), \; \Poi\big(2n\, p_{0 | \widetilde C_j}\big)\right)\nonumber\\
    & \leq \frac{1}{2}\Big[\Poi\big(2n\, q_{b| \widetilde C_j}\big)({\Aoncej}^c) + \Poi\big(2n\, p_{0 | \widetilde C_j}\big)({\Aoncej}^c
    ) + \int_{\Aoncej}\big|p_{0 | \widetilde C_j}^{\otimes \widetilde n}-q_{b| \widetilde C_j}^{(\widetilde n)}\big| \Big] \label{TV_Aj}
\end{align}
We now compute the term $\int_{\Aoncej}\big|p_{0 | \widetilde C_j}^{\otimes \widetilde n}-q_{b| \widetilde C_j}^{(\widetilde n)}\big|$.
\begin{equation}\label{Split_TV_Aj}
    \int_{\Aoncej}\big|p_{0 | \widetilde C_j}^{\otimes \widetilde n}-q_{b| \widetilde C_j}^{(\widetilde n)}\big| = \underbrace{\int_{\{\widetilde N_j = 0\}}\big|p_{0 | \widetilde C_j}^{\otimes \widetilde n}-q_{b| \widetilde C_j}^{(\widetilde n)}\big|}_{\text{Term 1}} ~+~ \underbrace{\int_{\{\widetilde N_j = 1\}}\big|p_{0 | \widetilde C_j}^{\otimes \widetilde n}-q_{b| \widetilde C_j}^{(\widetilde n)}\big|}_{\text{Term 2}}.
\end{equation}
\begin{align*}
    \text{Term 1} &= \big| \Poi\big(2n\, p_{0 | \widetilde C_j}\big)(\widetilde N_j=0) -  \Poi\big(2n\, q_{b | \widetilde C_j}\big)(\widetilde N_j=0) \big|\\
    & = \big|e^{-2np_j} - \pi_j e^{-2n\pup_j} - (1-\pi_j)e^{-2n\pdown_j} \big|\\
    & \leq 2n^2 p_j^2 + \pi_j 2n^2{\pup_j}^2 + (1-\pi_j)2n^2{\pdown_j}^2 ~~ \text{ using Fact 1}\\
    &\leq 4 n^2 p_j^2 + \pi_j 2n^2{\pup_j}^2.
\end{align*}
Moreover:
\begin{align}
    \pi_j 2n^2{\pup_j}^2 &= \Big[\frac{1}{2\cu} p_jn^2\sum_{l \geq U} p_l\Big] 2n^2\left[p_j + \frac{\cupward_j \Cint}{n^2 \int_{\T(\uA)}p_0}\right]^2 \nonumber \\
    & \leq \Big[\frac{1}{2\cu} p_j\, n^2\sum_{l \geq U} p_l\Big] 2n^2 \Bigg[\frac{1}{n^2 \sum\limits_{l \geq U} p_l} \Bigg]^2 \left[\cu + \cupward_j \Cint \CL{\ref{Toute_la_case_OK_moment1}} \right]^2 \text{ by Lemma \ref{Toute_la_case_OK_moment1}} \nonumber\\
    & \leq \frac{p_j}{\sum\limits_{l \geq U} p_l} \cu\left[1 + 2 \cdown \Cint \CL{\ref{Toute_la_case_OK_moment1}} \right]^2 =: \Cpiup \frac{p_j}{\sum\limits_{l \geq U} p_l}, \label{pi_j_n2_pj2}
\end{align}
where $\Cpiup$ \ssmall{$\cu$}. It follows that 
\begin{equation}\label{Term1}
    \text{Term 1} \leq 4n^2 p_j^2 + \Cpiup\frac{p_j}{\sum\limits_{l \geq U} p_l} .
\end{equation}

We now consider Term 2. We set $\pup(x) = p_0(x) + \gammaup_j(x)$ and $\pdown(x) = p_0(x) - \gammadown_j(x)$.
\begin{align}
    \text{Term 2} =& \int_{\{\widetilde N_j = 1\}}\big|p_{0 | \widetilde C_j}^{\otimes \widetilde n}-q_{b| \widetilde C_j}^{(\widetilde n)}\big| = \int_{\widetilde C_j} \big|p_0(x)\mathbb{P}_{\ptens}\left(\widetilde N_j =1\right) - q_b(x)\mathbb{P}_{\qtens}\left(\widetilde N_j =1\right) \big|dx \nonumber\\
    = & \int_{\widetilde C_j} \big|p_0(x)\mathbb{P}_{\ptens}\left(\widetilde N_j =1\right) - \left(\pi_j \pup(x) + (1-\pi_j)\pdown(x)\right)\mathbb{P}_{\qtens}\left(\widetilde N_j =1\right) \big|dx \nonumber\\
    \leq & \; p_j\left|\mathbb{P}_{\ptens}\left(\widetilde N_j =1\right) - \mathbb{P}_{\qtens}\left(\widetilde N_j =1\right) \right| \nonumber\\
    & + \mathbb{P}_{\qtens} \left(\widetilde N_j =1\right) \int_{\widetilde C_j} \big|\left(\pi_j \gammaup_j(x) - (1-\pi_j)\gammadown_j(x)\right)\big|dx \nonumber\\
    \leq & \, \Cq n^2p_j^3 + \mathbb{P}_{\qtens} \left(\widetilde N_j =1\right) 2 \Gammadown_j ~ \text{ by Fact 2 and recalling } \pi_j\Gammaup_j = (1-\pi_j) \Gammadown_j\nonumber\\
    \leq & ~ \Cq n^2p_j^3 + \Big(2n p_j + \Cq n^2p_j^2\Big) 2 p_j ~~ \text{ by Fact 2 and recalling $\Gammadown_j \leq p_j$}\nonumber\\
    \leq & ~ (3\Cq+4) n^2p_j^2. \label{Term2}
\end{align}
We now control the term $\Poi\big(2n\, q_{b| \widetilde C_j}\big)({\Aoncej}^c) + \Poi\big(2n\, p_{0 | \widetilde C_j}\big)({\Aoncej}^c)$ from equation \eqref{TV_Aj}. We have:
\begin{align}
    \Poi\big(2n\, p_{0 | \widetilde C_j}\big)({\Aoncej}^c)&= \mathbb{P}(\Poi(2np_j) \geq 2) = 1-e^{-2np_j} (1+2np_j)\nonumber \\ &\leq 1-(1-2np_j)(1+2np_j) = 4n^2 p_j^2. \label{A_jc_p0}
\end{align}
\begin{align}
    \text{Moreover we have: } ~~ &\Poi\big(2n\, q_{b| \widetilde C_j}\big)({\Aoncej}^c) = \mathbb{P}\left(\Poi(2n\int_{\widetilde C_j}q_b) \geq 2\right)\nonumber \\
    &= \pi_j \mathbb{P}(\Poi(2n\pup_j) \geq 2) + (1-\pi_j)\mathbb{P}(\Poi(2n\pdown_j) \geq 2)~~~~~~~~~~~~~~~~~~~~~~~~ \nonumber \\
    & \leq 4\pi_j n^2 {\pup_j}^2 + (1-\pi_j) 4n^2 {\pdown_j}^2 \text{ by \eqref{A_jc_p0}}\nonumber \\
    & \leq 2\Cpiup \frac{p_j}{\sum\limits_{l \geq U} p_l} + 4n^2 p_j^2 ~~ \text{ by \eqref{pi_j_n2_pj2}.} \label{A_jc_qb}
\end{align}
Bringing together equations \eqref{TV_Aj}, \eqref{Split_TV_Aj}, \eqref{Term1}, \eqref{Term2}, \eqref{A_jc_p0} and \eqref{A_jc_qb}, we get:
\begin{align*}
    &d_{TV}^{(\Aoncej)}\left(\Poi\big(2n\, q_{b| \widetilde C_j}\big), \; \Poi\big(2n\, p_{0 | \widetilde C_j}\big)\right) \leq \\
    &\frac{1}{2} \left[ 4n^2 p_j^2 + 2\Cpiup \frac{p_j}{\sum\limits_{l \geq U} p_l} + 4n^2 p_j^2 + 4n^2 p_j^2 + 2\Cpiup\frac{p_j}{\sum\limits_{l \geq U} p_l} + (3\Cq + 4)n^2p_j^2\right]\\
    & = (8 + \frac{3}{2}\Cq)n^2p_j^2 + \frac{3}{2}\Cpiup\frac{p_j}{\sum\limits_{l \geq U} p_l}\\
    & =: \AL{\ref{TV_on_each_cell}} \; n^2p_j^2 + \BL{\ref{TV_on_each_cell}}\frac{p_j}{\sum\limits_{l \geq U} p_l}.
\end{align*}
\end{proof}
which ends the proof of Proposition \ref{lower_bound_tail}.

\end{proof}


\begin{lemma}\label{LB_prior_restricted}
Denote by $\mathbb{P}_b$  the probability distribution over the realizations of the random variable $b$. Set $\Asep = \{b \text{ is such that } \|p_0 - p_b\|_t \geq \CtailLB \rhot\}$ and $p_{b,cond} = \mathbb{E}(p_b^{(n)} | \Asep)$ where the expectation is taken according to the realizations of $b$. Suppose $\mathbb{P}_b(\Asep^c) \leq \epsilon$. Then $d_{TV}(p_0^{\otimes n}, p_{b,cond}) \leq d_{TV}(p_0^{\otimes n}, \mathbb{E}_b(p_b^{(n)})) + 2 \epsilon$.
\end{lemma}
\begin{proof}[Proof of Lemma]
We have:
\begin{align*}
    d_{TV}\big(p_{b,cond}, \mathbb{E}_b(p_b^{(n)})\big) &= \sup_{A} \big| p_{b,cond}(A) - p_{b,cond}(A)\mathbb{P}(\Asep) - \mathbb{E}\big[p_b^{(n)}(A|\Asep^c)\big]\mathbb{P}(\Asep^c) \big|\\
    &\leq \sup_{A} \big| p_{b,cond}(A) (1- \mathbb{P}(\Asep)) \big| + \mathbb{P}(\Asep^c)  \leq 2 \epsilon,\\
    \text{so that: } ~~ d_{TV}\big(p_0^{\otimes n}, p_{b,cond}\big) &\leq d_{TV}\big(p_0^{\otimes n}, \mathbb{E}_b(p_b^{(n)})\big) + d_{TV}\big(\mathbb{E}_b(p_b^{(n)}), p_{b,cond}\big)\\
    &\leq d_{TV}\big(p_0^{\otimes n}, \mathbb{E}_b(p_b^{(n)})\big) + 2\epsilon.
\end{align*}
\end{proof}

\subsection{Technical results}

\begin{lemma}\label{very_small_dominate}
It holds: $\int_{\T(2\Cmom L h^\alpha)} p_0 \geq \frac{1}{2} \int_{\T(\uA)} p_0.$
\end{lemma}

\begin{proof}[Proof of Lemma \ref{very_small_dominate}]
We have:
\begin{align*}
    2\Cmom L h^\alpha \int_{\T(\uA) \setminus \T(2\Cmom L h^\alpha)} p_0 \leq \int_{\T(\uA)} p_0^2 \leq \frac{\Cmom}{n^2h^d} ~~ \text{ by Lemma \ref{int_p0squared_uA_small}.}
\end{align*}
Therefore:
\begin{align*}
     \int_{\T(\uA)} p_0 - \int_{\T(2\Cmom L h^\alpha)} p_0 \leq \frac{1}{2} \int_{\T(\uA)} p_0.
\end{align*}
\end{proof}

\begin{lemma}\label{pU_leq_Lh^alpha+d}
It holds $p_U < \CL{\ref{pU_leq_Lh^alpha+d}}Lh^{\alpha+d}$ where $\CL{\ref{pU_leq_Lh^alpha+d}} = 2\Cmom(1+\cstar) + \sqrt{d}^\alpha$ and $h = \hT(\uA)$. 
\end{lemma}

\begin{proof}[Proof of Lemma \ref{pU_leq_Lh^alpha+d}]
Fix $j \in \Ibar$ such that $\widetilde C_j \cap \T(2\Cmom Lh^{\alpha+d}) \neq \emptyset$ and $x \in \widetilde C_j$ such that $p_0(x) < 2\Cmom Lh^{\alpha+d}$. Then by Assumption \eqref{simplifyingAssumption}, for all $y \in \widetilde C_j$
\begin{align*}
    p_0(y) \leq (1+\cstar)p_0(x) + Lh^\alpha \sqrt{d}^\alpha < \CL{\ref{pU_leq_Lh^alpha+d}} Lh^\alpha
\end{align*}
so that $p_j \leq \CL{\ref{pU_leq_Lh^alpha+d}} L^{\alpha+d}$. Therefore, if we had $p_U \geq \CL{\ref{pU_leq_Lh^alpha+d}} Lh^{\alpha+d}$, then necessarily, $\T(2\Cmom Lh^\alpha) \subset \bigcup\limits_{j\geq U} \widetilde C_j$, hence:
\begin{align*}
    \frac{\cu}{n^2} \geq p_U \sum_{j\geq U} p_j \geq \CL{\ref{pU_leq_Lh^alpha+d}} Lh^{\alpha+d} \int_{\T(2\Cmom Lh^\alpha)} p_0 \geq \frac{\CL{\ref{pU_leq_Lh^alpha+d}}}{2n^2}>\frac{\cu}{n^2}
\end{align*}
for $\cu$ small enough. Contradiction.
\end{proof}

\begin{lemma}\label{only_small_jumps}
Set $h=\hT(\uA)$ and assume that the tail dominates i.e. $\CBT \rhob \leq \rhot$. For all $j \in \Ibar$, $j \geq 2$, if $p_j\leq \CL{\ref{pU_leq_Lh^alpha+d}} Lh^{\alpha + d}$ then $ p_{j-1} \leq \CL{\ref{only_small_jumps}} L h^{\alpha + d}$ where $\CL{\ref{only_small_jumps}}$ is a constant depending only on $\CL{\ref{pU_leq_Lh^alpha+d}}$. 
\end{lemma}
\begin{proof}[Proof of Lemma \ref{only_small_jumps}]
Let $j \in \Ibar$, $j \geq 2$ and $z_j \in \widetilde C_j$ such that $p_0(z_j) h^d = p_j$ and assume $p_j < \CL{\ref{pU_leq_Lh^alpha+d}} Lh^{\alpha+d}$. 
Set $\CL{\ref{only_small_jumps}}' = 4 (1+\cstar) \CL{\ref{pU_leq_Lh^alpha+d}} + \sqrt{d}^\alpha$.\\
\vspace{-3mm}

Let $y \in \T(\uA)$ such that $p_0(y) = \CL{\ref{only_small_jumps}}' L h^\alpha$. 
We can assume $\CL{\ref{only_small_jumps}}' L h^\alpha < \uA$ by choosing $\CBT$ large enough.
Indeed, by Lemma \ref{htail_leq_hbulk}, we have $\uA \geq \cA L\inf\limits_{x \in \B} \hB^\alpha(x) \geq \cA L \CBT^{(2)\alpha} \hT^\alpha(\uA)$ and choosing $\CBT$ large enough ensures that $\CBT^{(2)\alpha}$ is large enough. Denote by $l$ the index of the cube $\widetilde C_l$ containing $y$. 
\begin{itemize}
    \item First, $p_l > p_j$. Indeed, for all $z \in \widetilde C_l$ we have
    \begin{align*}
        p_0(z) \geq (1-\cstar)\CL{\ref{only_small_jumps}}' L h^\alpha - Lh^\alpha \sqrt{d}^\alpha \geq (3 \CL{\ref{pU_leq_Lh^alpha+d}} - \sqrt{d}^\alpha)Lh^\alpha \geq 2 \CL{\ref{pU_leq_Lh^alpha+d}} Lh^\alpha,
    \end{align*}
    hence $p_l = \int_{\widetilde C_j} p_0(z)dz \geq 2 \CL{\ref{pU_leq_Lh^alpha+d}} Lh^ {\alpha+d} > p_j$.
    \item Second, for all $z \in \widetilde C_l$ we have $p_0(z) \leq (1+\cstar)\CL{\ref{only_small_jumps}}' Lh^\alpha + \sqrt{d}^\alpha Lh^\alpha =: \CL{\ref{only_small_jumps}} Lh^\alpha$ hence $p_l \leq \CL{\ref{only_small_jumps}}Lh^\alpha$. 
\end{itemize}
Since the $(p_l)_l$ are sorted in decreasing we also have $p_{j-1} \leq \CL{\ref{only_small_jumps}} Lh^\alpha$.
\end{proof}

\begin{lemma}\label{sum_pj_geq_int}
Suppose that the tail dominates, i.e.: $\rhot \geq \CBT \rhob$. There exists a constant $D$ 
such that whenever $\int_{\T(\uA)} p_0 \geq \frac{\ctail}{n}$, it holds: $$\sum_{j\geq U} p_j \geq D \int_{\T(\uA)} p_0.$$ Moreover, $D$ \ssmall{$\cu$} and $\ctail$ large enough, successively.
\end{lemma}


\begin{proof}[Proof of Lemma \ref{sum_pj_geq_int}]
Recall that by Lemma \ref{pU_leq_Lh^alpha+d}, $p_U \leq  \CL{\ref{pU_leq_Lh^alpha+d}} Lh^{\alpha+d}$.
Therefore, we cannot have $U=1$.
Indeed, there always exists $j \in \Ibar$ such that for some $x \in \widetilde C_1, p_0(x) = \uA$ and for this index $j$:
\begin{align}
    \forall y \in \widetilde C_j: p_0(y) \geq (1-\cstar) \uA - Lh^\alpha \sqrt{d}^\alpha > \CL{\ref{pU_leq_Lh^alpha+d}} Lh^{\alpha}\label{case_inter_uA}
\end{align}
for $\CBT$ large enough, by Lemma \ref{htail_leq_hbulk} and recalling $\uA \geq \cA L \min\limits_\B \hB^\alpha$. Therefore, $p_1 >  \CL{\ref{pU_leq_Lh^alpha+d}} Lh^{\alpha+d} \geq p_U$ hence $U \geq 2$. We can then write, by definition of $U$:
\begin{align*}
    p_{U-1} \sum_{j \geq U-1} p_j > \frac{\cu}{n^2},
\end{align*} where $p_{U-1} \leq \CL{\ref{only_small_jumps}} L h^{\alpha+d}$ by Lemma \ref{only_small_jumps}. Therefore:
\begin{align*}
    \sum_{j \geq U} p_j &> \frac{\cu}{n^2 p_{U-1}} - p_{U-1} \geq \frac{\cu}{n^2 \CL{\ref{only_small_jumps}} L h^{\alpha + d}} - \CL{\ref{only_small_jumps}} L h^{\alpha + d}\\
    & = \frac{\cu}{\CL{\ref{only_small_jumps}}} \int_{\T(\uA)} p_0 - \frac{\CL{\ref{only_small_jumps}}}{n^2 \int_{\T(\uA)} p_0}\\
    & \geq \frac{\cu}{\CL{\ref{only_small_jumps}}} \int_{\T(\uA)} p_0 - \frac{\CL{\ref{only_small_jumps}}}{\ctail^2} \int_{\T(\uA)} p_0\geq D \int_{\T(\uA)} p_0.
\end{align*}
Choosing $\ctail^2 \geq \frac{\CL{\ref{only_small_jumps}}^2}{2\cu}$ yields the result with $D = \frac{\cu}{\CL{\ref{only_small_jumps}}} - \frac{\CL{\ref{only_small_jumps}}}{\ctail^2}$.

\end{proof}

\begin{lemma}\label{Toute_la_case_OK_moment1}
In the case where the tail dominates, \textit{i.e.} when $\rhot \geq \CBT \rhob$, there exists a constant $\CL{\ref{Toute_la_case_OK_moment1}}$ such that $\sum\limits_{j \geq U} p_j \leq \CL{\ref{Toute_la_case_OK_moment1}}  \int_{\T(\uA)} p_0$. 
\end{lemma}

\begin{proof}[Proof of Lemma \ref{Toute_la_case_OK_moment1}]
Set $h = \hT(\uA)$.
Proceeding like in equation \eqref{case_inter_uA}, it is impossible that $\exists j \geq U,\; \exists x \in \widetilde C_j: p_0(x) = \uA$. For all $j \geq U$, we therefore have: $\sup_{\widetilde C_j} p_0 < v \leq \uA$, hence $\bigcup_{j \geq U} \widetilde C_j \subset \T(\uA)$, which yields $\int_{\bigcup_{j \geq U} \widetilde C_j} p_0 \leq \int_{\T(\uA)} p_0$.
\end{proof}

\begin{lemma}\label{Toute_la_case_OK_moment2}
Whenever the tail dominates, \textit{i.e.} when $\rhot \geq \CBT \rhob$, there exists a constant $\CL{\ref{Toute_la_case_OK_moment2}}$ such that $\int_{\bigcup\limits_{j \geq U} \widetilde C_j} p_0^2 \leq \CL{\ref{Toute_la_case_OK_moment2}} \, \frac{1}{n^2 h^d}$.
\end{lemma}

\begin{proof}[Proof of Lemma \ref{Toute_la_case_OK_moment2}]
Set $h = \hT(\uA)$. Proceeding like in equation \eqref{case_inter_uA}, we have: $\sup_{\widetilde C_j} p_0 < v \leq \uA$, hence $\bigcup\limits_{j \geq U} \widetilde C_j \subset \T(\uA)$, which yields by Lemma \ref{int_p0squared_uA_small}: $\int_{\bigcup\limits_{j \geq U} \widetilde C_j} p_0^2 \leq \frac{\Cmom}{n^2 h^d}$.
\end{proof}

\begin{lemma}\label{control_norm_t_Rd}
Let $p \in \mathcal{P}_{\R^d}(\alpha,L,\cstar)$ and $ \Omega' \subset \Omega$ a countable union of cubic domains of $\R^d$.
\begin{enumerate}
    \item If $\int_{\Omega'} p \leq \frac{c}{n}$ for some constant $c>0$, then 
    $$ \left(\int_{\Omega'} p^t\right)^{1/t} \leq \AL{\ref{control_norm_t_Rd}}(c) \cdot \rhor,$$
    for $\AL{\ref{control_norm_t_Rd}}(c) $ a constant depending only on $c, \eta, d, \alpha$ and $t$. Moreover, we have $\AL{\ref{control_norm_t_Rd}}(c) \underset{c \to 0}{\to} 0$ and $\AL{\ref{control_norm_t_Rd}}(c) \underset{c \to + \infty}{\to} + \infty$.
    \item There exists a constant $\BL{\ref{control_norm_t_Rd}}$ such that $\|p\|_t \leq \BL{\ref{control_norm_t_Rd}} \cdot L^{\frac{d(t-1)}{t(\alpha+d)}}$.
\end{enumerate}
\end{lemma}

\begin{proof}[Proof of Lemma \ref{control_norm_t_Rd}]
Let $x \in \R^d \setminus \Omega'$ and let $h = \left(\frac{p(x)}{4L}\right)^{1/\alpha}$. By Assumption \eqref{simplifyingAssumption}, we have for all $y \in B(x,h)$:
\begin{align*}
    p(y) \geq \frac{p_0(x)}{2} -L\left(\frac{p(x)}{4}\right) = \frac{p(x)}{4}.
\end{align*}
Moreover, by assumption over $\Omega'$, $\text{Vol}\left(B(x,h) \cap (\R^d \setminus \Omega')\right) \geq \frac{1}{2}\text{Vol}(B(x,h))$. Therefore:
\begin{equation}\label{majorant_p_Rd}
    \int_{\R^d \setminus \Omega'} p ~\geq ~ \frac{p(x)}{4}\frac{1}{2}\text{Vol}(B(x,h)) ~ = ~ \Cd \frac{p(x)^{\frac{\alpha+d}{\alpha}}}{L^{d/\alpha}},
\end{equation}
where $\Cd = \frac{\text{Vol}\left(B(0,1)\right)}{8 \times 4^{d/\alpha}}$. 
\begin{enumerate}
    \item Therefore, if $\int_{\R^d \setminus \Omega'}p \leq \frac{c}{n}$, then $p(x) \leq \left(\frac{c}{\Cd}\frac{L^d}{n^\alpha}\right)^\frac{1}{\alpha+d}$, which yields:
\begin{align*}
    \int_{\R^d \setminus \Omega'} p^2 \leq \left(\frac{c}{\Cd}\frac{L^d}{n^\alpha}\right)^\frac{1}{\alpha+d} \times \frac{c}{n} = c \left(\frac{c}{\Cd}\right)^\frac{1}{\alpha+d}\frac{L^\frac{d}{\alpha+d}}{n^\frac{2\alpha + d}{\alpha + d}}.
\end{align*}
Now, by Hölder's inequality, we have
\begin{align*}
    \int_{\R^d \setminus \Omega'} p^t \leq \left(\int_{\R^d \setminus \Omega'} p\right)^{2-t}\left(\int_{\R^d \setminus \Omega'} p^2\right)^{t-1} =: \AL{\ref{control_norm_t_Rd}}^t(c) {\rhor}^t
\end{align*}
    \item The proof of the second assertion follows the same lines. We have by Equation \eqref{majorant_p_Rd} that $\forall x \in \R^d : 1=\int_{\R^d}p \geq \Cd \frac{p(x)^{\frac{\alpha+d}{\alpha}}}{L^{d/\alpha}}$, hence $p(x) \leq \frac{1}{\Cd^{\alpha/(\alpha+d)}}L^{\frac{d}{\alpha + d}}$ for all $x \in \R^d$. Therefore, $\int_{\R^d} p^2 \leq \frac{1}{\Cd^{\alpha/(\alpha+d)}}L^{\frac{d}{\alpha + d}} \int_{\R^d} p \leq \frac{1}{\Cd^{\alpha/(\alpha+d)}}L^{\frac{d}{\alpha + d}}$, so that by Hölder's inequality:
    \begin{align*}
    \int_{\R^d } p^t \leq \left(\int_{\R^d} p\right)^{2-t}\left(\int_{\R^d} p^2\right)^{t-1} = \Cd^{\frac{\alpha(1-t)}{\alpha+d}}L^{\frac{d(t-1)}{\alpha+d}} =: \BL{\ref{control_norm_t_Rd}}^t L^{\frac{d(t-1)}{\alpha+d}}.
\end{align*}
\end{enumerate}
\end{proof}

\boundedcase{\section{Analysis of the degenerate cases }\label{degenerate_case_appendix}
Throughout the section, assume $\Omega = [-\homega/2, \homega/2]$. Define $\text{diam}\; \mathcal{P}(\alpha, L ,\cstar) = \sup \big\{\|p-q\|_t~|~ p,q \in \mathcal{P}(\alpha,L,\cstar)\big\}$.

\begin{proposition}\label{degenerate_case_proposition}
Assume that $\alpha > 1$ and $\sup \big\{\|p-q\|_\infty~|~ p,q \in \mathcal{P}(\alpha,L,\cstar)\big\} \leq \homega^{-d}/2$. If $\text{diam}\; \mathcal{P}(\alpha,L,\cstar) \geq \frac{c}{\sqrt{n}\homega^{d-d/t}}$ for some $c > 0$ then for all $p_0 \in \mathcal{P}(\alpha,L,\cstar)$: $\rho^*(p_0) \asymp \frac{1}{\sqrt{n}\homega^{d-d/t}}$.
\end{proposition}

\subsection{Proof of Proposition \ref{degenerate_case_proposition}: Upper bound}
In this subsection, we set $h = \cch \homega$ for $\cch$ a constant, sufficiently small. Define $\hat \Delta(x) = \frac{1}{k} \sum\limits_{i=1}^k K_{\homega}(x-X_i) - p_0(x)$ and $\hat \Delta'(x) = \frac{1}{k} \sum\limits_{i=k+1}^n K_{\homega}(x-X_i) - p_0(x)$ and $\taun = \Ctaun (nh^d)^{-1}$ for $\Ctaun$ a large constant. Define the test statistic
\begin{equation}\label{def_statistic_degen}
    \Tdeg = \int_{\Omega} \hat{\Delta}(x)\hat{\Delta}'(x) dx
\end{equation}
and the test statistic
\begin{equation}\label{def_test_degen}
    \psidegen = \mathbb{1}\left\{\Tdeg > \taun\right\}.
\end{equation}

\begin{lemma}\label{upper_bound_degen}
\begin{enumerate}
    \item Under $H_0$, $\mathbb{P}_{p_0}(\psidegen=1) \leq \frac{\eta}{2}$.
    \item There exists a large constant $C>0$ such that whenever $p \in \mathcal{P}(\alpha,L,\cstar)$ satisfies $\|\Delta\|_t \geq \frac{C}{\sqrt{n} \homega^{d-d/t}}$, we have $\mathbb{P}_{p}(\psidegen=0) \leq \frac{\eta}{2}$.
\end{enumerate}
\end{lemma}

\begin{proof}[Proof of Lemma \ref{upper_bound_degen}]
We essentially follow the same steps as for the proof of the bulk upper bound by replacing $\omega$ by $1$ and $\TB$ by $\Omega$. \\

\begin{enumerate}
    \item Under $H_0$, we have
\begin{align*}
    \mathbb{E}_{p_0}\Tdeg = \int_\Omega \mathbb{E}_{p_0}^2[\hat{\Delta}(x)] dx \leq \CK^2 \cch^2 L \homega^{2\alpha+d} \leq \frac{\taun}{2} ~~ \text{ for $\Ctaun$ large enough. Moreover,}
\end{align*}

\begin{align*}
    \mathbb{V}_{p_0}\Tdeg \leq \mathbb{E}_{p_0}(\Tdeg^2) = \iint_{\Omega^2} \mathbb{E}_{p_0}^2\left[\hat{\Delta}(x)\hat{\Delta}(y)\right] dx dy \leq  (J_1^{1/2} + J_2^{1/2})^2,
\end{align*}
by Equation \eqref{exp_delta2_delta2}, where $ J_1 $ and $ J_2$ are obtained from the expressions \eqref{def_J1} and \eqref{def_J22} by replacing $\omega$ by $1$ and $\TB$ by $\Omega$. We get
\begin{align*}
    \sqrt{J_1} &= \int_\Omega \mathbb{E}_{p_0}^2\left(K_{\homega}(x-X) - p_0(x)\right) \leq \CK^2 \cch^2 L^2 \homega^{2\alpha+d},\\
    J_2 &= \frac{1}{k^2} \iint_{\Omega^2} \text{cov}^2\left(K_{\homega}(x-X),K_{\homega}(y-X)\right) dx dy\\
    & \leq \frac{1}{k^2} \iint_{\Omega^2} \mathbb{V}_{p_0}\left(K_{\homega}(x-X)\right)\mathbb{V}_{p_0}\left(K_{\homega}(y-X)\right) dx dy\\
    & \leq \frac{1}{k^2} {\CK^{(2)}}^2 \iint_{\Omega^2} \frac{\max p_0^2}{\homega^{2d}} dx dy\\
    & \leq \frac{1}{k^2}{\CK^{(2)}}^2 \frac{4}{h^{2d}} \ll \taun ~~ \text{ for $\Ctaun$ large enough,}
\end{align*}
since $\max_{p_0} \leq 2\homega^{-d}$. Hence, $\mathbb{V}_{p_0}(\Tdeg) \leq \frac{\eta}{2} \taun^2$ for $\Ctaun$ large enough. Now, by Chebyshev's inequality, $\mathbb{P}_{p_0}(\psidegen=1) = \mathbb{P}_{p_0}\left(\Tdeg>\mathbb{E}\Tdeg + \taun \right) \leq \mathbb{V}\Tdeg/\taun^2 \leq \eta/2$.
\item Under the alternative, assume $p \in \mathcal{P}(\alpha,L,\cstar)$ is such that $\|\Delta\|_t \geq \frac{C}{\sqrt{n} \homega^{d-d/t}}$ for some sufficiently large constant $C>0$. By Hölder's inequality:
\begin{align*}
    \|\Delta\|_2^t\; \homega^{d \frac{2-t}{2}} \geq \|\Delta\|_t^t \geq \frac{C^t}{\sqrt{n}^t \homega^{dt-d}} ~~ \text{ hence } ~~ \|\Delta\|_2^2 \geq \frac{C^2}{n\homega^d}.
\end{align*}
\begin{align*}
    \text{Now, } ~~ \mathbb{E}_p \Tdeg &= \int_\Omega \mathbb{E}_p\left[\hat{\Delta}(x)\right]^2 dx \geq \int_\Omega \left(|\Delta(x)| - \CK L\homega^\alpha \right)^2 dx\\
    & = \big\|\,|\Delta| - \CK L \homega^\alpha\, \big\|_{L_2}^2 \geq \big|\|\Delta\|_{L_2} - \|\CK L \homega^\alpha\|_{L_2}\big|^2\\
    & \geq \left|\frac{C}{\sqrt{n}\homega^{d/2}} - \frac{C}{2\sqrt{n}\homega^{d/2}} \right|^2 = \frac{C^2}{4n\homega^d},
\end{align*}
where at the last inequality we used $L^2 \homega^{2\alpha+d} \leq \csmall^2/\left(n^2\homega^d\right)$ and chose $\csmall$ small enough. Next, by Lemma \ref{splitting_variance}:
\begin{align*}
    \mathbb{V}\Tdeg \leq \left[(\|\Delta\|_2 + \taun)^2 + \sqrt{J_2}\right]^2 - \mathbb{E}^2\Tdeg,
\end{align*}
where, similarly as above, $J_2 \leq \frac{16 {\CK^{(2)}}^2}{n^2\homega^{2d}} \leq \taun^2$ for $\Ctaun$ large enough. Moreover, 
\begin{align*}
    \mathbb{E}^2(\Tdeg) \geq \left(\|\Delta\|_2^2 - \CK^2 L^2 \homega^{2\alpha+d}\right)^2 \geq \left(\|\Delta\|_2^2 - \taun\right)^2 ~~ \text{ for $\Ctaun$ large enough.}
\end{align*}
Hence $\mathbb{V}(\Tdeg) = O(\|\Delta\|_2^3 \taun)$ as $C \to +\infty$ and when $\Ctaun$ is fixed. Taking successively $\Ctaun$ and $C$ large enough, $\mathbb{V} \Tdeg \leq \frac{\eta}{4}\mathbb{E}^2\Tdeg$. We conclude by applying Chebyshev's inequality:
\begin{align*}
    \mathbb{P}_p(\psidegen=0) \leq \mathbb{P}_p\left(|\Tdeg-\mathbb{E \Tdeg}|\geq \mathbb{E}\Tdeg /2\right) \leq \frac{\mathbb{V}_p\Tdeg}{\left(\mathbb{E}_p \Tdeg \right)^2/2} \leq\frac{\eta}{2}.
\end{align*}
\end{enumerate}
\end{proof}

\subsection{Proof of Proposition \ref{degenerate_case_proposition}: Lower bound}
\begin{lemma}\label{lower_bound_degen}
Set $$\phi(x) = \cphi \frac{x_1}{\sqrt{n}\homega^{d+1}}$$
where $\cphi$ is a small constant, and $x_1$ denotes the first coordinate of any vector $x \in \Omega$. Set moreover for any $\epsilon \in \{\pm 1\}: p_\epsilon^{(n)} = p_0 + \epsilon \phi$ and assume $\epsilon \sim Rad(\frac{1}{2})$. We have:
\begin{enumerate}
    \item For all $\epsilon \in \{\pm 1\}$, $p_\epsilon^{(n)} \in \mathcal{P}(\alpha,L,\cstar)$,
    \item If the constant $\cphi$ is small enough, it holds $d_{TV}(p_0^{\otimes n}, \mathbb{E}_{\epsilon}(p_{\epsilon}^{(n)})) < \frac{1}{\eta}$, where the expectation is taken according to the prior,
    \item For fixed $\epsilon \in \{\pm 1\}$ it holds $\|p_0 - p_\epsilon^{(n)}\|_t = \frac{\cphi}{\sqrt{n} \homega^{d-d/t}}$.
\end{enumerate}
\end{lemma}

\begin{proof}[Proof of Lemma \ref{lower_bound_degen}]
\begin{enumerate}
    \item We have for all $x,y \in \Omega$:
    \begin{align*}
        |p_0(x)+\phi(x) - p_0(y)-\phi(y)| &\leq |p_0(x) - p_0(y)| + \cphi\frac{|x-y|}{\sqrt{n}\homega^{d+1}} \leq \cstar p(x) + L\homega^\alpha \sqrt{d}^\alpha + \frac{\cphi}{\sqrt{n}\homega^d} \leq \cstar''\homega^{-d}
    \end{align*}
    where $\cstar'' < \frac{1}{2}$ for $\cphi$ small enough. Moreover, we have $p_0 + \phi \in H(\alpha, L)$, $\int_{\Omega} \phi = 0$ and $p_0 \pm \phi \geq 0$.
    \item Proceeding like in the proof of the bulk lower bound (Proposition \ref{lower_bound_bulk}) for $N=1$ and $\phi$ replacing $\phi_j$, we obtain from equation \eqref{quantite_chi2}
    \begin{align*}
        d_{TV}(p_0, \mathbb{E}_\epsilon(p_\epsilon^{(n)})) &\leq -1 + \chi^2\left(p_0 || \mathbb{E}_\epsilon(p_\epsilon^{(n)})\right) \leq -1 + \exp\left(\frac{1}{2} n^2 \left(\int_\Omega \frac{\phi^2}{p_0}\right)^2\right) \\
        & \leq -1+\exp(\cphi) < 1-\eta ~~ \text{ for $\cphi$ small enough.}
    \end{align*}
    \item The last assertion can be shown by direct algebra.
\end{enumerate}
\end{proof}

\subsection{Case $\max\limits_\B \hB > \homega$}\label{hb>homega}
Let $x = \arg\max\limits_\Omega p_0$. By definition of the bulk we have $p_0(x) \geq \cA L \hB^\alpha(x) \geq \cA \Clarge^\alpha L \homega^\alpha$. Now, by Assumption \eqref{simplifyingAssumption}, we have for all $y \in \Omega$:
\begin{align*}
    p_0(y) \geq (1-\cstar)p_0(x) - L\homega^\alpha \sqrt{d^\alpha} \geq \left(1-\cstar - \frac{\sqrt{d}^\alpha}{\cA\Clarge^\alpha}\right)p_0(x) \geq \frac{1}{2} \homega^{-d}
\end{align*}
for $\Clarge$ large enough.
Moreover, for $y \in \Omega$ such that $p_0(y) = \homega^{-d}$ we have 
\begin{align*}
    p_0(x) \leq (1+\cstar) p_0(y) + L\homega^\alpha \sqrt{d}^\alpha \leq (1+\cstar) p_0(y) + \frac{\sqrt{d}^\alpha}{\cA\Clarge^\alpha}p_0(x)
    \Longrightarrow ~ p_0(x) \leq \frac{3}{2} \homega^{-d}
\end{align*}
for $\Clarge$ large enough. Proposition \ref{degenerate_case_proposition} yields the result.

\subsection{Case $Ln \homega^{\alpha+d}\leq \csmall$ and $\alpha \leq 1$}\label{degen_alpha_leq_1}

If $\alpha \leq 1$ then let $x \in \Omega$ such that $p_0(x) = \homega^{-d}$. By the definition of the Hölder class, for all $y \in \Omega$:
\begin{align*}
    |p(x) - p(y)| \leq L\|x-y\|^\alpha \leq L\homega^\alpha \sqrt{d}^\alpha \leq \frac{\csmall\sqrt{d}^\alpha}{n\homega^d}
\end{align*}
so that for all $y \in \Omega, p(y) \in [h^{-d} \pm \frac{\csmall}{n} \sqrt{d}^\alpha \homega^{-d}]$. Thus, the $L_t$ diameter of the class is at most:
\begin{align*}
    \left(\int_{\Omega} \big(2\frac{\csmall}{n} \sqrt{d}^\alpha \homega^{-d}\big)^t dx\right)^{1/t} = 2 \frac{\csmall}{n} \sqrt{d}^\alpha \frac{1}{h^{d-d/t}}.
\end{align*}

Now, the total variation between any two probability distributions $p,q \in \mathcal{P}(\alpha,L)$ writes 
\begin{align*}
    d_{TV}(p,q) = \frac{1}{2} \int_\Omega |p-q| \leq \frac{\csmall}{n}\sqrt{d},
\end{align*}
so that $d_{TV}(p^{\otimes n}, q^{\otimes n}) \leq  \csmall \sqrt{d}^\alpha < 1-\eta$ for $\csmall$ sufficiently small.

\subsection{Case $Ln\, \homega^{\alpha+d} \leq \csmall$ and $\alpha>1$}\label{degen_alpha>1}
By Assumption \eqref{simplifyingAssumption} any $p \in \mathcal{P}(\alpha,L,\cstar)$ takes only values in $[1 \pm \frac{1}{2}]\homega^{-d}$. Indeed, for $x\in\Omega$ such that $p(x) = \homega^{-d}$, we have 
\begin{align}
    |p_0(x) - p_0(y)| \leq \cstar \homega^{-d} + L\homega^\alpha\sqrt{d}^\alpha \leq  \left(\cstar+\frac{\csmall \sqrt{d}}{n}\right)\homega^{-d} \leq \frac{1}{2} \homega^{-d}   \nonumber
\end{align}
for $\csmall$ sufficiently small. Proposition \ref{degenerate_case_proposition} yields the result.
}

\section{Homogeneity and rescaling}\label{Extension_unbounded_domains}

For any cubic domain $\Omega \subset \R^d$, introduce
\begin{align}
     \mathcal{P}_{\Omega}(\alpha,L,\cstar) &= \left\{p\text{  density over } \Omega' ~\big|~ p\in H(\alpha,L) \text{ and $p$ satisfies \eqref{simplifyingAssumption} over $\Omega'$} \right\}, \label{def_P_cubic}
\end{align}

For $\lambda>0$, define the rescaling operator:
\begin{equation}\label{rescaling_operator}
    \Phi_{\lambda}: \left\{ \begin{array}{ccc}
        \mathcal{P}_{\lambda \Omega}(\alpha,L,\cstar) & \longrightarrow & \mathcal{P}_{\Omega}(\alpha, L \lambda^{\alpha+d},\cstar)\\
        &&\\
        p & \longmapsto & \lambda^d\, p(\lambda \, \cdot \,) 
    \end{array}\right. 
\end{equation}
where $p(\lambda \, \cdot \,) : x \mapsto p(\lambda x)$ and $\lambda \Omega =\{\lambda x : x \in \Omega\}$. For any cubic domain $\Omega \subset \R^d$, we define $\rho_{\Omega}^*(p_0,\alpha,L,n)$ as the minimax separation radius for the following testing problem over $\Omega$, upon observing $X_1,\dots,X_n$ iid with density $p \in \mathcal{P}_{\Omega}(\alpha,L,\cstar)$ 
\begin{align}\label{testing_pb_lambda}
\begin{split}
    H_0~~~ &: \; p = p_0 \;\;~~~~~~\; \text{ versus }\\
    H_1^{(\Omega)}(\rho) &:\; p \in \mathcal{P}_{\Omega}(\alpha,L', \cstar') ~ \text{ s.t. } \|p-p_0\|_t \geq \rho.
\end{split}
\end{align}

\begin{proposition}\label{Rescaling_proposition}(Rescaling)
Let $\lambda >0$ and let $p_0 \in \mathcal{P}_{\lambda \Omega}(\alpha,\, L,\,\cstar)$. It holds
$$ \rho_{\Omega}^*\left(\Phi_{\lambda}(p_0),\, \alpha,\, L\lambda^{\alpha+d},\, n\right) = \lambda^{d-d/t}\,\rho_{\lambda \Omega}^*\left(p_0, \alpha, L, n \right).$$
\end{proposition}

\begin{proposition}\label{restriction_of_support}(Restriction of support)
Let $p_0 \in \mathcal{P}_{\Omega}(\alpha,L,\cstar)$ and $\Omega' \subset \Omega$ another possibly bounded cubic domain of $\R^d$. Assume that the support of $p_0$ is included in $\Omega'$. Then $$\rho_{\Omega}^*(p_0,n,\alpha,L) \asymp \rho_{\Omega'}^*(p_0,n,\alpha,L).$$
\end{proposition}

\begin{corollary}\label{cor:notion_difficulty}
The quantity $\rho^*_{\R^d}(p_0,\alpha,L,n,t)/\rhor(\alpha,L,n=1,t)$ is invariant by rescaling.
\end{corollary}
Corollary~\ref{cor:notion_difficulty} is a direct consequence of the definition of $\rhor$ in~\eqref{eq:def_rhor} and Proposition~\ref{Rescaling_proposition}.

\begin{proof}[Proof of Proposition \ref{Rescaling_proposition}]
It is direct to prove that $\forall \lambda>0, \; \Phi_\lambda$ is well-defined and bijective. We can also immediately check that $$\forall p,q \in \mathcal{P}_{\lambda \Omega}(\alpha,L,\cstar): \|\Phi_{\lambda}(p) - \Phi_{\lambda}(q)\|_t = \lambda^{d-d/t}\|p-q\|_t.$$ 
Let $\psi_\lambda^*$ be a test such that $$\forall p\in \mathcal{P}_{\lambda \Omega}(\alpha,L,\cstar): \|p-p_0\|_t \geq C \rho_{\lambda \Omega}^*\left(p_0, \alpha, L, n \right) \Longrightarrow \mathbb{P}_{p_0}(\psi_\lambda^* = 1) + \mathbb{P}_{p}(\psi_\lambda^* = 0) \leq \eta,$$ 
for some constant $C$. 
Now, let $\widetilde p \in \mathcal{P}_{\Omega}(\alpha,L\lambda^{\alpha+d},\cstar)$ such that $$\|\widetilde p - \Phi_{\lambda}(p_0)\|_t \geq C \lambda^{d-\frac{d}{t}}\,\rho_{\lambda \Omega}^*\left(p_0, \alpha, L, n \right).$$ 
It then follows that $\|\Phi_{\lambda^{-1}}(\widetilde p) - p_0\|_t \geq C \rho_{\lambda \Omega}^*\left(p_0, \alpha, L, n \right)$ hence $\mathbb{P}_{p_0}(\psi_\lambda^* = 1) + \mathbb{P}_{\Phi_{\lambda^{-1}}(\widetilde p)}(\psi_\lambda^* = 0) \leq \eta$ i.e $\mathbb{P}_{\Phi_\lambda(p_0)}(\psi^* = 1) + \mathbb{P}_{\widetilde p}(\psi^* = 0) \leq \eta$ where $\psi^*(x_1, \dots, x_n) = \psi_\lambda^*(\lambda x_1, \dots, \lambda x_n)$. 
Therefore, $C \lambda^{d-d/t} \rho_{\lambda \Omega}^*\left(p_0, \alpha, L, n \right) \geq \rho_{\Omega}^*\left(\Phi(p_0), \alpha, L \lambda^{\alpha+d}, n \right)$ and the converse bound can be proved by symmetry using $\Phi_{\lambda^{-1}}$.
\end{proof}

\begin{proof}[Proof of Proposition \ref{restriction_of_support}]
Clearly $\rho^*_{\Omega}(p_0) \geq \rho^*_{\Omega'}(p_0)$. For the converse bound, we define $\psiout = \mathbb{1}\left\{\bigvee\limits_{i=1}^n (X_i \notin \Omega')\right\}$ the test rejecting $H_0$ whenever one of the observations $X_i$ belongs to $\Omega \setminus \Omega'$. Lemma \ref{control_norm_t_Rd} shows that, for $\cout$ and $\Cout$ two large enough constants, if $p \in \mathcal{P}_{\Omega}(\alpha,L,\cstar)$ is such that $\int_{\Omega \setminus \Omega'}p^t \geq \Cout \rhor^t$  then $\int_{\Omega\setminus \Omega'} p \geq \frac{\cout}{n}$, so that $\mathbb{P}_p(\psiout=1) > 1-\eta/2$. 
Now, assume  $\|p-p_0\|_t \geq C\rho^*_{\Omega'}(p_0)$ over $\Omega$, for $C$ a large enough constant, and let $\psi^*$ be an optimal test over $\Omega'$, i.e. such that $\mathbb{P}_{p_{0,\Omega'}}(\psi^*=1) + \mathbb{P}_{p_{\Omega'}}(\psi^*=0) \leq \eta$ whenever $\|p-p_0\|_t \geq C' \rho_{\Omega'}^*(p_0)$. Then if $\int_{\Omega \setminus \Omega'}p^t \geq \Cout \rhor^t$, $\mathbb{P}_{p_{\Omega'}}(\psiout\lor\psi^*=0)\leq \eta/2$. Otherwise, $\int_{\Omega'}|p-p_0|^t \geq C'\rho^*_\Omega(p_0) - \Cout\rhor \geq \frac{C'}{2}\rho^*_\Omega(p_0)$ so that $\mathbb{P}_{p_{\Omega'}}(\psiout\lor\psi^*=0)\leq \eta/2$ for $C'$ large enough. Moreover, under $H_0$ , we clearly have $\mathbb{P}_{p_0}(\psiout\lor\psi^*=1)\leq \eta/2$. Hence the result.
\end{proof}

\section{Proofs of examples}\label{Proof_examples}

\subsection{Uniform distribution}
\boundedcase{We set $p_0 = \lambda^{-d}$ over $\Omega := [0,\lambda]^d$. By the definition of $\uI$ \eqref{def_uI}, we can immediately check that $\uI < \lambda^{-d}$ if and only if $\TL \leq \big(\cI |\Omega|\big)^{-\alpha - d}$. Moreover, by the definition of $\uA$ \eqref{def_uA}, and using $\cI^{\alpha+d} = \cA^{4\alpha+d}$, we can also immediately check that $\uI < \lambda^{-d} \Longleftrightarrow \uA < \lambda^{-d} \TL \leq \big(\cI |\Omega|\big)^{-\alpha - d}$. Therefore, only exactly one of the two terms $\rhob$ or $\rhot$ is non-zero. The bound \eqref{rate_uniforme} can be immediately obtained by simplifying the expression of the minimax separation radius \eqref{expression_rho_star}.}
\unboundedcase{See \cite{ingster2012nonparametric} for $\lambda=1$ and use Proposition \ref{Rescaling_proposition} for arbitrary $\lambda>0$.}

\subsection{Arbitrary $p_0$ over $\Omega = [-1,1]^d$ with $L=1$}

First, note that by equation \eqref{majorant_p_Rd}, $p_0$ is upper bounded by a constant denoted by $C_{\max}$ since $L=1$. 
For any small constant $c$, there exists a fixed  constant $\delta>0$ such that for all $p_0$ with support over $[-1,1]^d$, the set $\{p_0 \geq c\}$ has Lebesgue measure at least $\delta$. Fix such a $c$.
Now, there exists a constant $n_0$ such that for all $n \geq n_0$, for all $p_0$, $\uI(p_0) \leq c$. We then have 
\begin{align*}
    C_{\max}^r 2^d \geq \int_{\B(\uI)} p_0^r \geq c^r \delta ~~ \text{ which is a constant.}
\end{align*}
Therefore, $\rhob^t \asymp n^{-\frac{2\alpha t}{4\alpha+d}}$.\\

As for the tail, we have by the Cauchy-Schwarz inequality and Lemma \ref{int_p0squared_uA_small}
\begin{align*}
    \left(\int_{\T(\uA)}p_0\right)^2 \leq \frac{\left(\int_{\T(\uA)}p_0\right)^2}{\big|\T(\uA) \cap [0,1]^d\big|} \leq \int_{\T(\uA)} p_0^2 \leq \Cmom \TL^\frac{1}{\alpha+d} \left(\int_{\T(\uA)}p_0\right)^\frac{d}{\alpha+d}
\end{align*}
hence 
\begin{align*}
    p_0[\T]\leq \Cmom^\frac{\alpha+d}{2\alpha+d}\, n^{-\frac{2\alpha}{2\alpha+d}}.
\end{align*}
We can now immediately check that $\rhob \gg \rhot$ and $\rhob \gg \rhor$ as $n \to \infty$. Since $\rhob$ is independent of $p_0$, the result is proven.

\subsection{Spiky null}

Set $\widetilde{p_0}(x) = \frac{f}{\|f\|_1}$ over $\R^d$\boundedcase{ and $\widetilde{\widetilde{p_0}}(x)$ the restriction of $\widetilde{p_0}(x)$ over $[\pm \frac{1}{2}]^d$}. Since $\widetilde{p_0}$ takes nonzero values only over $[\pm \frac{1}{2}]^d$ we have \boundedcase{by Proposition \ref{restriction_of_support} $\rho^*(\widetilde{p_0},\alpha,1,n) \asymp \rho^*(\widetilde{\widetilde{p_0}},\alpha,1,n) $ and} $\rho^*(\boundedcase{\widetilde{}}\unboundedcase{\widetilde{p_0}},\alpha,1,n)\asymp n^{-\frac{2\alpha}{4\alpha+d}}$ by the preceding case. Now, by homogeneity (see Proposition \ref{Rescaling_proposition}), we have $\rho^*(p_0, \alpha,L, n) = L^\frac{d(t-1)}{t(\alpha+d)} \rho^*(\widetilde{p_0}, \alpha,1,n)$, which yields the result. 

\subsection{Gaussian null}

Note that
\begin{equation}\label{equivalent_normale}
    \int_{\|x\|>b} p_0(x)dx = e^{-\frac{b^2}{2\sigma^2}(1+o(1))} ~~ \text{ when $b \to +\infty$.}
\end{equation}
Therefore, noting $\bI$ the unique value such that if $\|x\|=\bI$, then $p_0(x) = \uI$, we have by the definition of $\uI$ that $\bI = \sigma^2 \frac{4\alpha}{2\alpha+d} \log(n) (1+o(1))$ when $n \to +\infty$. By Lemma \ref{htail_all_equal} and using \eqref{equivalent_normale}, it holds $\int_{\T(\uA)}p_0 \asymp \int_{\widetilde{\T}(\uI)}p_0 = \int_{\T(\uI)}p_0 = n^{-\frac{2\alpha}{2\alpha+d}(1+o(1))} \gg \frac{1}{n}$, so that the tail rate writes 
$$ \rhot \asymp L^{\frac{d(t-1)}{t(\alpha+d)}} n^{-\frac{2\alpha}{2\alpha + d}(1+o(1))} \gg \rhor.$$
Now, by direct calculation, $\rhob \asymp \frac{L^\frac{d}{4\alpha+d}}{n^\frac{2\alpha}{4\alpha+d}} (\sigma^d)^{\frac{(4-3t)\alpha+d}{t(4\alpha+d)}}$ and we can immediately check that it is the dominant term.

\subsection{Pareto null}

Fix $d=t=1$ and $\alpha \leq 1$. We let $\qI>x_1$ denote the unique value such that $p_0(\qI)=\uI$. By the definition of $\uI$ and using simple algebra we get $\qI \asymp \TL^{-\frac{1}{3\beta + \alpha +1}}$. Moreover, we have by Lemma \ref{htail_all_equal} that $\int_{\T(\uA)}p_0 \asymp \int_{\widetilde{\T}(\uI)}p_0 = \int_{\T(\uI)}p_0 = \TL^\frac{\beta}{3\beta+\alpha+1}$ so that (by recalling $t=1$): $\rhot \asymp \int_{\T(\uA)}p_0 \asymp\TL^\frac{\beta}{3\beta+\alpha+1} \gg \rhor$. Now, we can easily get $\rhob \asymp \TL^\frac{1}{4\alpha+1} \ll \rhot$ which ends the proof.

\end{document}